\documentclass[11pt, dvipsnames]{amsart}
\usepackage{xcolor}
\definecolor{myblue}{HTML}{1E3A5F} \definecolor{myred}{HTML}{FF0000} \usepackage[
  colorlinks,linkcolor=myblue, citecolor=myblue, urlcolor=myblue
]{hyperref} 
\usepackage[utf8]{inputenc}
\usepackage{graphicx}
\usepackage{todonotes}
\usepackage{a4wide}
\usepackage{color}
\usepackage{xcolor}
\usepackage{amsmath,amssymb}
\usepackage[nameinlink]{cleveref}
\crefname{equation}{}{}
\Crefname{equation}{Equation}{Equations}
\usepackage{xcolor}
\usepackage{float}
\usepackage{soul}
\usepackage{lipsum} 
\usepackage{comment}

\usepackage{bm}
\usepackage[foot]{amsaddr}
\usepackage{float}
\usepackage{amsthm}
\newtheorem{theorem}{Theorem}[section]
\newtheorem{lemma}[theorem]{Lemma}
\newtheorem{example}{Example}[section]
\newtheorem{assumption}{Assumption}
\Crefname{assumption}{Assumption}{Assumptions}   

\newtheorem{remark}{Remark}
\newtheorem{corollary}{Corollary}
\newtheorem{definition}{Definition}
\usepackage{mathtools}
\usepackage{etoolbox}
\usepackage{algorithm}
\usepackage{algpseudocode}
\usepackage{fancybox}
\usepackage{enumitem}

\DeclareMathOperator*{\argmin}{arg\,min}

\DeclareFontFamily{U}{matha}{\hyphenchar\font45}
\DeclareFontShape{U}{matha}{m}{n}{
<-6> matha5 <6-7> matha6 <7-8> matha7
<8-9> matha8 <9-10> matha9
<10-12> matha10 <12-> matha12
}{}
\DeclareSymbolFont{matha}{U}{matha}{m}{n}

\DeclareFontFamily{U}{mathx}{\hyphenchar\font45}
\DeclareFontShape{U}{mathx}{m}{n}{
<-6> mathx5 <6-7> mathx6 <7-8> mathx7
<8-9> mathx8 <9-10> mathx9
<10-12> mathx10 <12-> mathx12
}{}
\DeclareSymbolFont{mathx}{U}{mathx}{m}{n}

\DeclareMathDelimiter{\vvvert} {0}{matha}{"7E}{mathx}{"17}
\DeclarePairedDelimiterX{\normiii}[1]
{\vvvert}
{\vvvert}
{\ifblank{#1}{\:\cdot\:}{#1}}

\newcommand{\ramy}[1]{\textcolor{black}{#1}}

\newcommand{\R}{\mathbb{R}}
\newcommand{\fa}{\text{ for all }}

\newcommand{\dd}{\operatorname{d}\!}

\newcommand{\Th}{\mathcal{T}_h}

\newcommand{\conv}{\operatorname{conv}}

\title{A priori error analysis 
of the proximal Galerkin method}
\date{\today}

\author{Brendan Keith$^1$}
\address{$^1$Division of Applied Mathematics, Brown University, Providence, RI 02912}
\email{brendan\_keith@brown.edu,   
 rami\_masri@brown.edu}
\author{Rami Masri$^1$} 
\author{Marius Zeinhofer$^2$}
\address{$^2$Seminar for Applied Mathematics, ETH Z\"urich, Switzerland} 
\email{marius.zeinhofer@math.ethz.ch}

\allowdisplaybreaks
\usepackage{lineno}
\begin{document}

\begin{abstract}
    The proximal Galerkin (PG) method is a finite element method for solving variational problems with inequality constraints.
    It has several advantages, including constraint-preserving approximations and mesh independence. This paper presents the first abstract \textit{a priori} error analysis of the PG method, providing a general framework to establish convergence and error estimates. As applications of the framework, we demonstrate optimal convergence rates \ramy{with respect to the assumed regularity} for both the obstacle and Signorini problems using various finite element subspaces.
    
\vspace{1em}
 \smallskip
  \noindent \textit{Key words}. 
 Proximal Galerkin, finite element method, a priori error analysis, pointwise inequality constraint, obstacle problem, Signorini problem, Bregman proximal point.  

  \smallskip 
   
  \noindent \textit{MSC codes.} 35J86, 35R35, 49J40, 65K15, 65N30. 
\end{abstract}
\maketitle

\section{Introduction}

The proximal Galerkin (PG) method \cite{keith2023proximal} plays a dual role, acting both as an algorithm and a discretization scheme for variational problems with pointwise inequality constraints.
It is an algorithm in the sense that it comprises a sequence of operations (i.e., subproblems to be solved) leading to an approximate solution of a variational problem.
It is a discretization scheme as it yields approximations that depend explicitly on chosen finite-dimensional subspaces, thereby providing a broad selection of discretization choices for the target solution.

This paper presents general \textit{a priori} error analysis of the PG method, which has demonstrated competitive efficacy across a range of problems in applied mathematics, including classical obstacle, contact, and elastoplasticity problems \cite{dokken2025latent}.
Prior analyses have focused on specific discretization choices for particular problems or examined the convergence properties of the PG subproblems only after linearization \cite{keith2023proximal,fu2024locally}.
Instead, the present work provides a foundational advancement, dispensing with the analysis of linearized subproblems to arrive at a general analytical framework for PG methods applied to quadratic optimization problems in Sobolev Hilbert spaces with pointwise inequality constraints.

To illustrate the flexibility of the proposed framework, we focus on two canonical applications: the obstacle and Signorini problems. Established approaches for these problems include the quadratic penalty method \cite{babuvska1973finite} and ``first-discretize-then-optimize" strategies such as the primal-dual active set \cite{hintermuller2002primal} and augmented Lagrangian methods \cite{glowinski1989augmented}. Each of these techniques faces its own distinct challenges: penalty methods require scaling the penalty parameter inversely with the mesh size to maintain accuracy \cite{scholz1984numerical}, while active set and augmented Lagrangian solvers can exhibit mesh-dependent behavior, which can be mitigated via multigrid techniques \cite{graser2009multigrid,bueler2024full}. There is a vast literature on the error analysis of the aforementioned approaches with various spatial discretization choices.
However, we note that the optimization error is often excluded from the analysis when ``first-discretize-then-optimize" approaches are employed. \ramy{In particular, this type of analysis usually assumes that the discrete variational inequality is solved exactly.}
For an overview, we refer the interested reader to the textbooks \cite[Chapter 5]{bartels2015numerical} and \cite[Chapters 5 and 6]{Chouly2023} for the obstacle and Signorini problems, respectively.

In this work, we provide general convergence guarantees for PG with respect to both the number of iterations and the mesh size, leading to the first result establishing the method's mesh-independent iteration complexity. 
In particular, we emphasize that, unlike penalty methods, PG requires no parameters to be sent to singular limits and, unlike the primal-dual active set method \cite{papadopoulos2026mesh}, PG is asymptotically mesh-independent on obstacle and Signorini-type problems.
We refer the interested reader to \cite{dokken2025latent,papadopoulos2024hierarchical} for complementary numerical experiments comparing PG to primal-dual active set and other popular alternatives.
\subsection{Outline}
The remainder of the introduction establishes the basic notation and problem setup, defines the PG method, summarizes the main results, and illustrates example problems.
\Cref{sec:Preliminaries} defines key concepts appearing in the analytical framework, including Legendre functions and Bregman divergence.
\Cref{sec:main_results} rigorously presents the main theoretical results of the paper, including checkable conditions for existence and uniqueness of solutions to the PG subproblems, a rigorous guarantee that the objective function value will decrease monotonically with each iteration, best approximation properties for the PG solution variables, convergence rates, and asymptotic mesh-independence.
\Cref{sec:error_rates,sec:error_rates_sig} are devoted to applications of the theory to the obstacle and Signorini problems, respectively, leading to optimal error convergence rates \ramy{with respect to the assumed regularity} in each case.
Finally, the paper concludes with \Cref{sec:Conclusion}, where we summarize our findings.

\subsection{Notation} Throughout the article, we let $\Omega \subset \R^n$ ($n = 1,2,3$) be an open bounded \ramy{polygonal/polyhedral} domain.   For a given \ramy{Hilbert} space $V$, we denote by $V'$ its topological dual space with duality pairing $\langle \cdot, \cdot\rangle$.
In particular, a member $F \in V'$ is a continuous linear functional mapping $V$ into $\R$,  $F(v) = \langle F, v \rangle \in \R$. The norm $\|\cdot\|_{V'}$ denotes the usual operator norm.  

We use the standard notation for the Sobolev Hilbert spaces $H^m(\Omega)$ and their vector-valued counterparts $H^m(\Omega; \R^n)$. The space $H^{1/2}(\partial \Omega)$ denotes the canonical trace space of $H^1(\Omega)$ functions onto the boundary $\partial \Omega$ with the quotient norm  $\|\cdot\|_{H^{1/2} (\partial \Omega)}$:  
$$\| \hat v \|_{ H^{1/2}(\partial \Omega)} = \inf_{ \substack{v \in H^1(\Omega) \\ \mathrm{tr} \, v = \hat v} } \|v\|_{H^1(\Omega)}.$$ Here, $\mathrm{tr}$ denotes the trace operator. When the setting is unambiguous, we write $v \vert_{\partial \Omega}$ instead of $\mathrm{tr}(v)$.  We also require the definition of  the Lions--Magenes space on measurable $\Gamma \subset \partial \Omega$ \cite{tartar2007introduction}:  
\begin{equation}
\widetilde{H}^{1/2}(\Gamma) = H^{1/2}_{00}(\Gamma) 
 = \{w \in H^{1/2}(\Gamma) \mid \tilde w \in H^{1/2}(\partial \Omega) \}, \label{eq:LM_space}
\end{equation}
where $\tilde w$ is the extension by zero of $w$ outside of $\Gamma$; i.e., $\tilde w = 0 $ on $\partial \Omega \backslash \overline{\Gamma}$ and $\tilde w = w$ on $\Gamma$.
We note that this space is normed by $\|w\|_{\widetilde{H}^{1/2}(\Gamma)} := \|\tilde w\|_{H^{1/2}(\partial \Omega)}$ and that ${H}^{-1/2}(\Gamma) = (\widetilde{H}^{1/2}(\Gamma))^\prime$.
For non-integer $s$, the notation $H^{s}(\Omega)$ denotes the Sobolev--Slobodeckij spaces \cite[Chapter 2]{ern2017finite}.  We use the notation $(\cdot, \cdot)_{\omega}$ to denote the $L^2(\omega)$-inner product over measurable $\omega \subset \overline{\Omega}$. If $\omega = \Omega$, we drop the subscript and denote the $L^2$-inner product over $\Omega$ by $(\cdot,\cdot)$.

The essential domain of a proper function $f\colon\R^n \rightarrow \R \cup \{+\infty\}$ is given by $\operatorname{dom} f : = \{ x \in \R^n : f(x) < \infty\}$. The Fr\'echet derivative of a mapping $F$ between normed vector spaces $X$ and $Y$ at a point $x$ is denoted by $F'(x)$.
For a linear continuous operator $B \in \mathcal{L}(U,V)$ where $U,V$ are normed vector spaces, the topological transpose (adjoint) operator $B' \in \mathcal{L}(V',U')$ is defined as 
\begin{equation}
\label{eq:DualOperator}
\langle B' v' , u \rangle = \langle v' , B u \rangle  \text{ for all } u \in U, v'  \in V'. 
\end{equation}
  
We consider a conforming affine shape regular simplicial mesh $\mathcal{T}_h$ of $\Omega$ with mesh size $h = \max_{T \in \mathcal{T}_h} h_T$ where $h_T = \operatorname{diam}(T)$.  \ramy{Every element $T \in \mathcal{T}_h$ is understood to be closed, so that $\overline T = T$.}
\ramy{ For any $h$ and $T \in \mathcal{T}_h$, 
\begin{equation}
\label{eq:shape_reg}
\frac{h_T}{\rho_T } \leq \sigma,
\end{equation} 
where $\rho_T$ is the radius of the ball inscribed in $T$ and $\sigma$ is the shape regularity constant.}
Define 
\begin{equation}
\mathbb{P}_q(\mathcal{T}_h) = \{v \in L^\infty(\Omega) \mid v \vert_T \in \mathbb{P}_q(T) \quad ~\fa T \in \mathcal{T}_h \}, 
\end{equation}
where $\mathbb{P}_q(T)$ denotes the space of polynomials of total degree less than or equal to $q$ on $T$.
Denote by $\mathcal{N}_h$ \ramy{and $\mathcal{N}_h^\partial$} the set of \ramy{interior (resp.~boundary)} vertices (nodes) in $\mathcal{T}_h$ and by \ramy{$\{\varphi_z\}_{z \in \mathcal{N}_h \cup \mathcal{N}_h^\partial}$} the associated nodal basis functions of polynomial degree 1. The collection of $n+1$ vertices of an element $T \in \mathcal{T}_h$ is denoted by $\mathcal{V}_T$. 

For constants $a$ and $b$, we use the standard notation $a \lesssim b$ whenever there exists a constant $c$ that depends on neither the mesh size $h$ nor on the proximal point parameters such that $a \leq cb$. 
\subsection{General setup}\label{sec:general_setup}
The PG method is versatile and has been successfully applied to a diverse set of variational problems with inequality constraints \cite{dokken2025latent}.
In this work, we consider constrained optimization problems of the following form
\begin{equation}
\min_{v \in K} E(v), \label{eq:general_problem}
\end{equation}
where $E\colon V \rightarrow \R$ is an energy function and $K$ is a closed, convex, and non-empty set taking the general form
\begin{equation}
K = \{v \in V \mid  B v(x) \in C(x)  \text{ for almost every } x \in \Omega_d \subset \overline \Omega \}.  \label{eq:constraint_set_general}
\end{equation}
Here, $V$ is a given (affine) Hilbert space, $\Omega_d$ is a Hausdorff-measurable set with Hausdorff dimension $d \leq n$ and Hausdorff measure $\dd \mathcal{H}_d$, $B\colon V \rightarrow Q$ is a surjective bounded linear map, whose image $Q = \operatorname{im} B$ is continuously and densely embedded in $L^2(\Omega_d;\mathbb{R}^m)$, and $C(x) \subset \R^m$, which may vary with $x$, is a closed convex set with a nonempty interior. \ramy{ For our purposes, the Hausdorff measure coincides with the induced surface measure on $\Omega_d$.  }

 In the definition of the feasible set $K$, the map $B$ can be understood to define observables $o \in Q$  that are restricted pointwise a.e.\ to $C(x)$.
 For example, in obstacle problems, where $V\subset H^1(\Omega)$ and functions $v \in K$ satisfy $v \geq \phi$ a.e.\ in $\Omega$, we have $\Omega_d = \Omega$, $B = \mathrm{id}$ (the identity operator), and $C(x) = [\phi(x), \infty)$ for a given obstacle $\phi \in L^{\infty}(\Omega)$.
 In this case, the observables are simply the unknown functions $v \in V$.
 For contact problems, $\Omega_d = \Gamma$ is a Hausdorff-measurable subset of $\partial \Omega$, and the observables are the normal traces of the functions $v \in V$ restricted to $\Gamma$; namely, $o = v|_{\Gamma} \cdot n$.
 Refer to \cite{dokken2025latent} as well as \Cref{example:obstacle,example:signorini}, given below, for more details.
 For notational convenience later on, we introduce specific notation for the set of constrained observables:
 \begin{equation}
    \mathcal{O}
    =
    \{ o \in L^2(\Omega_d;\mathbb{R}^m) \mid o(x) \in C(x) \text{ for almost every } x \in \Omega_d \subset \overline \Omega \}
 \label{eq:ConstrainedObservables}
 \end{equation}

Throughout this work, we consider only quadratic energies,
\begin{equation}
E(u) = \frac{1}{2} a(u,u) - F(u),   
\end{equation}
where $a$ is a symmetric, continuous, and coercive bilinear form over the space $V$, satisfying
\begin{equation}\label{eq:coercivity_a}
    a(u,u) \geq \nu \|u\|_V^2 , \quad a(u,v) \leq M \|u\|_V \|v\|_V ~\fa u , v \in V,
\end{equation}
for positive constants $\nu$ and $M$, and $F$ is a bounded linear functional on $V$.
\ramy{ We will also frequently use the Fr\'echet  derivative of $E$, $E'(u) \in V'$: 
\begin{equation}
\langle E'(u), v \rangle = a(u,v) - F(v) ~\fa u, v \in V. \label{eq:der_enerygy}
\end{equation}
}
 The convexity of $K$ implies that the model problem \cref{eq:general_problem} is equivalent to the following variational inequality \cite[Theorem 6.1-2]{ciarlet2013linear}: find $u^* \in K$, such that 
\begin{align} \label{eq:VI}
    a(u^*, v-u^*) \geq F (v-u^*) ~\fa v \in K.
\end{align}
Owing to the coercivity of $a$, \cref{eq:VI} admits a unique solution \cite[Theorem 5.6]{brezis2010functional}.
Moreover, $E(u^\ast) \leq E(v)$ for all $v \in K$.

\subsection{The proximal Galerkin method} \label{sec:pG_first}We present the method here and refer to \Cref{sec:main_results} for more details on its derivation. 
 Consider two conforming discrete subspaces $V_h \subset V$ and $ W_h \subset W := L^{\infty}(\Omega_d;\mathbb{R}^m)$.
The PG method, given in \Cref{alg:main_alg_discrete}, consists of iteratively solving for primal solutions $u_h^k \in V_h$ and latent solutions $\psi_h^k \in W_h$. 
For the derivation of \Cref{alg:main_alg_discrete} and precise definition of $ \mathcal{R}^*$, we refer to \Cref{sec:Legendre} and \Cref{sec:derivation}, respectively.
Note that some form of Newton's method is usually used to solve each subproblem in practice; see, e.g., the implementations in \cite{zenodo:proximalgalerkin}.
\begin{algorithm}[htb]
\caption{The Proximal Galerkin Method} 
\begin{algorithmic}[1]\label{alg:main_alg_discrete}
    \State \textbf{input:} Initial latent solution guess $\psi_h^0  \in W_h$, an \ramy{unsummable} sequence of positive proximity parameters $\{\alpha_k\}$, and a functional $\mathcal{R}^*$ with $\nabla \mathcal{R}^*: W \rightarrow \mathcal{O}$.
    \State Initialize \(k = 1\). 
    \State \textbf{repeat}
    \State \quad Find $u_h^{k} \in  V_{h}$
and $\psi_h^{k} \in W_h$ such that 
   \begin{subequations} \label{eq:discrete_lvpp}
\begin{alignat}{2}  
\alpha_{k} \, a(u^{k}_h, v_h )  + b(v_h,\psi^{k}_h - \psi_h^{k-1}) & = \alpha_k \, F(v_h)  && ~\fa v_h \in V_h, \label{eq:lvpp_g_0}\\ 
b(u^{k}_h, w_h) - ( \nabla \mathcal{R}^{*} (\psi^{k}_h), w_h)_{\Omega_d}& = 0 && ~\fa  w_h \in W_h.  \label{eq:lvpp_g_1}
\end{alignat}
\end{subequations} 
\State \quad Assign  \(k \gets k + 1\).
    \State \textbf{until} a convergence test is satisfied.
\end{algorithmic}
\end{algorithm}

In \Cref{alg:main_alg_discrete}, $b: V  \times W \rightarrow \R$ is a bilinear form corresponding to the operator $B$ in the feasible set~\cref{eq:constraint_set_general}. Namely,  
\begin{align}
b(v,w) = (B v, w)_{\Omega_d} ~\fa v \in V , \,  w \in W.
\end{align} 
The map $\nabla \mathcal{R}^*$ is the inverse of the Fr\'echet derivative of a suitably chosen Legendre function $\mathcal{R}$; see \Cref{sec:Legendre} and \Cref{example:shannon_entropy} for more details. We refer to \Cref{alg:alg_obstacle,alg:sign}, given in the sections below, for applications of this algorithm to the obstacle and Signorini problems, respectively, with particular choices of $\mathcal{R}$.

The saddle point system \cref{eq:discrete_lvpp} also produces a non-polynomial approximation of the observable $o^\ast = Bu^\ast$: 
\begin{equation} \label{eq:tilde_u} 
 o_h^k = \nabla \mathcal{R}^*(\psi_h^{k}), \quad k \geq 0. 
\end{equation}
This variable is always constraint-preserving because $\operatorname{im} \nabla \mathcal{R}^* \subset \mathcal{O}$.
Likewise, $o_h^k(x) \in C(x)$ for all $x \in \Omega_d$. 
In addition, \ramy{we note that the variables $\psi_h^k$ do not converge in norm}, however, the dual variables
\begin{equation}
    \lambda^k_h = (\psi_h^{k-1}- \psi_h^{k})/\alpha_{k} , \quad k \geq 1, \label{eq:discrete_Lagrange}
\end{equation}
which are viewed as $\lambda_h^k \in Q'$ via $\langle \lambda_h^k, q\rangle = (\lambda_h^k, q)_{\Omega_d}$
converge to the unique $\lambda^* \in Q'$ satisfying 
\ramy{
\begin{equation}
\label{eq:def_lambda*}
b(v, \lambda^*) = a(u^*, v) - F(v) ~\fa v \in V. 
\end{equation}}
We end this Section with two remarks on the computational aspects of \Cref{alg:main_alg_discrete}.
    \begin{remark}[Equivalent formulation]
    We note that \eqref{eq:discrete_lvpp} can be equivalently reformulated with the variables $(u_h^k, \lambda_h^k) \in V_h \times W_h$ satisfying:
\begin{subequations} \label{eq:discrete_lvpp_lambda}
\begin{alignat}{2}
a(u^{k}_h, v_h )  - b(v_h,\lambda_h^k) & =  F(v_h)  && ~\fa v_h \in V_h, \label{eq:lvpp_lambda_0}\\ 
b(u^{k}_h, w_h) - \left( \nabla \mathcal{R}^{*} (-\alpha_k\lambda_h^k + \psi_h^
{k-1}), w_h \right)_{\Omega_d}& = 0 && ~\fa  w_h \in W_h, \label{eq:lvpp_lambda_1}
\end{alignat}
where $\psi_h^{k}$ is updated via the formula 
\begin{align}
\psi_h^{k} = - \alpha_k \lambda_h^k +\psi_h^{k-1}. 
\end{align}
The formulation \eqref{eq:discrete_lvpp_lambda} is more favorable from a computational perspective since, as we show in \Cref{thm:BAEP,thm:stability}, the variable $\lambda_h^k$ converges to $\lambda^*$ and remains bounded with $k$ and $h$ in the $Q'$ norm. This is in contrast to the variable $\psi_h^k$; see \Cref{thm:stability}. Having mentioned this, we note that the $(u_h^k,\psi_h^k)$ formulation of \eqref{eq:discrete_lvpp} has been successfully implemented in many previous works; see the extensive selection of numerical experiments in \cite{keith2023proximal,dokken2025latent,papadopoulos2024hierarchical,fu2024locally}.
Recognizing that the two formulations are mathematically equivalent, we present the analysis of \eqref{eq:discrete_lvpp} because it appears to yield simpler arguments. 
\end{subequations} 
    \end{remark}

\begin{remark}[Solving the nonlinear subproblems] 
The nonlinear subproblems, \cref{eq:discrete_lvpp} or, equivalently, \cref{eq:discrete_lvpp_lambda}, can be solved with Newton's method initialized with the solution of the previous proximal subproblem. In this case, we note that an aggressive update strategy, $\alpha_k \gg \alpha^{k-1}$, $k = 1,2,\ldots$, may lead to numerical instabilities and, therefore, a large number of linear solves (i.e., Newton steps).
In turn, efficiently solving the nonlinear subproblems requires balancing the order of convergence \cref{eq:MainInequality} of the proximal (outer) loop with the local stability of Newton's method (inner loop).
In previous work \cite{keith2023proximal,dokken2025latent}, taking $\alpha_k = \min(c\alpha^{k-1}, C)$, with $c \geq 1,  C \gg 1$, has been found to be an effective choice leading to superlinear convergence in most model problems.
To increase stability, one can always adaptively decrease (i.e., backtrack) $\alpha_k$ whenever the Newton solver struggles.
Developing optimal \ramy{preconditioners} for the linear systems arising from Newton's method is part of ongoing work and is not the focus of this paper. However, we note that block-diagonal preconditioners have been applied efficiently in \cite{papadopoulos2024hierarchical}. Furthermore, reformulating \cref{eq:discrete_lvpp} as a first-order system of equations and hybridizing it via facet variables yields symmetric positive-definite stiffness matrices at each Newton step \cite{fu2024locally}, enabling the use of many well-known and efficient preconditioning strategies. 
\end{remark}
\subsection{Main results}
\label{sub:MainResults}
\begin{itemize}[leftmargin=*]
    \item We prove that every PG subproblem is well-defined for the general setup of \Cref{sec:general_setup,sec:pG_first}. More precisely, \Cref{thm:existence_uniqueness} establishes the existence and uniqueness of solutions to the discrete nonlinear subproblems \cref{eq:discrete_lvpp} provided certain compatibility conditions between the subspaces $V_h$ and $W_h$ are satisfied. Additionally, we prove important new energy dissipation and stability estimates for \cref{eq:discrete_lvpp}.  
    \item We provide a general framework for the error analysis of \Cref{alg:main_alg_discrete}. This framework shows that the existence and optimality of certain reconstruction and Fortin operators, defined below,  are sufficient to derive error estimates and mesh-independence results; see \Cref{thm:BAEP},  \Cref{cor:err_rate_general}, and \Cref{cor:WAMI}.
    \item We demonstrate applications of this framework to the analysis of obstacle and Signorini problems; see \Cref{sec:error_rates} and \Cref{sec:error_rates_sig}. In particular, we construct and prove error rates for the Fortin and reconstruction operators used in the proposed framework of \Cref{sec:main_results}.
    \item
    Finally, one of the main results of this paper may be summarized by the following estimate. If $V \subset H^1(\Omega)$ and the solution $u^* \in H^{1+s}(\Omega)$ with $E'(u^*) \in H^{r-1}(\Omega) $ for some $ s,r \in (0,1]$, then 
    \begin{align}
        \|u^* - u_h^\ell\|^2_{H^1(\Omega)} + \|\lambda^* - \lambda_h^\ell\|^2_{Q'} \leq \frac{C_{\rm{stab}}}{ \sum_{k=1}^\ell \alpha_k} + C_{\mathrm{reg}} h^{2\cdot\min\{r,s\}}
    \label{eq:MainInequality}
    \end{align}
    for all $\ell \geq 1$ and $h > 0$, where the constants $C_{\rm{stab}}$ and $C_{\rm{reg}}$ are independent of $h$, $\{\alpha_k\}$, and $\ell$. 
    We first state~\cref{eq:MainInequality} in \Cref{cor:err_rate_general}, where we prove it under general assumptions.
    We then prove~\cref{eq:MainInequality} again in \Cref{cor:Obstacle,cor:sig_error} for the obstacle and Signorini problems, respectively, by verifying the general assumptions for specific choices of $V_h$ and $W_h$.
\end{itemize}
\subsection{Example problems} 
We provide three examples illustrating the setup.
The forthcoming sections expand on the first two examples.
The PG method has not yet been applied to the third example, and so further analysis is reserved for follow-up work.
Together, these examples illustrate the need for a theory that comprises arbitrary Hilbert spaces $V$, observation maps $B$, subsets $\Omega_d$, and convex sets $C$.
Numerical experiments and additional examples can be found in \cite{dokken2025latent}.
\begin{example}[Obstacle problem]\label{example:obstacle}
For the obstacle problem, set the space $V = H^1_0(\Omega)$.
One seeks to minimize the Dirichlet energy
\begin{equation}
\label{eq:DirechletEnergy}
    E(u) = \frac{1}{2}  \|\nabla u \|^2_{L^2(\Omega)} - \int_\Omega f u \dd x,
\end{equation}
over the feasible set
\begin{equation}
K = \{v \in H^1_0(\Omega) \mid v \geq \phi \text{ a.e.\ in  }  \Omega \}, 
\end{equation}
where $f \in L^2(\Omega)$ and $\phi \in L^\infty(\Omega)$ is a given obstacle satisfying $\phi \leq 0  $ a.e.\ on $\partial \Omega$. 
The forms $a: H^1(\Omega) \times H^1(\Omega) \rightarrow \R$ and $F\colon H^1(\Omega) \rightarrow \R$ are  given by 
\begin{equation}
   a(u,v) = \int_{\Omega} \nabla u \cdot \nabla v \dd x,  \quad F(v) = \int_{\Omega} fv  \dd x.  \label{eq:from_a_f_obs}
\end{equation} 
Note that $a$ is coercive over $V$ due to Poincar\'e's inequality. 
To view  $K$ in the general form \cref{eq:constraint_set_general}, set $B$ to be the identity operator on $V = Q$, take $C(x) = [\phi(x), \infty)$, and let $\Omega_d = \Omega$. In this case, $\lambda^* = E'(u^*) \in V'$ with $\langle E'(u^*), \cdot \rangle = a(u^*, \cdot) - F(\cdot)$. 
\end{example}

\begin{example}[Signorini problem] \label{example:signorini}
We consider disjoint boundaries $\Gamma_\mathrm{T}$ and $\Gamma_\mathrm{D}$ which are  measurable subsets of $\partial \Omega$ with $\partial \Omega = \overline{\Gamma_\mathrm{D} \cup \Gamma_\mathrm{T}}$. The Signorini problem consists of finding the displacement $u \in V = (H^1_\mathrm{D}(\Omega))^n,$ where $H^1_{\mathrm{D}}(\Omega) := \{   v \in H^1(\Omega) \mid v|_{\Gamma_\mathrm{D}} = 0 \}$ minimizing the strain energy function 
\begin{equation}
 E(u) = \frac{1}{2} \int_{\Omega} \mathsf{C}\, \epsilon(u) : \epsilon(u) \dd x -  \int_{\Omega} f \cdot u  \dd x, 
\end{equation}
over the convex and closed set 
\begin{equation}
    K = \{ u \in V \mid u \cdot n \leq g \text{ a.e.\ on } \Gamma_{\mathrm{T}}\}. 
\end{equation}
Here, $\epsilon(u) = \frac{1}{2} (\nabla u + \nabla u^\top)$ is the linearized strain tensor, $\mathsf{C} \colon \R^{n\times n}_{\mathrm{sym}} \rightarrow \R^{n \times n}_{\mathrm{sym}}$ is a symmetric positive definite material tensor, $f \in L^2(\Omega;\R^n)$,  $n$ is the unit outward normal vector on $\partial \Omega$, and $g \in  L^{\infty}(\Gamma_{\mathrm{T}})$ with $g \geq 0$ is a prescribed gap function.  

We can write the set $K$ in the form \cref{eq:constraint_set_general} by setting  $B u = - u \vert_{\Gamma_{\mathrm{T}}} \cdot n$, $Q = \widetilde{H}^{1/2}(\Gamma_\mathrm{T})$, $\Omega_d = \Gamma_\mathrm{T}$, and $C(x) = [-g(x) , \infty) $. 
The bilinear form $a:H^1(\Omega; \R^n) \times H^1(\Omega; \R^n) \rightarrow \R $ and linear form $F\colon H^1(\Omega; \R^n) \rightarrow \R$ are given by 
\begin{equation}
a(u,v ) = \int_{\Omega} \mathsf{C}\, \epsilon(u) : \epsilon(v) \dd x, \quad F(v) = \int_{\Omega} f \cdot v \dd x. 
\end{equation}
Korn's inequality guarantees the coercivity of the bilinear form $a$ over $V$; see, e.g., \cite[Theorem 42.9]{ern2021finite2}. In this case, $\lambda^* = (\mathsf{C} \epsilon(u^*) n) \cdot n \in {H}^{-1/2}(\Gamma_{\mathrm T}) = Q^\prime$.
\end{example}
 
\begin{example}[Image restoration]\label{example:TVR}
\ramy{The image restoration problem does not directly fit the framework of \Cref{sec:general_setup}. However, we can apply the PG algorithm to solve the Fenchel predual problem, which fits the framework above. In particular,} fix $f \in L^2(\Omega)$ and $\beta > 0$.
It is well-known \cite{kunisch2004total} that the pre-dual of the classical total bounded variation-regularized tracking problem \cite{rudin1992nonlinear},
\[
    \text{minimize }~~
    J(u)
    :=
    |Du|(\Omega) + \frac{\beta}{2}\int_\Omega (u - f)^2 \dd x
    ~\text{ over } u\in BV(\Omega)\cap L^2(\Omega)
    \,,
\]
where $BV(\Omega)$ is the space of functions of bounded variation over $\Omega$ \cite{attouch2014variational,evans2018measure} and
\[
    |Du|(\Omega)
    =
    \sup\bigg\{
        \int_\Omega u \operatorname{div}\phi \dd x \mid
        \phi \in C^1_c(\Omega;\mathbb{R}^n)
        ,~
        |\ramy{\phi(x)}|_{\ell^\infty} \leq 1
        \text{ a.e.\ in } \Omega
    \bigg\}
\]
denotes the $BV(\Omega)$-seminorm, can be expressed as a bilaterally constrained optimization problem.
More specifically, we are interested in finding a unique $u^\ast \in BV(\Omega)\cap L^2(\Omega)$ such that
\(
    J(u^\ast) \leq J(v)
\)
for all $v \in BV(\Omega)\cap L^2(\Omega)$.
This problem is well-posed, and in \cite{kunisch2004total} it is shown that its solution satisfies the following identity:
\[
    u^\ast = f + \beta^{-1} \operatorname{div} p^\ast
    ,
\]
where
\(
    p^\ast \in 
    H_0(\operatorname{div},\Omega)
    =
    \{
        p \in L^2(\Omega;\mathbb{R}^n)
        \mid
        \operatorname{div} p \in L^2(\Omega)
        ,~
        p|_{\partial\Omega}\cdot n = 0
    \}
\)
is the unique minimizer of the energy function
\begin{equation}
\label{eq:DualFunctional}
    E(p) =
    \frac{1}{2}\int_\Omega (\operatorname{div} p)^2 \dd x
    +
    \frac{\gamma}{2}\int_\Omega (\operatorname{proj} p)^2 \dd x
    +
    \beta \int_\Omega f \operatorname{div} p \dd x
\end{equation}
over the convex set
\[
    K
    =
    \big\{ p \in H_0(\operatorname{div},\Omega)
    \mid
    |p_i| \leq 1
    \text{ a.e.\ in } \Omega
    \text{ for each } i = 1,\ldots,n
    \big\}
    .
\]
In \cref{eq:DualFunctional}, $\gamma > 0$ is a fixed parameter and $\operatorname{proj}$ is the orthogonal projection $H_0(\operatorname{div},\Omega) \to \{ p \in H_0(\operatorname{div},\Omega) \mid \operatorname{div} p = 0 \}$.

We can write the set $K$ in the form of \cref{eq:constraint_set_general} by taking $B$ to be the identity on $V = Q = H_0(\operatorname{div},\Omega)$, $\Omega_d = \Omega$, and $C = [-1 , 1]^n$. 
In this case, the bilinear form $a\colon V \times V \rightarrow \R $ and linear form $F\colon V \rightarrow \R$ are given by 
\begin{equation}
    a(p,q) = \int_{\Omega} \operatorname{div} p \, \operatorname{div} q \, \dd x
    +
    \gamma\int_{\Omega} \operatorname{proj} p \, \operatorname{proj} q \, \dd x,
    \quad F(q) = -\beta\int_{\Omega} f \operatorname{div} q \, \dd x. 
\end{equation}
Note that $a$ is coercive over the Hilbert space $V$ due to a Friedrichs' inequality; see, e.g., \cite[Lemma~2.8]{demkowicz2023mathematical}.
\end{example}

\subsection{Closed observation maps: A conjecture}
\label{sub:ClosedMaps}
    Some important problems with feasible sets of the form~\cref{eq:constraint_set_general} do not fit into the general setup described in \Cref{sec:general_setup} because the observation map $B$ does not map onto a dense subset of $L^2(\Omega_d;\R^m)$.
    For instance, consider the classical elastoplastic torsion problem \cite{TWTing_1969,Brezis1971}, which involves minimizing the Dirichlet energy~\cref{eq:DirechletEnergy} over a feasible set with gradient constraints, such as
    \begin{equation}
    \label{eq:GradientConstraint}
        K =
        \{
            v \in H^1_0(\Omega) \mid |\nabla v| \leq 1 \text{ a.e.\ in } \Omega
        \}
        .
    \end{equation}
    Note that this set is recovered from~\cref{eq:constraint_set_general} by setting $C$ to be the closed unit ball in $\mathbb{R}^n$, $V = H^1_0(\Omega)$, $B = \nabla$, and $\Omega_d = \Omega$.
    For further details, see~\cite[Example 6]{dokken2025latent}.
    If $n \geq 2$, then $\operatorname{im} B$ is a closed proper subspace of $L^2(\Omega;\mathbb{R}^n)$.
    
    Informed by numerical experiments in \cite{dokken2025latent,papadopoulos2024hierarchical}, we conjecture that the PG iterates $u_h^k$ can also be shown to converge to the exact solution of~\cref{eq:general_problem} if $B \colon V \to L^2(\Omega_d;\mathbb{R}^m)$ is a closed operator.
    However, unlike the analysis below, we can not treat each subproblem~\cref{eq:discrete_lvpp} as a \textit{singularly-perturbed} nonlinear saddle-point problem in such a setting, and we must also account for the possibly non-trivial kernel of $B^\prime$.
    In turn, we expect different general results with such observation maps, and we do not consider them further in this work.

\section{Preliminaries}\label{sec:Preliminaries}
We briefly recall two key concepts fundamental to the derivation and analysis of PG methods: Legendre functions and Bregman divergences. We refer the interested reader to \cite{rockafellar1967conjugates,burachik1998generalized,bauschke1997legendre,keith2023proximal} for more details.
We then derive the PG method.

\subsection{Legendre functions}\label{sec:Legendre} \Cref{alg:main_alg_discrete} depends on the specific choice of the functional $\mathcal{R}^*$, which satisfies $\nabla \mathcal{R}^*:L^\infty(\Omega_d, \mathbb{R}^m) \rightarrow \mathcal{O}$. The construction of $\mathcal{R}^*$ relies on the concept of a Legendre function \cite{RTRockafellar_1970}.
In this work, a function $L: \R^m \rightarrow \R \cup \{+\infty\}$ is called a Legendre function if it is proper with $\operatorname{int} (\operatorname{dom} L) \neq \emptyset $, strictly convex and differentiable on $\operatorname{int} (\operatorname{dom} L)$ with a singular gradient on the boundary of  $\operatorname{dom} L$.  This subsection aims to briefly show how Legendre functions are utilized to define $\mathcal{R}^*$ as the convex conjugate of a superposition operator  $\mathcal{R}$. 

Consider a Carath\'eodory function $R: \Omega_d \times \R^m  \rightarrow \R \cup \{+ \infty\}$ where $ R(x, \cdot)$ is a Legendre function with $\operatorname{dom}(R(x,\cdot)) = C(x)$ f.a.e.\ $x \in \Omega_d$. 
Let $$\mathcal{R}(w)(x) = R(x,w(x)), \quad x \in \Omega_d, \,  w \in L^2(\Omega_d),$$
be the corresponding superposition operator \ramy{\cite{ambrosetti1995primer}} with 
$$\nabla \mathcal{R}(u)(x) = \partial_u R(x,u(x)).$$

\begin{assumption}[Continuity]\label{assumption:R}
\label{assum:cont} 
We assume that  $\mathcal{R}:\mathcal{O} \rightarrow L^1(\Omega_d)$ is continuous; i.e., if $\{o_n\},\; o \in \mathcal{O}$, and $ \lim_{n\rightarrow \infty} \|o_n - o\|_{L^2(\Omega_d)} \rightarrow 0$,  then $ \lim_{n \rightarrow \infty} \| \mathcal{R}(o_n) - \mathcal{R}(o) \|_{L^1(\Omega_d)} = 0 $.
\end{assumption}
\ramy{Here, for simplicity,  we implicitly assumed that $R(x, u(x)) \lesssim 1 + |u(x)|^2$ for all $u$ with $u(x) \in C(x)$ f.a.e. $x \in \Omega_d$. This is satisfied for all the examples considered here and ensures that $\mathcal{R}$ is well defined over $\mathcal{O}$. More general functions $R(x, \cdot)$ can be incorporated by modifying the definition of $\mathcal{O}$.} 
The convex conjugate of $R(x,\cdot)$ and its associated superposition operator are given by
\begin{equation}\label{eq:convex_conj_S}
    R^*(x,z)
    = \sup_{y \in \mathbb{R}}  \big\{ zy - R(x,y) \big\}, 
    \, \; \;
     \mathcal{R}^*(\psi)(x)
    =
    R^*(x,\psi(x))
    \,.
\end{equation} 
\ramy{ The operator $\mathcal{R}^*: L^\infty(\Omega_d;\mathbb{R}^m)\to L^1(\Omega_d)$.} In the sequel, we tacitly assume a supercoercivity property of the map $R(x,\cdot)$; namely, $R(x, y)/|y| \rightarrow \infty$ as $|y| \rightarrow \infty$ for a.e.\ $x \in \Omega$.
This ensures that $R^*(x,\cdot)$ is well-defined and continuously differentiable over all of $\R^m$ \cite[Proposition 2.16]{bauschke1997legendre}; see also \cite[Corollary 13.3.1]{RTRockafellar_1970}.
Likewise, we can conclude that
\ramy{ $\nabla \mathcal R^*: L^\infty(\Omega_d;\mathbb{R}^m)\to \mathcal{O}$} is continuous on $L^\infty(\Omega_d;\R^m)$.
We now utilize the following relation, first demonstrated for Legendre functions in \cite{rockafellar1967conjugates}:
\begin{equation} \label{eq:conj_grad_inverse}
 \nabla \mathcal{R}^*  = (\nabla \mathcal{R})^{-1}. 
\end{equation} 
In turn, we conclude that $\operatorname{dom}(\nabla \mathcal R) = \operatorname{im}(\nabla \mathcal R^*) \subset \mathcal{O}$, implying $ \nabla \mathcal R^* (\psi)(x) \in \operatorname{int} C(x)$ f.a.e.\ $x \in \Omega_d$. 

We provide three explicit examples in \Cref{example:shannon_entropy,example:fermi_dirac,example:Hellinger} below, and we refer to \cite[Table 1]{dokken2025latent} for more. 

Identity \cref{eq:conj_grad_inverse} also allows us to define a latent representation of every observable $Bu \in \operatorname{dom}(\nabla \mathcal{R})$; namely,
\begin{equation}
\psi = \nabla \mathcal{R} ( Bu) \iff \nabla \mathcal{R}^*(\psi) = Bu.
\label{eq:latent_primal}
\end{equation}
We refer to such functions $\psi\colon \Omega_d \to \mathbb{R}^m$ as latent variables.

In the subsequent sections, we will make use of the following identity, which can be derived by directly expressing the  convex conjugate $R^*(x,z)$ given in \cref{eq:convex_conj_S} by 
\begin{equation}\label{eq:direct_expression_R*}
 R^*(x,z) = z \left( \partial_u R(x, \cdot) \right)^{-1}(z) - R(x, \left( \partial_u R(x, \cdot) \right)^{-1}(z)).
\end{equation}
Likewise, we deduce that
\begin{equation} \label{eq:explicit_R*}
\mathcal R^*(\psi) = \psi Bu - \mathcal{R}(Bu)
\end{equation}
for any $u$ and $\psi$ satisfying  \cref{eq:latent_primal}.
 
\begin{example}[Shannon entropy] \label{example:shannon_entropy}
Consider \Cref{example:obstacle} and define
\[ 
R(x,y) = (y - \phi(x)) \ln (y - \phi(x)) - (y -\phi(x)) ,
\]
if $y \geq \phi(x)$ and $R(x,y) = + \infty$ otherwise.  The corresponding superposition operator is 
\begin{equation*}
\mathcal{R}(u) = (u -\phi) \ln (u -\phi) - (u - \phi). 
\end{equation*}
We deduce that  $\nabla \mathcal{R}(u) = \ln (u - \phi)$ whenever $u \in \operatorname{dom}(\nabla \mathcal R)$ where
\[ 
\operatorname{dom}(\nabla \mathcal{R}) = \{ u \in L^{\infty}(\Omega) \mid \operatorname{ess} \operatorname{inf} (u - \phi) > 0 \}. 
\]
The continuity of $\mathcal{R}$ on $\mathcal{O}$ (\Cref{assum:cont}) follows from \cite[Theorem 4.1]{keith2023proximal}. 
A simple computation shows that $R^*(x,z) = \exp(z) + \phi(x) z$ with 
\begin{equation}
\mathcal{R}^*(\psi) = \exp(\psi) + \phi \psi, \quad \quad \nabla \mathcal{R}^* (\psi) = \exp(\psi) + \phi.   \label{eq:shannon}
\end{equation}
Note that $\nabla \mathcal{R}^*$ is well defined for all of $L^{\infty}(\Omega)$ and $\nabla \mathcal{R}^*(\psi) > \phi$ a.e. in $\Omega$. 
\end{example} 

\begin{example}[Fermi--Dirac binary entropy] \label{example:fermi_dirac} For bilateral constraints,  $\underline{u}(x) \leq Bu(x) \leq \overline{u}(x)$ in $\Omega_d$, we define 
\begin{equation}
R(x,y) = (y - \underline{u}(x)) \ln (y - \underline{u}(x)) + (\overline{u}(x) - y) \ln(\overline{u}(x) - y), 
\end{equation}
if $\underline{u}(x) \leq y \leq \overline{u}(x)$ and $R(x,y) = +\infty$ otherwise.
The corresponding superposition operator $\mathcal{R}$ is continuous on $\mathcal{O}$; i.e., it satisfies \Cref{assum:cont} \cite[Lemma 3.2]{keith2023proximal}.  Here, $\nabla \mathcal{R}(o) = \ln(o- \underline{u})  - \ln (\overline{u} - o)$ with 
\[ 
\operatorname{dom}(\nabla \mathcal{R}) = \{ o \in L^\infty(\Omega_d) \mid \operatorname{ess} \operatorname{inf} (o - \underline{u}) > 0 \text{ and } \operatorname{ess} \operatorname{sup} (o - \overline{u}  ) <  0 \}.
\]
With \cref{eq:conj_grad_inverse}, we derive that
\begin{equation}
\nabla \mathcal{R}^*(\psi) = \frac{\underline{u} + \overline{u} \exp(\psi)}{1+ \exp(\psi)}. 
\end{equation}
Observe that $ \underline{u} < \nabla \mathcal{R}^*(\psi) < \overline{u}$. This entropy functional is suitable for \Cref{example:TVR} with $\underline{u} = -1$ and $\overline{u}= 1$ where the latent variable space is vector-valued, $W = L^{\infty}(\Omega; \R^n)$. 
\end{example}

\begin{example}[Hellinger entropy]
\label{example:Hellinger}
For Euclidean norm constraints $|Bu| \leq \gamma$, cf.\ \cref{eq:GradientConstraint},  one can select the Hellinger entropy 
\[ 
\mathcal{R}(u) = - \sqrt{\gamma^2 - |u|^2},
 \text{ with } \nabla \mathcal{R}^*(\psi) = \frac{\gamma}{\sqrt{1 + |\psi|^2}} \psi. 
\]
In this case, we have that 
$$ 
\nabla \mathcal{R}(o) = \frac{o}{\sqrt{ \gamma^2 - |o|^2}},  \quad \operatorname{dom}(\nabla \mathcal{R}) = \{ o \in L^{\infty}(\Omega_d; \R^m) \mid \operatorname{ess} \operatorname{sup} (|o| - \gamma ) < 0 \}.
$$
This entropy has been used in the PG framework for gradient norm constraints, see \Cref{sub:ClosedMaps} and \cite[Example 6]{dokken2025latent}. 
\end{example}
\subsection{Bregman divergences}\label{sec:Bregman} The Legendre functions defined in \Cref{sec:Legendre} allow one to define the Bregman divergence, which is a key ingredient in the derivation of generalized proximal point methods \cite{burachik1998generalized, keith2023proximal}.  For $u \in \rm{dom}(\mathcal{R})$ and $v \in \mathrm{dom}(\nabla \mathcal{R})$, the Bregman divergence is given by the error in  the first order Taylor expansion of an associated convex function $\mathcal{R}$:
\begin{equation}
    \label{eq:Bregman}
    \mathcal{D}(u,v) = \mathcal{R}(u) - \mathcal{R}(v) - \nabla \mathcal{R}(v)(u - v) . 
\end{equation}
Observe that $\mathcal{D}(u,v) \geq 0 $ and $\mathcal{D}(v,v) = 0$.
We will also use the dual or conjugate divergence  
\begin{equation} \label{eq:dual_divergence}
 \mathcal{D}^*(\eta, \psi) = \mathcal{R}^*(\eta) -\mathcal{R}^*(\psi) - \nabla \mathcal{R}^*(\psi)(\eta - \psi). 
\end{equation}
The Bregman divergence and its conjugate are linked as follows. If $\eta = \nabla \mathcal{R}(v)$ and $\psi = \nabla \mathcal{R}(u)$ for $u, v \in \mathrm{dom}(\nabla \mathcal{R}) $, then 
\begin{equation} \label{eq:linking_D*_D}
 \mathcal{D}^*(\eta, \psi)  = \mathcal{D}(u,v). 
\end{equation} 
The proof can be found in \ramy{ \cite[Section 1.4]{amari2016information}}, see also \cite[Theorem 3.9]{bauschke1997legendre}.
We also recall the three points identity \cite[Lemma 3.1]{chen1993convergence}:
\begin{align}
\mathcal{D}(u,v) - \mathcal{D}(u,w) + \mathcal{D}(v,w) & = (\nabla \mathcal{R}(v) -\nabla \mathcal{R}(w)) ( v- u).   \label{eq:three_point_iden_star}
\end{align}
The same identity holds for $\mathcal{D}^*$ with $\mathcal{R}^*$ replacing $\mathcal{R}$.

\subsection{Deriving the proximal Galerkin method} \label{sec:derivation}
The PG method can be seen as a conforming finite element discretization of a continuous-level algorithm known as the latent variable proximal point algorithm (LVPP) \cite{dokken2025latent}.
However, LVPP is itself just a convenient rewriting of the Bregman proximal point algorithm \cite{chen1993convergence}:
\begin{equation}
\label{eq:generalized_proximal}
    u^k = \argmin_{u \in K} 
    E(u)  
    + \alpha_k^{-1}\int_{\Omega_d}\mathcal D(Bu,Bu^{k-1}) \dd \mathcal{H}_d,
    \qquad
    k = 1,2,\ldots
\end{equation}
The algorithm~\cref{eq:generalized_proximal} leverages the Legendre function $\mathcal{R}$ to adaptively regularize~\cref{eq:general_problem}, resulting in iterates $u^k$ that converge at a controllable speed to the global minimizer $u^\ast$.
The message behind the LVPP reformulation is that the subproblems in~\cref{eq:generalized_proximal} are easy to discretize and solve if the latent variable $\psi$ in~\cref{eq:latent_primal} is incorporated.

Formally, choosing $R(x,\cdot)$ with singular derivatives at $\partial C(x)$ f.a.e.\ $x\in \Omega_d$ (cf.\ \Cref{sec:Legendre}), one expects that $B u^k \in  \operatorname{im} \nabla \mathcal{R}^*$; in particular, $Bu^k(x) \in \mathrm{int} \, C(x) $ f.a.e.\ $x \in  \Omega_d$.
Likewise, the solutions $u^k$ of the regularized subproblems~\cref{eq:generalized_proximal} are characterized by variational equations:
\begin{equation} \label{eq:intro:fooc}
    \text{find } u^k \in K \text{ such that }~ \alpha_k \langle E^\prime(u^k), v \rangle + (\nabla \mathcal{R}(Bu^k),Bv)_{\Omega_d} = (\nabla \mathcal{R}(Bu^{k-1}),Bv)_{\Omega_d}
\end{equation}
for all $v \in V$.
Introducing the latent variables $\psi^k = \nabla \mathcal{R}(Bu^k)$ (i.e., $Bu^k = \nabla \mathcal{R}^*(\psi^k)$ by~\cref{eq:conj_grad_inverse}) to rewrite the resulting equations in saddle-point form yields the LVPP algorithm, and discretizing the resulting saddle-point problems leads to the PG method \cite{dokken2025latent}.
We refer the reader to \cite{keith2023proximal} for a detailed derivation of LVPP for the obstacle problem and to \cite[Section 3]{fu2024locally} for a brief summary.

The LVPP algorithm reads as follows: for some starting point $\psi^0 \in {W}$ and an unsummable sequence of positive parameters $ \{ \alpha_k \}$,  find ($u^k, \psi^k) \in V \times W$ such that 
\begin{subequations} \label{eq:cont_alg}
\begin{alignat}{3}
\label{eq:cont_alg_0}
\alpha_k a(u^k,v) + b(v, \psi^k - \psi^{k-1}) & = \alpha_k F(v) &&  ~\fa v \in V,  \\ 
b(u^k, w) - ( \nabla \mathcal{R}^*(\psi^k), w)_{\Omega_d} & = 0  && ~\fa w \in W,    \label{eq:cont_alg_1}
\end{alignat}
\end{subequations}
and $k = 1,\ldots$   
Note that from \cref{eq:cont_alg_1}, the latent variables $\psi^k$  satisfy a crucial identity: $Bu^k = \nabla \mathcal{R}^*(\psi^k) \in  \mathcal{O}$.
Proving that \cref{eq:cont_alg} is well-posed is generally a challenging task.
To date, it has only been accomplished for the obstacle problem (cf.\ \Cref{example:obstacle}) under suitable regularity assumptions on the problem data \cite{keith2023proximal}. The main difficulty in proving a general well-posedness result for $\eqref{eq:cont_alg}$ is in identifying the correct functional setting, which typically involves a non-reflexive Banach space $W$. For instance, for the obstacle problem (cf.\ \Cref{example:obstacle}), it is known that $W=L^\infty(\Omega)$ is the correct choice under suitable regularity assumptions on the problem data \cite{keith2023proximal,keith2026analysis}. 
Fortunately, in the \textit{a priori} error analysis that follows, we do not rely on the continuous-level subproblems~\cref{eq:cont_alg} in any way.
Instead, we focus solely on the discretized subproblems~\cref{eq:discrete_lvpp}.

\section{Framework for the analysis of the proximal Galerkin method} \label{sec:main_results}
This section presents our main results in a general framework. We first introduce a compatible subspace condition to demonstrate that the PG method is well-defined.
We then show that PG is endowed with a convenient energy decay property, leading to best approximation results and error convergence rates if a so-called reconstruction operator exists. The assumptions of this section are verified for the obstacle and Signorini problems in \Cref{sec:error_rates} and \Cref{sec:error_rates_sig}, respectively. 
\subsection{Compatible subspaces} \label{sec:compatibility}
A critical condition for our analysis is that the finite-\linebreak{}dimensional subspaces $V_h \subset V$ and $W_h \subset W \subset Q^\prime$ satisfy the discrete inf-sup or Ladyzhenskaya--Babu\v{s}ka--Brezzi (LBB) condition
 \begin{equation} 
    \inf_{w \in W_h} \sup_{v \in V_h} \frac{ b(v,w)}{\|v\|_{V} \|w\|_{Q^\prime}} = \beta_h >  \beta_0,
\label{eq:inf_sup}
\end{equation}
where $\beta_0 > 0$ is a mesh-independent positive constant.
This condition is closely related to the continuous inf-sup condition 
\begin{equation}\label{eq:inf_sup_continuous}
 \inf_{w \in W} \sup_{v \in V} \frac{b(v,w)}{\|v\|_V \|w\|_{Q^\prime}} = \beta > 0.
\end{equation}
The density of $W$ in $Q^\prime$ and the closed range theorem can be used to show that \cref{eq:inf_sup_continuous} implies $B^\prime \colon Q^\prime \to V^\prime$ is bounded from below. 
Likewise,~\cref{eq:inf_sup} ensures that $B^\prime$ remains bounded from below (uniformly in $h$) after discretization.

To prove \cref{eq:inf_sup}, it suffices to exhibit a continuous so-called Fortin operator $\Pi_h \colon V \rightarrow V_h$ satisfying $\|\Pi_h v \|_V\lesssim \|v\|_V$ for all $v \in V$ and   
\begin{equation} \label{eq:fortin_map} 
   b(v - \Pi_h v, w_h) = 0
    ~\fa
    v \in V,\; w_h \in W_h;
\end{equation}  
see, e.g., \cite[Lemma 26.9]{ern2021finite2} and \cite[Section~5.4.3]{boffi2013mixed}.
In what follows, we let $\|\Pi_h\|$ denote the operator norm of $\Pi_h$.

We consider the following examples to further illustrate the inf-sup conditions. 
\begin{example}[The obstacle problem, part 2] Recall that $V = H^1_0(\Omega)$ with norm $\|\cdot\|_V = |\cdot|_{H^{1}(\Omega)}$ and $Q = H^{1}_0(\Omega)$. The norm $\|\cdot\|_{Q^\prime}$ is the $H^{-1}(\Omega)$ norm. For $w \in W = L^\infty(\Omega)$, we have
\begin{equation}
\|w\|_{H^{-1}(\Omega)} = \sup_{v\in H^1_0(\Omega)} \frac{|\int_{\Omega} w v \dd x |}{\|\nabla v\|_{L^2(\Omega)}},
\end{equation}
and the bilinear form $b$ is simply the $L^2(\Omega_d)$-inner product, given by
\begin{equation}
b(v, w) = (v, w)_{\Omega_d} = \int_{\Omega} v w \dd x ~\fa v, w \in L^2(\Omega). 
\end{equation} 
The associated LBB condition \cref{eq:inf_sup_continuous} holds with equality for $ \beta = 1$, which immediately follows from the definition of the $H^{-1}(\Omega)$ norm. In \Cref{sec:error_rates}, we provide two examples of compatible subspaces $V_h \times W_h$ satisfying \cref{eq:inf_sup} with their corresponding Fortin operators. 
\end{example}
\begin{example}[The Signorini problem, part 2] Recall that $V = (H_{\mathrm{D}}^1(\Omega))^n$. 
 The bilinear form $b$ is given by 
\begin{equation}
    b(v, w) = \ramy{-}  \int_{\Gamma_{\mathrm{T}}} v \cdot n \,  w \dd s ~\fa v \in (H_{\mathrm{D}}^1(\Omega))^n ,~ w \in L^2(\Gamma_{\mathrm{T}}). 
\end{equation}
The norm $\|\cdot\|_{Q^\prime}$  is defined as follows:
\begin{equation}
\|w\|_{Q^\prime} = \|w\|_{H^{-1/2}(\Gamma_{\mathrm{T}})}
= \sup_{\hat v \in \widetilde{H}^{1/2}(\Gamma_{\rm T}) }  \frac{|\int_{\Gamma_{\mathrm{T}}} w \hat v \dd s| } {\|\hat v\|_{\widetilde{H}^{1/2}(\Gamma_{\rm T})}}
\end{equation}
The following inf-sup condition holds with a constant $\beta > 0$ \cite[Proposition 7.2 and Remark 7.2]{Chouly2023}: 
\begin{equation}
\inf_{w \in L^{\infty}(\Gamma_{\mathrm{T}})} \sup_{v \in \ramy{(H_{\mathrm{D}} ^1(\Omega))^n}} \frac{\int_{\Gamma_\mathrm{T}} v \cdot n w \dd s }{ \|v\|_{H^1(\Omega)} \|w\|_{H^{-1/2}(\Gamma_\mathrm{T})}} \geq \beta .  
\end{equation}
In \Cref{sec:error_rates_sig}, we provide an example of $V_h \times W_h$ satisfying \cref{eq:inf_sup}, and we construct the corresponding Fortin operator.  
\end{example}
\begin{example}[Point-wise divergence constraints] \label{example:divergence_constraint} The present framework also allows one to handle a limited number of cases where $B$ is a differential operator; cf.\ \Cref{sub:ClosedMaps}.
For example, consider 
\[
K = \{ v \in H(\mathrm{div}, \Omega) \mid |\nabla \cdot v |  \leq 1 \text{ a.e.\  in } \Omega \},
\]
and define the energy functional $E(v) = \frac12 \|v\|^2_{H(\mathrm{div}, \Omega)} - (f,v)$ for all $v \in V = H(\mathrm{div},\Omega)$ and some fixed $f \in L^2(\Omega;\R^n)$.
The form $b: V \times Q \rightarrow \R$ then reads 
\[ 
b(v, w) = \int_{\Omega} \nabla \cdot v  \, w \dd x, 
\]
with $Q = L^2(\Omega)$.
For the Legendre function $\mathcal{R}$, one can choose the Fermi--Dirac entropy given in \Cref{example:fermi_dirac} with $\underline{u} = -1$ and $\overline{u} = 1$, although other convenient choices are also appropriate; cf.\ \cite{fu2024locally}. 
.
The continuous LBB condition \cref{eq:inf_sup_continuous} holds thanks to the surjectivity of the divergence operator from $H(\mathrm{div}, \Omega)$ to $L^2(\Omega)$; see, e.g., \cite[Lemma 51.2]{ern2021finite2}. To ensure that \cref{eq:inf_sup} holds, a natural choice of subspace would be the $H(\mathrm{div})$-conforming Raviart--Thomas space for $V_h$ and the broken polynomial space of the same order for $W_h$; see, e.g., \cite[Lemma 51.10]{ern2021finite2}.
\end{example}

\subsection{The PG method is well-defined} We prove that each nonlinear subproblem \cref{eq:discrete_lvpp} has a unique solution.
\begin{theorem}[Existence and uniqueness of solutions] \label{thm:existence_uniqueness}
Assume we are in the setting outlined in \Cref{sec:general_setup,sec:compatibility}.  Then for every $k \geq 1$, the nonlinear saddle point problem \cref{eq:lvpp_g_0}-\cref{eq:lvpp_g_1} admits a unique solution pair $(u_h^{k}, \psi_h^{k})\in V_h\times W_h$.
\end{theorem}
\begin{proof}
We drop the superscript $k$ to simplify notation and start the proof with the uniqueness assertion. Indeed, if $(\hat u_h, \hat \psi_h)$ and $(u_h, \psi_h)$ solve \cref{eq:discrete_lvpp}, then, assuming $\hat u_h \neq u_h$ and using coercivity of $a$ and the strict monotonicity of $\nabla \mathcal R^*$, we obtain 
\begin{align*}
  \ramy{\alpha}  \nu \|\hat u_h - u_h\|_V^2 
    \leq
    \alpha a(\hat u_h - u_h, \hat u_h - u_h)
    &= 
    -  b( \hat u_h - u_h, \hat \psi_h - \psi_h ) 
    \\
    &= 
    - (\nabla \mathcal{R}^*(\hat \psi_h) - \nabla \mathcal{R}^*(\psi_h), \hat \psi_h - \psi_h)_{\Omega_d} 
    \\&
    < 
    0,
\end{align*}
which implies $\hat u_h = u_h$.
\Cref{eq:lvpp_g_0} then implies $b(v_h, \hat \psi_h - \psi_h) = 0$ for all $v_h\in V_h$ and the assumed compatibility assumptions on $V_h$ and $W_h$ yield $\hat \psi_h = \psi_h$.
Indeed, a direct consequence of \cref{eq:inf_sup} is that we can estimate
\begin{equation}
    \|b( \cdot, w)\|_{V_h'} \geq \beta_0 \| w \|_{Q'}, \label{eq:injectivity_b'}
\end{equation}
for all $w\in W_h$.
In other words, the map $w\mapsto b(\cdot,w)$ is injective with closed range between the spaces $(W_h, \|\cdot\|_{Q'})\to(V_h', \|\cdot\|_{V_h'})$.

To show the existence of solutions, we consider the following Lagrangian $\mathcal{L}: V_h \times W_h \rightarrow \mathbb R$: 
\begin{equation}
\label{eq:def_Lagrangian}
    \mathcal{L}(v, w) 
    =
    \frac{\alpha}{2}a(v,v) - \alpha F(v)
    +
    b(v,w) 
    -
    b(v,\psi_h^{k-1})
    - 
    ( \mathcal{R}^* (w) , 1)_{\Omega_d}. 
\end{equation}
Clearly, every critical point of $\mathcal L$ is a solution to 
\cref{eq:lvpp_g_0}-\cref{eq:lvpp_g_1}. We will show the existence of a critical point by minimizing in the first variable and maximizing in the second variable of $\mathcal L$. As $a$ is coercive, for fixed $w\in W_h$ we can find a unique $v(w)\in V_h$ satisfying
\begin{equation} 
\label{eq:def_v(w)}
    v(w) 
    =
    \underset{v\in V_h} \argmin  \,  \mathcal L(v,w).
\end{equation}
Note that $v(w)$ solves 
\begin{equation}\label{eq:optimality_within_saddlepoint_lemma}
    \alpha [a(v(w), v) - F(v)] = b(v,\psi_h^{k-1}) - b(v,w) ~\fa v\in V_h,
\end{equation}
or equivalently, setting $A = v\mapsto a(v,\cdot)$, we have
\begin{equation}\label{eq:formula_for_vw}
    v(w)
    =
    A^{-1}[F + \alpha^{-1}b(\cdot,\psi_h^{k-1}-w)].
\end{equation}
Substituting $v = v(w)\in V_h$ into~\cref{eq:def_Lagrangian}, we obtain
\begin{equation*}
    J(w)
    \coloneqq 
    \mathcal L(v(w), w) 
    =
    - \frac{\ramy{\alpha}}{2}a(v(w), v(w)) - \int_{\Omega_d}\mathcal R^*(w) \dd \mathcal{H}_d.
\end{equation*}
Now, employing formula \cref{eq:formula_for_vw} we bound $\|v(w)\|_V$ from below:
\begin{align*}
    \|v(w)\|_{V}
    &=
    \| A^{-1}[F + \alpha^{-1}b(\cdot,\psi_h^{k-1}-w)] \|_{V}
    \\
    &\geq
    c \| F + \alpha^{-1}b(\cdot,\psi_h^{k-1}-w) \|_{V_h'}
    \\
    &\geq
    c\left[
        \| b(\cdot,w) \|_{V_h'} - \| F + b(\cdot, \psi_h^{k-1}) \|_{V_h'}
    \right]
    \\
    &\geq
    c\left[
        \beta_0 \| w \|_{Q'} - 1
    \right]
\end{align*}
where, in the first step, we used that the isomorphism $A^{-1}:V_h'\to V_h$ is injective with closed range.  In the last estimate, we used \cref{eq:injectivity_b'}. 
\ramy{Here, $c$ is a generic constant independent of $w$. }
Together with the fact that we can lower bound $(\mathcal R^*(w) ,1 )_{\Omega_d}$ by an affine linear function and the equivalence of norms in finite dimensional spaces, we deduce that $-J(w)\to\infty$ as $\|w\|_{Q^\prime}\to\infty$. The map $w\mapsto v(w)$ is continuous as the first term is affine linear and $\mathcal R^*$ is continuous, hence we can guarantee the existence of a maximizer $w^*$ for $J$. 

We now show that the pair $(v(w^*), w^*)$ is the sought-after saddle point of $\mathcal L$. This follows from standard arguments; see e.g., \cite[Exercise 7.16-4]{ciarlet2013linear}. We provide some details for completeness.  First observe that since $\mathcal{L}(v, \cdot)$ is concave, we have that 
\begin{align}
\theta \mathcal{L}(v(\delta_\theta) , w ) + (1-\theta)\mathcal{L}(v(\delta_\theta) , w^* )  \leq \mathcal{L}(v(\delta_\theta) , \delta_\theta ) = J(\delta_\theta) \leq J(w^*),  
\end{align}
for any $\theta \in [0,1]$, $w \in W_h$ and $\delta_{\theta} = \theta w + (1-\theta) w^*$. Thus, 
\begin{align}
\theta \mathcal{L}(v(\delta_\theta), w) \leq J(w^*) + (\theta - 1) \mathcal{L}(v(\delta_\theta) , w^* ) \leq \theta J(w^*) = \theta \mathcal{L}(v(w^*) , w^*), 
\end{align}
where we used that $\mathcal{L}(v(w^*), w^*) \leq \mathcal{L}(v(\delta_\theta) , w^* )$, see \cref{eq:def_v(w)}. Then, we obtain that for any $\theta > 0$, $\mathcal{L}(v(\delta_\theta) , w ) \leq \mathcal{L}(v(w^*), w^*)$. With the continuity of the map $w \mapsto v(w)$ and of $\mathcal{L}(\cdot, w)$, we conclude that 
\begin{equation*}
\mathcal{L}(v(w^*), w)  \leq \mathcal{L}(v(w^*), w^*)  ~\fa w \in W_h.     
\end{equation*}
Hence, 
\begin{equation}
\inf_{v_h \in V_h} \sup_{w \in W_h} \mathcal{L}(v, w )  \leq \sup_{w \in W_h} \mathcal{L}(v(w^*), w ) \leq \mathcal{L}(v(w^*), w^*) = \sup_{w \in W_h} \inf_{v \in V_h} \mathcal{L}(v,w). 
\end{equation}
Since $\sup_{w \in W_h} \inf_{v \in V_h} \mathcal{L}(v,w) \leq \inf_{v_h \in V_h} \sup_{w \in W_h} \mathcal{L}(v, w ) $ always holds, we conclude that 
\[ 
  \sup_{w \in W_h} \mathcal{L}(v(w^*), w ) =  \mathcal{L}(v(w^*), w^*) = \inf_{v \in V_h } \mathcal{L}(v , w^*), 
\] 
which finishes the proof. 
\end{proof} 
In addition to the existence and uniqueness result of \Cref{thm:existence_uniqueness}, we seek stability bounds on the iterates $(u_h^k, \psi_h^k)$, showing that these discrete solutions $u_h^k$ remain uniformly bounded in suitable norms independently of $h$, $\{\alpha_k\}$ and $\ell$. Uniform stability of $\psi_h^\ell$ in weak norms with respect to $h$ is also expected.  Such bounds can be established under additional technical assumptions on $\nabla \mathcal{R}^*$, with details provided in 
\Cref{appendix:stability}; in particular, see \Cref{thm:stability}.

\subsection{Energy dissipation}
We establish an energy dissipation property, which serves as a key tool for proving both stability and convergence of \Cref{alg:main_alg_discrete}.

\begin{lemma}[Energy dissipation]
\label{lemma:energy_dissipation}
The following property holds for all $k \geq 1$:
    \begin{equation}
      E(u_h^{k+1}) + \frac{1}{\alpha_{k+1}}  ( \mathcal{D}^{*}(\psi_h^{k+1}, \psi_h^k) + \mathcal{D}^*(\psi_h^k, \psi_h^{k+1}),  1 )_{\Omega_d} \leq E(u_h^k).            \label{eq:energy_decreasing_discrete}
 \end{equation} 
\end{lemma}
\begin{proof}
Utilizing the coercivity of $a$ \eqref{eq:coercivity_a} and \eqref{eq:lvpp_g_0}--\eqref{eq:lvpp_g_1}, we directly observe that 
\begin{align}
  E(u_h^{k+1}) & \leq E(u_h^k) + \langle E'(u_h^{k+1}), u_h^{k+1} - u_h^k \rangle \\ 
  & =  E(u_h^k) + \frac{1}{\alpha_{k+1}}b (u_h^{k+1} - u_h^k, \psi_h^k - \psi_h^{k+1}) \nonumber  \\ 
& =  E(u_h^k)  -  \frac{1}{\alpha_{k+1}} (\nabla \mathcal{R}^*(\psi_h^{k+1}) -\nabla \mathcal{R}^*(\psi_h^{k}) , \psi_h^{k+1} - \psi_h^k)_{\Omega_d} \nonumber. 
\end{align}
The three points identity \eqref{eq:three_point_iden_star} provides the result. 
\end{proof}
\subsection{Best approximation error estimates}
We now derive a priori best approximation estimates on the error between the discrete iterates $u_h^k$ of \Cref{alg:main_alg_discrete} and the true solution $u^*$ of \cref{eq:general_problem}. In addition, we derive estimates between the dual variables $\lambda_h^k$ and $\lambda^*$ and the observables $o_h^k$ and $o^*$.
These estimates yield convergence rates and a general mesh-independence property, even for low-regularity solutions.

\begin{theorem}[Best approximation of the primal iterates $u_h^\ell$]
\label{thm:BAEP}
Let $u^* \in  K$ be the solution to \cref{eq:general_problem},  let $\{ u_h^k\}_{k=1}^\ell $ be defined via \Cref{alg:main_alg_discrete}. Assume we are in the setting outlined in \Cref{sec:general_setup,sec:compatibility}. Then the following estimate is valid for every $\ell \geq 1$:
\begin{multline}\label{eq:main_estimate}
 \frac{\nu}{2} \| u^* - u_h^\ell\|_V^2
        \leq
  \frac{( \mathcal{D}(o^*, o_h^0) , 1 )_{\Omega_d}}{\sum_{k=1}^\ell \alpha_k}
+ \frac{M^2}{2\nu} \|\Pi_h u^* - u^* \|^2_V \\   + |\langle E'(u^*), \Pi_h u^* - u^* \rangle|  +   \inf_{v \in K} |\langle E'(u^*) , v - u_h^\ell \rangle|, 
\end{multline}
where we recall that $o^* = Bu^*$, $o_h^0= \nabla \mathcal{R}^*(\psi_h^0)$, and $\Pi_h$ is the Fortin map given in \eqref{eq:fortin_map}. 
\end{theorem}
\begin{proof}
Observe that with the coercivity property of $a$ \eqref{eq:coercivity_a}, we have 
for any $u,v \in V$
\begin{equation}
E(u) - E(v) =  \langle E'(\ramy{v}), u- v \rangle + \frac{1}{2} a(u - v, u-v) \geq \langle E'(\ramy{v}), u- v \rangle  + \frac{\nu}{2} \|u - v\|^2_V,   \label{eq:strong_convexity_1}
\end{equation}
Setting $v = u_h^{k}$ above, we arrive at the following inequality for any $u \in V$:
\begin{align}
E(u) &  \geq E(u_h^{k}) + \langle E'(u_h^{k}), u -u_h^{k} \rangle  + \frac\nu2 \|u -  u_h^{k}\|_V^2  \nonumber \\ 
& = E(u_h^{k}) + \langle E'(u_h^{k}), \Pi_h u  -u_h^{k} \rangle  +\langle E'(u_h^{k}), u -\Pi_h u \rangle  + \frac\nu2 \|u -  u_h^{k}\|_V^2. \nonumber
\end{align}
We now rewrite the second term above using \eqref{eq:lvpp_g_0},  the definition of the dual variable $\lambda_h^k$ \eqref{eq:discrete_Lagrange}, and the Fortin map \eqref{eq:fortin_map}. We obtain
\begin{equation}
\langle E'(u_h^{k}), \Pi_h u -u_h^{k} \rangle =  b( \Pi_h u - u_h^k,\lambda_h^k) = b( u - u_h^k, \lambda_h^k). 
\end{equation}
Proceeding, we use some ideas from \cite{eckstein1998approximate,fu2025proximal}. Namely, we use \eqref{eq:lvpp_g_1}, the definition of $o_h^k$ \eqref{eq:tilde_u}, the observation that $\psi_h^k = \nabla \mathcal{R} (o_h^k)$ from \eqref{eq:conj_grad_inverse}, and the three points identity \eqref{eq:three_point_iden_star}. This yields
\begin{align} \label{eq:three_points_primal}
 b( u - u_h^k, \lambda_h^k) = (Bu - Bu_h^k, \lambda_h^k)_{\Omega_d}&  = 
 \frac{1}{\alpha_k}(B u - o_h^k,  \nabla \mathcal{R}(o_h^{k-1}) - \nabla \mathcal{R}(o_h^{k}))_{\Omega_d} \\ 
 & = \frac{1}{\alpha_k} ( \mathcal{D}(Bu, o_h^k) - \mathcal{D}(Bu,o_h^{k-1}) + \mathcal{D}(o_h^k,o_h^{k-1}),1)_{\Omega_d}. \nonumber
\end{align}
Collecting the above and using $\mathcal{D}(o_h^k,o_h^{k-1}) \geq 0$, we obtain the following identity: 
\begin{multline}
E(u_h^k) - E(u) + \frac{1}{\alpha_k} ( \mathcal{D}(Bu, o_h^k) - \mathcal{D}(Bu,o_h^{k-1}), 1)_{\Omega_d}  + \frac\nu2 \|u -  u_h^{k}\|_V^2 \\ \leq \langle E'(u_h^k), \Pi_h u -u\rangle. 
\end{multline}
Proceeding, we multiply by $\alpha_k$, sum from $k= 1$ to $k = \ell$, and use the energy dissipation property \Cref{lemma:energy_dissipation}. This yields 
 \begin{multline}
     \label{eq:Step0_1}
 (E(u_h^\ell) - E(u)) \sum_{k=1}^{\ell} \alpha_{k}  + \frac{\nu}{2}\sum_{k=1}^{\ell} \alpha_{k}\| u -  u_h^{k}\|_V^2 + ( \mathcal{D}(Bu, o_h^\ell) ,1 )_{\Omega_d}\\
 \leq
( \mathcal{D}(Bu, o_h^0) , 1 )_{\Omega_d} + \sum_{k=1}^{\ell} \alpha_{k}  \langle E'(u_h^{k}),  \Pi_h u - u \rangle . 
 \end{multline} 
 We now define the weighted average $\overline{u}_h^\ell = \sum_{k=1}^{\ell} \alpha_{k} u_h^{k} / \sum_{k=1}^{\ell} \alpha_{k}$ and use Jensen's inequality
\begin{align}
\| u  -  \overline{u}_h^\ell\|_V^2 
\leq \frac{\sum_{k=1}^{\ell} \alpha_{k} \| u -  u_h^{k}\|_V^2}{\sum_{k=1}^{\ell} \alpha_{k}}.
\end{align}
Thus, upon dividing~\cref{eq:Step0_1} by $\sum_{k=1}^\ell \alpha_k$ then adding and subtracting $\langle E'(u), \Pi_h u -u\rangle$, we obtain 
\begin{multline}
    E(u_h^\ell) - E(u) + \frac12 \nu \| u -  \overline{u}_h^\ell\|_V^2  + \frac{( \mathcal{D}(Bu, o_h^\ell) ,1 )_{\Omega_d}}{\sum_{k=1}^\ell \alpha_k}\\  \leq  \frac{( \mathcal{D}(Bu, o_h^0) , 1 )_{\Omega_d}}{\sum_{k=1}^\ell \alpha_k}
 + \langle E'(\overline u_h^\ell) - E'(u),  \Pi_h u  - u \rangle + \langle E'(u) , \Pi_h u  - u\rangle. 
\end{multline}
 Proceeding, 
we use the continuity of $a$ \eqref{eq:coercivity_a} and Young's inequality to bound 
\begin{align}
\langle E'(\overline u_h^\ell)  -E'(u),  \Pi_h u  - u \rangle \leq M \| \overline u_h^\ell - u\|_{V}\| \Pi_h u  - u\|_V  \leq \frac\nu2 \| \overline u_h^\ell - u\|_{V}^2 + \frac{M^2}{2\nu}  \| \Pi_h u  - u\|_V^2. \nonumber 
\end{align}
Combining the above estimates yields 
\begin{multline}
\label{eq:Step11}
    E(u_h^\ell) - E(u) + \frac{( \mathcal{D}(Bu, o_h^\ell) ,1 )_{\Omega_d}}{\sum_{k=1}^\ell \alpha_k}\\ \leq  \frac{( \mathcal{D}(Bu, o_h^0) , 1 )_{\Omega_d} }{\sum_{k=1}^\ell \alpha_k}
+ \frac{M^2}{2\nu}  \|\Pi_h u  - u \|^2_V  +   \langle E'(u) ,\Pi_h u - u \rangle. 
\end{multline}
At this stage,  we select $u = u^*$ and note that for any  $v \in K$, \cref{eq:strong_convexity_1,eq:VI} provide the bound 
\begin{align*}
    E(u_h^\ell) - E(u^\ast)
    +
    \langle E'(u^\ast), v - u_h^\ell \rangle 
    &\geq
    \langle E'(u^\ast), v - u^\ast \rangle
    + \frac{\nu}{2} \| u^* -  u_h^\ell\|_V^2
            \geq
    \frac{\nu}{2} \| u^* - u_h^\ell\|_V^2
    \,.
\end{align*}
Thus, adding $\langle E'(u^\ast), v - u_h^\ell \rangle$ to both sides of~\cref{eq:Step11} and recalling that $Bu^* = o^*$, we arrive at
\begin{multline}
 \frac{\nu}{2} \| u^* - u_h^\ell\|_V^2 + \frac{( \mathcal{D}(o^*, o_h^\ell) ,1 )_{\Omega_d}}{\sum_{k=1}^\ell \alpha_k}
    \leq
  \frac{( \mathcal{D}(o^*, o_h^0) , 1 )_{\Omega_d}}{\sum_{k=1}^\ell \alpha_k} + \frac{M^2}{2\nu}  \|\Pi_h u^* - u^* \|^2_V 
 \\  + |\langle E'(u^*) , \Pi_h u^* - u^* \rangle| + |\langle E'(u^*) , v - u_h^\ell \rangle|  . 
\end{multline}
The result follows because the choice $v$ was arbitrary and $\mathcal{D}(o^*, o_h^\ell) \geq 0$. 
\end{proof}

The next two results relate the error in $u_h^\ell$ to the errors in the dual variables $\lambda_h^\ell = (\psi_h^{\ell - 1}  - \psi_h^\ell)/\alpha_\ell$ and observables $o_h^\ell = \nabla \mathcal{R}^*(\psi_h^\ell)$.

\begin{lemma}
[Best approximation of the dual variables $\lambda_h^\ell$] \label{thm:conv_multipliers}
There exists a constant $\beta >0$, such that for any $\ell \geq 1$, \begin{equation} \label{eq:bap_multiplier}
\beta \|\lambda^* - \lambda_h^\ell\|_{Q'} \leq \sup_{v \in V} \frac{  | \langle E'(u^*),  v - \Pi_h v \rangle |}{\|v\|_{V}}  + M \|\Pi_h\| \|u^* - u_h^\ell\|_V. 
 \end{equation}
\end{lemma}
\begin{proof}
We first estimate $\|B'(\lambda^* - \lambda_h^\ell)\|_{V'}$. 
For any $v \in V$, we use the Fortin operator \cref{eq:fortin_map} and equation \cref{eq:lvpp_g_0} to write
\begin{align*} 
\langle B'( \lambda^* - \lambda_h^{\ell}), v \rangle 
& = \langle B'\lambda^*, v - \Pi_h v \rangle + \langle  B'\lambda^*,  \Pi_h v \rangle - \frac{1}{\alpha_\ell } b( v, \psi_h^{\ell - 1}  - \psi_h^\ell)  \\  
& = \langle E'(u^*), v - \Pi_h v \rangle + \langle E'(u^*), \Pi_h v \rangle - \frac{1}{\alpha_\ell } b( \Pi_h v, \psi_h^{\ell - 1}  - \psi_h^\ell) \\ 
& = \langle E'(u^*), v - \Pi_h v \rangle + \langle E'(u^*)  - E'(u_h^\ell), \Pi_h v \rangle. 
\end{align*}
Invoking the continuity of the bilinear form $a$ and the map $\Pi_h$, we obtain that 
\begin{align}\label{eq:B'_estimate}
\| B'( \lambda^* - \lambda_h^{\ell})\|_{V'}
\leq 
\sup_{v \in V} \frac{  | \langle E'(u^*),  v - \Pi_h v \rangle |}{\|v\|_{V}}  + M \|\Pi_h\| \|u^* - u_h^\ell\|_V. 
\end{align}
Since the map $B: V\rightarrow Q$ is surjective, it follows from the closed range Theorem, see e.g. \cite[Lemma A.40]{ern2004theory}, that there exists a constant $\beta >0 $ with 
\begin{equation}\label{eq:result_closed_range}
\beta \|\lambda^* - \lambda_h^\ell\|_{Q
'} \leq \| B'( \lambda^* - \lambda_h^{\ell})\|_{V'}. 
\end{equation}
Combining \cref{eq:B'_estimate} with \cref{eq:result_closed_range} finishes the proof.
\end{proof}

\begin{lemma}[Best approximation of the discrete observables $o_h^\ell$] \label{lemma:nonpoly_approx}
The following estimate holds for all $\ell \geq 1$:
\begin{align} \label{eq:estimate_tilde_general}
\|B u^*  - o_h^\ell \|_{L^2(\Omega_d)} \leq \|B(u^* - u_h^\ell )\|_{L^2(\Omega_d)} + \|(I- P_{h}) (B u_h^\ell - o_h^\ell)\|_{L^2(\Omega_d)}, 
\end{align}
where $P_{h}$ is the $L^2(\Omega_d;\R^m)$-projection operator onto $W_h$, defined by
\begin{equation}
\label{eq:L2Proj}
    (P_{h} o, w_h)_{\Omega_d}
    =
    (o, w_h)_{\Omega_d}
    ~\fa o \in L^2(\Omega_d;\R^m),~ w_h \in W_h
    .
\end{equation} 
\end{lemma}
\begin{proof}
This is a direct consequence of the triangle inequality and the observation that $P_{h} B u_h^\ell = P_{h} o_h^\ell$, which follows from \cref{eq:lvpp_g_1}. 
\end{proof}

\subsection{Convergence rates and mesh-independence}

We now derive abstract error convergence rates from \Cref{thm:BAEP} and \Cref{thm:conv_multipliers,lemma:nonpoly_approx}. 
For simplicity, we focus on the following spaces $V \subset H^1(\Omega)$ and $V \subset H^1(\Omega; \R^n)$. Hereafter, we do not explicitly distinguish between scalar and vector-valued spaces, using the notation for the scalar space while noting that the same results hold mutatis mutandis for vector-valued spaces.
To derive error rates, we require a \textit{reconstruction operator} that allows us to control the last term in \eqref{eq:main_estimate}. In turn, we make the following general assumption, which is critical to the definition of this map. First, we motivate this assumption with a remark.  
\begin{remark}[Motivating \Cref{assumption:maptointerior}]
\Cref{eq:lvpp_g_1} allows the PG method to be viewed as a partially-\textit{nonconforming} finite element method in the sense that the approximations $\{u_h^\ell\}$ generally do not belong to the feasible set $K$, where the true solution $u^*$ resides.
Instead, \cref{eq:lvpp_g_1} characterizes an approximate feasible set $K_h \not\subset K$ containing the iterates $\{u_h^\ell\}$.
For example, consider the obstacle problem, \Cref{example:obstacle}, with the Shannon entropy from \Cref{example:shannon_entropy}, and suppose that $W_h = \mathbb{P}_0(\mathcal{T}_h)$.
Then, by choosing $w_h = \chi_T$ (the indicator function of an element $T$), we see that 
\begin{equation}
\label{eq:ObstacleK_h}
 u_h^\ell \in K_h = \left \{ v_h \in V_h \mid \int_{K} (v_h -\phi) \dd x \ge 0  \right \} .
\end{equation}
As expected from the analysis of \textit{nonconforming} discretizations of obstacle problems \cite[Section 5.2.1]{bartels2015numerical}, one then requires an operator $\mathcal{E}_h : K_h \rightarrow K$ with suitable approximation properties. 
See \Cref{sec:obstacle_problem} for an explicit construction of the reconstruction operator for $K_h$ defined in~\cref{eq:ObstacleK_h}. 
\end{remark}
\begin{assumption}[Reconstruction operator]\label{assumption:maptointerior}
There exists a continuous map $\mathcal{E}_h: V \rightarrow V$ with the property that
\begin{subequations}
\begin{equation}
\mathcal{E}_h: K \cup K_h \rightarrow  K,
\end{equation}
where     
\begin{equation}
K_h :=
\{ 
v_h \in V_h \mid
P_{h} B v_h \in P_{h}(\mathcal{O})
\},
\label{eq:general_set_Kh}
\end{equation}
and $P_{h} \colon L^2(\Omega_d;\R^m) \to W_h$ is the $L^2$-projector onto $W_h$ \cref{eq:L2Proj}.
\end{subequations}
\end{assumption}
The set $K_h$ contains functions $u_h\in V_h$ such that the $L^2$-projection of $B u_h \in W_h$ matches the $L^2$-projection of a constrained observable $o \in \mathcal{O}$. 
It is constructed to ensure that $u_h^k \in K_h$ for all iterations $k$; see \cref{eq:lvpp_g_1}.
To derive error estimates, we require the following quasi-interpolation assumption on the stability and approximation of the Fortin and reconstruction operators. The decomposition of the reconstruction operator given in~\cref{eq:EnrichingMapDecomposition} reflects the constructions used in \Cref{sec:error_rates,sec:error_rates_sig}, below. 
\begin{assumption} \label{assumption:err_rates}
Assume that the Fortin operator satisfying \cref{eq:fortin_map} is stable in the sense that  
\begin{subequations}
\label{eq:stability_fortin_general}
\begin{align}
\|\Pi_h v \|_{L^2(\Omega)} & \lesssim \|v\|_{L^2(\Omega)} + h \|\nabla v\|_{L^2(\Omega)}, \label{eq:fortin_stability_assump_L2} \\ 
\|\nabla (\Pi_h v)\|_{L^2(\Omega)} &  \lesssim \|\nabla v\|_{L^2(\Omega)} , \label{eq:fortin_stability_assump_H1}
\end{align} 
\end{subequations}
for all $v \in V$. Further, assume that the reconstruction operator  $\mathcal{E}_h$  of \Cref{assumption:maptointerior} is affine linear:
\begin{equation}
\label{eq:EnrichingMapDecomposition}
    \mathcal{E}_h w = \mathcal{C}_h w + \varepsilon, 
\end{equation}
where $\mathcal{C}_h: \ramy{H^1(\Omega) \rightarrow \mathbb{P}_1(\mathcal{T}_h) \cap H^1(\Omega)}$  is linear and $\varepsilon \in V$. 

For $0 \leq s \leq 1$, $0 \leq t \leq 1$, and $w \in H^{1+s}(\Omega) \cap V$, assume that 
\begin{subequations}
\label{eq:map_rates}
\begin{align}
    \| w - \Pi_h w \|_{L^2(\Omega)} + h \|\nabla (w - \Pi_hw)\|_{L^2(\Omega)} 
    & \lesssim 
    h^{1+s} | w |_{H^{1+s}(\Omega)}, \label{eq:fortin_rates} \\ 
      \| w - \mathcal{C}_h w \|_{L^2(\Omega)} + h  \|\nabla (w - \mathcal{C}_h w)\|_{L^2(\Omega)} 
    &  \lesssim  
    h^{1+s} | w |_{H^{1+s}(\Omega)} ,   \label{eq:map_Ch_rates} \\ 
   \|\varepsilon \|_{H^t(\Omega)} & \lesssim h^{1+s-t} .\label{eq:map_Eh_rates_r}
    \end{align}
\end{subequations}
\end{assumption}

\begin{theorem}[Convergence rates] \label{cor:err_rate_general}
 Assume that $u^* \in H^{1+s}(\Omega)$ and $E'(u^*)\in H^{r-1}(\Omega)$ for some $s,r \in (0,1]$, and let $\|\psi_h^0\|_{L^\infty(\Omega_d;\R^m)} \lesssim 1$. 
Let \Cref{assumption:R,assumption:maptointerior,assumption:err_rates} hold. Then 
\begin{equation}
\label{eq:GeneralConverge}
\| u^* - u_h^\ell\|_{H^1(\Omega)}^2 + \|\lambda^* - \lambda_h^\ell \|^2_{Q'} \lesssim \frac{C_{\mathrm{stab}}}{\sum_{k=1}^\ell \alpha_k}
+ C_{\mathrm{reg}}\, h^{2\cdot\min\{r,s\}}
\end{equation}
for all $\ell \geq 1$ and $h > 0$, where $C_{\mathrm{stab}}$ and $C_{\mathrm{reg}}$ are positive constants independent of $h$, $\ell$, and the parameters $\alpha_k$. 
\end{theorem}
\begin{proof}
Since $E'(u^*) \in H^{r-1}(\Omega)$, we further bound the last two terms in \cref{eq:main_estimate} by $\|  E'(u^*) \|_{H^{r-1}(\Omega)}  \|  \Pi_h u^* -  u^* \|_{H^{1-r}(\Omega)} $ and by $\|  E'(u^*) \|_{H^{r-1}(\Omega)}  \| v - u_h^\ell \|_{H^{1-r}(\Omega)}$ respectively. Thus, we obtain 
\begin{multline}\label{eq:main_estimate_1}
 \frac{\nu}{2} \| u^* - u_h^\ell\|_V^2
        \leq
  \frac{( \mathcal{D}(o^*, o_h^0) , 1 )_{\Omega_d}}{\sum_{k=1}^\ell \alpha_k}
+ \frac{M^2}{2\nu}  \|\Pi_h u^* - u^* \|^2_{H^1(\Omega)} \\   + \|  E'(u^*) \|_{H^{r-1}(\Omega)}  \|  \Pi_h u^* -  u^* \|_{H^{1-r}(\Omega)} +  \inf_{v \in K} \|  E'(u^*) \|_{H^{r-1}(\Omega)}  \| v - u_h^\ell \|_{H^{1-r}(\Omega)}. 
\end{multline}
We now proceed by bounding each term in  \cref{eq:main_estimate_1}. Recalling  that  $o_h^0 = \nabla \mathcal{R}^*(\psi_h^0)$,  $\|\psi_h^0\|_{L^\infty(\Omega_d;\R^m)} \lesssim 1$, and the definition of $\mathcal{D}$ \eqref{eq:Bregman}, it readily follows that
\begin{equation}
    \mathcal{D}(o^*, o_h^0)   \leq C_{\mathrm{stab}},  \label{eq:bound_initial_distance}
\end{equation}
for some constant $C_{\mathrm{stab}}$ independent of $h, \ell$ and $\alpha_k$ since $\mathcal{R}$ is continuous on $\mathcal{O}$ and $\nabla \mathcal{R}^*$ is continuous on $L^\infty(\Omega_d;\R^m)$.
Proceeding, \Cref{assumption:err_rates}  directly  provides the  bound 
\begin{align}
\|u^* - \Pi_h u^* \|_{H^1(\Omega)}^2  \lesssim  h^{2s} |u^*|^2_{H^{1+s}(\Omega)}.  \label{eq:bound_first_term_err}
\end{align}
To handle the $\|u^* - \Pi_h u^*\|_{H^{1-r}(\Omega)}$  term in~\cref{eq:main_estimate_1}, we first note that from space interpolation between $L^2(\Omega)$ and $H^1(\Omega)$ (see, e.g., \cite[Chapter 14, Proposition 14.1.5]{brenner2008mathematical}) and \cref{eq:fortin_rates}, we obtain that for any $w \in H^{1+s}(\Omega)$, 
\begin{align}
\|w - \Pi_h w\|_{H^{t}(\Omega)} + \|w - \mathcal{C}_h w\|_{H^{t}(\Omega)} \lesssim 
h^{1+s-t } |w|_{H^{1+s}(\Omega)},  \quad s, t \in [0,1].  \label{eq:interpolation_fortin}
\end{align}
Therefore, we obtain that 
\begin{equation}  \label{eq:bound_second_term_err}
\|\Pi_h u^* - u^*\|_{H^{1-r}(\Omega)} \lesssim h^{r+s} |u^*|_{H^{1+s}(\Omega)}.  
\end{equation} 
 It remains to estimate the term in \cref{eq:main_estimate_1} involving $\|v - u_h^\ell\|_{H^{1-r}(\Omega)}$. To this end, we choose $v = \mathcal{E}_h(u_h^\ell)$. This is a valid choice since $u_h^\ell \in K_h$ by \cref{eq:lvpp_g_1} and $\mathcal E_h u_h^\ell \in K$ by \Cref{assumption:maptointerior}. With the help of \cref{eq:map_rates} and \eqref{eq:interpolation_fortin}, we now formulate the following estimate:
\begin{align} \label{eq:bound_third_term_err}
\|u_h^\ell -\mathcal{E}_h u_h^\ell \|_{H^{1-r}(\Omega)}
& =  \|(u_h^\ell - u^*) -\mathcal{C}_h ( u_h^\ell - u^*) + (u^* -\mathcal{E}_h u^*)  \|_{H^{1-r}(\Omega)}  \\
&  \nonumber
\leq \|(u_h^\ell - u^*) -\mathcal{C}_h ( u_h^\ell - u^*)\|_{H^{1-r}(\Omega)}  + \|u^* -\mathcal{E}_h u^*  \|_{H^{1-r}(\Omega)} \\ 
& \nonumber
\leq c h^r \|\nabla (u_h^\ell - u^*)\|_{L^2(\Omega)} + \tilde c h^{r+s} (|u^*|_{H^{1+s}(\Omega)} + 1)\\ 
& \nonumber 
\leq \frac14 \nu   \|E'(u^*)\|_{H^{r-1}(\Omega)}^{-1}   \|\nabla (u_h^\ell - u^*)\|_{L^2(\Omega)}^2 + c^2  \nu^{-1}  \|E'(u^*)\|_{H^{r-1}(\Omega)} h^{2r} \\ \nonumber &  \quad + \tilde c h^{r+s} (|u^*|_{H^{1+s}(\Omega)} + 1), 
\end{align}
where $c$ and $\tilde c$ are mesh-independent constants.  
Incorporating \cref{eq:bound_initial_distance}, \cref{eq:bound_first_term_err}, \cref{eq:bound_second_term_err}, and \cref{eq:bound_third_term_err} into \cref{eq:main_estimate_1} yields the required bound on $\|u_h^\ell - u^*\|_{H^1(\Omega)}$.  

The estimate on $\|\lambda^* - \lambda_h^\ell \|_{\ramy{Q'}}$ follows from \Cref{thm:conv_multipliers}.
In particular, the first term in \cref{eq:bap_multiplier} is bounded by:
$$\sup_{v \in \ramy{V}} \frac{  | \langle E'(u^*),  v - \Pi_h v \rangle |}{\|v\|_{\ramy{V}}} \leq \sup_{v \in \ramy{V}} \frac{  \|  E'(u^*) \|_{H^{r-1}(\Omega)} \|  v - \Pi_h v\|_{H^{1-r}(\Omega)}}{\|v\|_{\ramy{V}}} \lesssim  h^{r}, $$
using \cref{eq:interpolation_fortin} for the second inequality.  
\end{proof}

The final result of this section follows from the fact that the approximation error in~\cref{eq:GeneralConverge} is controlled by independent optimization and discretization error terms, each depending only on $\ell$ and $h$, respectively.

\begin{corollary}[Asymptotic mesh-independence]
\label{cor:WAMI}
Let $\epsilon > 0$.
Under the assumptions of~\Cref{cor:err_rate_general}, there exists a critical mesh size $h_\epsilon > 0$ and iteration number $\ell_\epsilon \geq 1$ such that
    \begin{equation} \label{eq:mesh_independ}
        \| u^\ast -  u^\ell_h \|_V
        +
        \| \lambda^\ast -  \lambda^\ell_h \|_{Q'}
        \leq
        \epsilon
        ~\fa
        0 < h \leq h_\epsilon
        \text{ and }
        \ell \geq \ell_\epsilon
        \,.
    \end{equation}
\end{corollary}

\subsection{Discussion of results}

\Cref{thm:BAEP,cor:err_rate_general} rigorously demonstrate three important features of the PG method, \Cref{alg:main_alg_discrete}, for the first time.
To date, these features have only been observed numerically \cite{keith2023proximal,fu2024locally,dokken2025latent,papadopoulos2024hierarchical}.
\begin{itemize}[leftmargin=*]
    \item 
    First, the PG iterates $u^\ell_h$ converge to the true solution $u^\ast$ as $\ell\to \infty$ and $h \to 0$, even with bounded proximity parameters, $\alpha_k \leq \textrm{const.}$
    This stands in contrast to penalty \cite[Chapter 1.7]{Glowinski1984} and interior point methods \cite[Chapter 19]{nocedal2006numerical} --- and even augmented Lagrangian methods in infinite-dimensional spaces \cite{ito1990augmented,antil2023alesqp} --- which all require taking a relaxation parameter to a singular limit.
    \item 
    Second, the number of iterations is asymptotically independent of the mesh size.
    More specifically, \Cref{cor:WAMI} shows that a user can guarantee convergence to any desired accuracy by independently selecting the mesh size~$h$ and iteration count~$\ell$.
    Asymptotic mesh-independence is a desirable feature of a numerical method \cite{schwedes2017mesh}, often observed in, e.g., Newton's method \cite{allgower1986mesh,weiser2005asymptotic}.
    Notably, asymptotic mesh-independence is sometimes witnessed in the primal-dual active set method \cite{hintermuller2002primal}, but for a more restrictive class of problems than considered here \cite{hintermuller2004mesh,papadopoulos2026mesh}.
    \item 
    Finally, PG can be applied without modification to problems with low-regularity solutions or multipliers.
    Indeed, \Cref{cor:err_rate_general} provides sufficient conditions on $\Pi_h$ and $\mathcal{E}_h$ to obtain optimal convergence rates in $h$ depending naturally on the solution $u^\ast$ and Fr\'echet derivative $E^\prime(u^\ast)$.
    The following two sections show that these abstract conditions are checkable in practice.
    It is well-known that penalty methods can also be applied without modification in the low-regularity setting \cite{hintermuller2006feasible}.
    However, to the best of our knowledge, the existing theory for discretized penalty methods requires $u^* \in H^{1+s}(\Omega)$ and $s \in (1/2,1]$ when applied to the obstacle and Signorini problems.
    See \cite{scholz1984numerical,gustafsson2017finite} for the obstacle problem and \cite[Section 6]{Chouly2023} for the Signorini problem.
    Notably, the PG results hold for any $s \in \ramy{(0,1/2)}$.
\end{itemize}

\begin{remark}[Higher-order convergence rates] The current estimates deliver at most first-order convergence rates, in which case it is required that $u^\ast \in H^2(\Omega)$ and $E'(u^*)\in L^2(\Omega)$. However, the PG method is not limited to low-degree polynomial subspaces $V_h \times W_h$ and, notably, high-order rates have been numerically observed in several studies with high-degree subspaces \cite{keith2023proximal,fu2024locally,papadopoulos2024hierarchical}.
Thus, a crucial question remains unanswered: If the exact solution is sufficiently smooth, then what are the most general conditions that guarantee high-order convergence rates?
Indeed, a major difficulty in extending the current framework to the examples below lies in constructing high-order positivity-preserving approximations, which are encoded in the map $\mathcal{E}_h$. This is known to be a challenging task, impeded by some impossibility results \cite{nochetto2002positivity}.
\end{remark}
\section{Application I: The obstacle problem} \label{sec:error_rates}
In this section, we apply the general framework developed in \Cref{sec:main_results} to derive error estimates for the following unilateral obstacle problem, see also \Cref{example:obstacle}: 
\begin{align}
\min_{v\in K} E(v),   \quad E(v) = \frac{1}{2} \|\nabla v\|^2_{L^2(\Omega)} - (f,v)_\Omega
,
\label{eq:energy_min_obstacle}
\end{align}
where the closed and convex set $K$ is 
\begin{equation}
K = \{v \in H^1_0(\Omega) \mid v \geq \phi \text{ a.e.\ in  }  \Omega \}, 
\end{equation}
and the obstacle $\phi \in H^1(\Omega) \cap  L^{\infty}(\Omega)$ is fixed. For simplicity, we  assume that $ \phi \vert_{\partial \Omega} = - \delta \leq  0$ for some non--negative constant $\delta$.  It is well-known that there exists a unique solution to \cref{eq:energy_min_obstacle}  \cite[Theorem 6.9-1]{ciarlet2013linear}, here denoted by $u^\ast$.  

For completeness, we write the PG  method for the obstacle problem in \Cref{alg:alg_obstacle}.
In this case, we have used
\begin{equation}
    \mathcal{R}(u) = (u-\phi) \ln(u - \phi) - (u - \phi), \label{eq:R_obstacle}
\end{equation}
which implies that
\begin{equation}
    \nabla \mathcal{R^*} (\psi) = \exp(\psi) + \phi.  \label{eq:nablaR*_obstacle}
\end{equation}
\Cref{assum:cont} holds for this choice of Legendre function \cite[Theorem~4.1]{keith2023proximal}; see also \cite[Proposition~A.8]{keith2023proximal}.
Alternative choices for $\mathcal{R}$ are also admissible, without affecting the proceeding error analysis.
\begin{algorithm}[htb]
\caption{The Proximal Galerkin Method for the Obstacle Problem}
\begin{algorithmic}[1]\label{alg:alg_obstacle}
    \State \textbf{input:} Initial latent solution guess $\psi_h^0  \in W_h$, a sequence of positive proximity parameters $\{\alpha_k\}$.
    \State Initialize \(k = 1\). 
    \State \textbf{repeat}
    \State \quad Find $u_h^{k} \in  V_{h}$
and $\psi_h^{k} \in W_h$ such that 
   \begin{subequations} \label{eq:discrete_lvpp_obstacle}
\begin{alignat}{2}  
\alpha_k \, (\nabla u^{k}_h, \nabla v_h )  + (v_h,\psi^{k}_h - \psi_h^{k-1}) & = \alpha_k (f, v_h)  && ~\fa v_h \in V_h, \label{eq:lvpp_obstacle_0}\\ 
(u^{k}_h, w_h)  - (\exp(\psi^{k}_h) + \phi , w_h) & = 0 && ~\fa  w_h \in W_h.  \label{eq:lvpp_obstacle_1}
\end{alignat}
\end{subequations} 
\State \quad Assign  \(k \gets k + 1\).
    \State \textbf{until} a convergence test is satisfied.
\end{algorithmic}
\end{algorithm}

In this section, where $\Omega_d = \Omega$ and $B=\mathrm{id}$, we use the following notation to remain consistent with \cite{keith2023proximal}, where the PG method was first introduced:
\begin{equation}
\label{eq:BoundPreservingSolution}
\tilde u_h^k 
:=
o_h^k = \exp(\psi_h^{k}) + \phi. 
\end{equation}
We also consider two different choices for the finite element subspaces $V_h \times W_h$, both also introduced in \cite{keith2023proximal}. 
\medskip

\noindent\textbf{Main goal}: We derive error estimates in \Cref{cor:Obstacle} for the PG method applied to the obstacle problem (\Cref{alg:alg_obstacle}) for the two choices of $V_h \times W_h$ given in \cref{eq:bubble_pair,eq:continuous_Lagrange}, below.
To this end, we utilize \Cref{cor:err_rate_general} and verify that \Cref{assumption:maptointerior,assumption:err_rates} hold for each pair of spaces, respectively.
We follow with preliminary results for each case. 
\medskip

\textbf{Case I.} $(\mathbb{P}_1\text{\normalfont-bubble}, \mathbb{P}_{0}\text{\normalfont-broken})$   Define the space 
\begin{equation} 
    \mathbb{B}(T) = \operatorname{span}\{b_T\},
\end{equation}
where the bubble function $b_T\colon T \to \mathbb{R}$ is the product of the linear nodal basis functions of the element $T$.
The pair $(V_h,W_h)$ is then defined as
\begin{equation} \label{eq:bubble_pair}
\begin{aligned}
    V_h &  = \{ v \in L^\infty(\Omega) \mid  v \vert_T \in \mathbb{P}_1(T)\oplus \mathbb{B}(T) ~\fa T \in \mathcal{T}_h \} \cap H^1_0(\Omega), \\ 
      W_h &= \mathbb{P}_{0}(\mathcal{T}_h).  
    \end{aligned}
\end{equation}
\smallskip

\textbf{Case II.} $(\mathbb{P}_1,\mathbb{P}_{1})$, i.e., continuous Lagrange elements.

Define
\begin{equation}
 V_h = W_h =  \mathbb{P}_1(\mathcal{T}_h) \cap H^1_0(\Omega) .
 \label{eq:continuous_Lagrange}
\end{equation}
Refer to \Cref{rem:Dirichlet} for enforcing Dirichlet boundary conditions with these elements.  \ramy{As such, we assume that  $\phi \vert_{\partial \Omega
}= -\delta < 0.$}

\begin{remark}[Dirichlet boundary conditions]
\label{rem:Dirichlet}
Both of the $V_h$ subspaces in~\cref{eq:bubble_pair,eq:continuous_Lagrange} imply fixing the degrees of freedom of $u_h^k$ on the Dirichlet boundary.
This is a textbook procedure that is simplified in our setting because of the homogeneous boundary conditions specified in~\cref{eq:energy_min_obstacle}.
Likewise, the continuous Lagrange elements defining $W_h$ 
 in~\cref{eq:continuous_Lagrange} also require fixing the degrees of freedom of the latent variable $\psi_h^k$ on the Dirichlet boundary.
In this case, we construct the lift
$$ \psi_{0,h} : =  \sum_{z \in \mathcal{N}_h^\partial} \nabla \mathcal{R} (0) (z)  \varphi_z , $$
where $\mathcal{N}_h^{\partial}$ are the boundary dofs and $\{\varphi_z\}_{z \in \mathcal{N}_h}$ are the global shape functions corresponding to the (nodal) dofs $\mathcal{N}_h \cup \mathcal{N}_h^{\partial}$.
This construction ensures that the bound-preserving discrete solution $\tilde{u}_h^k$~\cref{eq:BoundPreservingSolution} satisfies the same homogeneous Dirichlet boundary conditions as the solution $u^\ast$ it's meant to approximate.
Indeed, observe that
\[
    \tilde{u}_h^k|_{\partial\Omega} = \nabla \mathcal{R}^*(\psi_h^k|_{\partial\Omega} + \psi_{0,h}|_{\partial\Omega})|_{\partial\Omega} = \nabla \mathcal{R}^*(\nabla \mathcal{R} (0) )|_{\partial\Omega} =  0
    ,
\]
with $\nabla \mathcal{R}^*(\psi) = \exp(\psi) + \phi$ and $\nabla \mathcal{R}(u) = \ln(u-\phi)$ for the particular choice of Legendre function in \Cref{alg:alg_obstacle}.
In turn, when using the elements in Case II, we understand \cref{eq:lvpp_obstacle_1} as 
\begin{align}
(u_h^k, w_h) - (\nabla \mathcal{R}^*(\psi_h^k + \psi_{0,h}), w_h)  = 0 ~\fa w_h \in W_h. 
\end{align}
\end{remark}

\subsection{Preliminaries for the $(\mathbb{P}_1\text{\normalfont-bubble}, \mathbb{P}_{0}\text{\normalfont-broken})$-element pair} \label{subsec:obstacle_prelim1}
We begin with the Fortin operator for the subspaces in \cref{eq:bubble_pair}, and establish the operator's approximation properties required in \Cref{assumption:err_rates}. 
\begin{lemma}[Fortin operator]\label{lemma:fortin_bubble}
The subspaces given in \cref{eq:bubble_pair} satisfies the inf-sup condition \cref{eq:inf_sup} and there exists a stable Fortin operator satisfying \cref{eq:fortin_map} and 
\begin{equation} \label{eq:stability_fortin}
|\Pi_h w |_{H^{m}(\Omega)} \lesssim |w |_{H^{m}(\Omega)} \quad ~\fa
w \in
\ramy{ H^1_0(\Omega)},  m \in \{ 0 ,1 \}
\end{equation}
In addition, for every $ 0 \leq s\leq 1$ and for all $ w \in H^{1+s}(\Omega) \cap H^1_0(\Omega), $ it holds that
\begin{alignat}{2}\label{eq:fortin_approximation}
\| w - \Pi_h w \|_{L^2(\Omega)} + h \|\nabla (w - \Pi_h w)\|_{L^2(\Omega)} & \lesssim h^{1+s} | w |_{H^{1+s}(\Omega)}.  \end{alignat}
\end{lemma}
\begin{proof}
The proof is based on the arguments in \cite[Appendix B]{keith2023proximal}. The operator $\Pi_h:L^2(\Omega) \rightarrow V_h$ is constructed as follows: 
\begin{align}
\Pi_h  = \tilde{\mathcal{I}}_h  + \tilde \Pi_h (I - \tilde{\mathcal{I}}_h ), \label{eq:structure_fortin}
\end{align}
where $\tilde{\mathcal{I}}_h: L^1(\Omega) \rightarrow \mathbb{P}_1(\mathcal{T}_h) \cap H^1_0(\Omega)$ is the quasi-interpolant introduced in \cite[Section 6]{ern2017finite}, $I$ is the identity operator, and $\tilde \Pi_h: L^2(\Omega) \rightarrow V_h$ is defined element-wise to satisfy $ (\tilde \Pi_h v )\vert_T := b_T v_T \in \mathbb{B}(T)$, where  $v_T \in \mathbb{P}_{0}(T) = \mathbb{R}$ solves 
\begin{equation}
(b_T v_T , \varphi)_T = (v, \varphi)_T ~\fa \varphi \in \mathbb{P}_{0}(T). 
\end{equation}
It is easy to see that $v_T = (v,1)_T/(b_T,1)_T$ is the unique solution to the above.  
Owing to mesh regularity and to the properties of the bubble function \cite[Lemma 4.1]{verfurth1994posteriori}, we have 
\begin{align}
\|v_T\|^2_{L^2(T)} \lesssim (b_T v_T,v_T)_T = (v, v_T)_T \leq \|v\|_{L^2(T)} \|v_T\|_{L^2(T)}.
\end{align}
Therefore, we obtain 
\begin{equation}
\|\tilde \Pi_h v \|_{L^2(T)}= \|b_T v_T\|_{L^2(T)} \leq \|v_T\|_{L^2(T)} \lesssim  \|v\|_{L^2(T)}. \label{eq:l2_stab_fortin_0} 
\end{equation}
Using \cref{eq:l2_stab_fortin_0} in \cref{eq:structure_fortin} along with the triangle inequality and the stability of $\tilde{\mathcal I}_h$, we obtain the stability of $\Pi_h $ in $L^2(\Omega)$. To show stability in the $H^1$-semi norm, we apply triangle inequality, a local inverse estimate, \cref{eq:l2_stab_fortin_0}   and the stability and approximation properties of $\tilde{\mathcal{I}}_h$ \cite[Lemma 6.3 and Theorem 6.4]{ern2017finite}:
\begin{multline} \label{eq:fortin_stab_H1_0}
|\Pi_h w|_{H^1(T)}  \lesssim |\tilde{\mathcal I}_h w|_{H^1(T)} + h_T^{-1}\| \tilde \Pi_h ( I -\tilde{\mathcal I}_h ) w \|_{L^2(T)}  \\  \lesssim |\tilde{\mathcal I}_h w|_{H^1(T)} + h_T^{-1}\| w  -\tilde{\mathcal I}_h  w \|_{L^2(T)} 
\lesssim |w|_{H^1(\Delta_T)}, 
\end{multline}
where $\Delta_T$ is a macro-element. Summing over mesh elements and using mesh regularity shows stability in the $H^1(\Omega)$-seminorm. We conclude that \cref{eq:stability_fortin} holds. This stability estimate and the observation that  
\[ 
(\Pi_h v, w_h) = (\tilde{\mathcal I}_h v , w_h) + (\tilde \Pi_h (v - \tilde{\mathcal I}_h v), w_h) =   (v,w_h) ~\fa w_h \in W_h
\]
shows \cref{eq:inf_sup}. The stated error estimate \cref{eq:fortin_approximation} is proven by applying the triangle inequality, the stability of $\Pi_h$ \cref{eq:stability_fortin}, and the approximation properties of $\mathcal{\tilde{I}}_h$ \cite{ern2017finite}.  We omit the details.
 \end{proof} 
Note that the discrete iterates $u_h^k$ have bound-preserving local averages. Indeed, testing \cref{eq:lvpp_obstacle_1} by the indicator function of one element (an admissible test function because $W_h = \mathbb{P}_{0}(\mathcal{T}_h)$) and using that $\nabla \mathcal{R}^*(\psi) > \phi$ shows that $u_h^k \in K_h$ for all $k$, with
\begin{equation} \label{eq:discrete_constraint_set}
    K_h = \left\{ v_h \in V_h \mid  \int_{T} (v_h - \phi ) \dd x \geq 0 \; ~\fa  T \in \mathcal{T}_h  \right\}. 
\end{equation}

To satisfy \Cref{assumption:maptointerior}, we first require a map defined over $V$ and mapping the set $K_h$ to the set $K$. We will later shift this map to construct the reconstruction operator $\mathcal{E}_h$, see \cref{eq:reconstruction_map_obstacle}. The natural choice is the Cl\'ement interpolant \cite{clement1975approximation,bartels2015numerical}. Here, we \ramy{devise} a specialized variant of this interpolant that employs weighted local averages over vertex patches to maintain second-order accuracy for smooth functions. 
{
\color{black}
We will make use of element patches around interior nodes $\mathcal{N}_h$ of the mesh. For $z \in \mathcal{N}_h$, we denote by 
\[ 
\mathcal{T}_z^1 : = \{ T\in\mathcal{T}_h \mid \{ z\} \cap  T \neq \emptyset \},  \quad  \Delta_z^1 = \operatorname{int} \{  \cup_{T \in \mathcal{T}_z^1} T \} 
\] 
For $k \geq 2$, higher order element patches are defined recursively by  
\[
  \mathcal T_z^{k}=\bigl\{T\in\Th \mid  T\cap \overline \Delta_z^{k-1}\ne\emptyset\bigr\},
  \qquad
  \Delta_z^{k}= \operatorname{int} \{  \cup_{T\in\mathcal T_z^{k}} T \} .
\]
Denote the centroid of an element $T \in \mathcal{T}_h$ by $s_T$:
$$s_T = \frac{1}{n+1}\sum_{v\in \mathcal V_T} v.  $$  
The collection of the centroids of the elements in $\mathcal{T}_z^k$ is denoted by 
\begin{equation}
S_z^k = \{ s_T \mid T \in \mathcal{T}_z^k \}. 
\end{equation}
We also define the set of vertices that lie on the boundary $\partial \Omega$ and belong to the vertex patch: 
\begin{equation}
B_z^k = \{ b \in (\cup_{T\in \mathcal \mathcal{T}_z^k} \mathcal{V}_T)  \cap \mathcal{N}_h^\partial\}. 
\end{equation}
The following technical assumption on the mesh is required to enable the construction of our bound preserving optimal quasi--interpolant. In \Cref{sec:locally_quasi_uniform}, we prove that locally quasi--uniform meshes (see \Cref{ass:lqu}) satisfy \Cref{assum:conv_hull}. Since local quasi-uniformity is not necessary for \Cref{assum:conv_hull} to hold, we keep \Cref{assum:conv_hull} in our presentation. 
\begin{assumption}\label{assum:conv_hull}
There exists a constant $k_0$ independent of $h$ such that for any $z \in \mathcal{N}_h$
\begin{equation}
z \in \operatorname{conv} (S_z^{k_0} \cup B_z^{k_0
}). 
\end{equation}
 In other words, there exist weights 
$\{\alpha_{z,T}\}_{T\in\mathcal T_z^{k_0}}$ and $\{\beta_{z,b}\}_{b \in B_z^{k_0}}$ with
\begin{equation}
\label{eq:conv_hull_weights}
  \alpha_{z,T}, \beta_{z,b} \ge0,\qquad \sum_{T \in \mathcal{T}_z^{k_0}} \alpha_{z,T} +\sum_{b \in B_z^{k_0}} \beta_{z,b} =1,\qquad \sum_{T \in \mathcal{T}_z^{k_0}} \alpha_{z,T} s_T + \sum_{b \in B_z^{k_0}} \beta_{z,b} b =z.
\end{equation}
\end{assumption}

We now define $\mathcal{C}_h: H^1(\Omega) \rightarrow \mathbb{P}_{1}(\mathcal{T}_h) \cap H^1(\Omega)$ as follows: 
\begin{equation} \label{eq:clement_case_1}
\mathcal{C}_h v = \sum_{z \in \mathcal{N}_h \cup \mathcal{N}_h^\partial} v_z  \varphi_z,  \quad v_z = \sum_{ T \in \mathcal T_z^{k_0} } \frac{\alpha_{z,T}}{|T|}\int_{T} v \dd x +\sum_{b\in B_z^{k_0}}\beta_{z,b}\,(\mathcal{SZ}_h v)(b) \, \text{ if } z \in \mathcal{N}_h. 
\end{equation}
For $z \in \mathcal{N}_h^\partial$, we set $v_z = (\mathcal{SZ}_h v )(z)$ where $\mathcal{SZ}_h: H^1(\Omega) \rightarrow \mathbb{P}_1(\mathcal{T}_h)$ is the canonical Scott--Zhang interpolant defined in \cite{scott1990finite}.  
Note that with this definition, we recover that $(\mathcal{C}_h v - v) \vert_{\partial \Omega} = 0 $ for any $v$ with piecewise-polynomial boundary data.

\begin{lemma}[Approximation properties of the modified Cl\'ement interpolant]\label{lemma:Clement} Let Assumption~\ref{assum:conv_hull} hold and if $k_0 >1$ assume that 
\begin{equation}
\gamma_1 h_T \leq h_{T'} \leq \gamma_2 h_T \quad \fa T, T' \in \mathcal{T}_z^{k_0}. \label{eq:local_uniformity}
\end{equation}    
For any $T \in \mathcal{T}_h$ and for $ 0 \leq s \leq 1$, it holds that
\begin{alignat}{2}\label{eq:clement_error_estimate}
\|w- \mathcal{C}_h w \|_{L^2(T)} + h_T \|\nabla(w - \mathcal{C}_h w)\|_{L^2(T)}  &\lesssim h_T^{1+s} | w|_{H^{1+s}(\Delta_T)}  && ~\fa w \in H^{1+s}(\Delta_T), 
\end{alignat}
where $\Delta_T$ is the macro-element given by
$\Delta_T = \bigcup_{z \in \mathcal{V}_T} \Delta_z^{k_0}$. 
\end{lemma}
\begin{proof}
The key observation, as we show next, is that this interpolant preserves linear functions locally; i.e., $\mathcal{C}_h p = p$ on $T$ for any $ p \in \mathbb{P}_1(\Delta_T)$ (the affine functions over the $\Delta_T$) and $T \in \mathcal{T}_h$. For an interior node $z \in \mathcal{N}_h $ and any $p \in \mathbb{P}_1(\Delta_z^{k_0})$, using \eqref{eq:conv_hull_weights}, it follows for a linear polynomial $p$ that:
\begin{align}
\mathcal{C}_h p(z) = \sum_{T \in \mathcal T_z^{k_0} } \alpha_{z,T} p(s_T) +\sum_{b\in B_z^{k_0}}\beta_{z,b}\,p(b)= p \left( \sum_{T \in \mathcal T_z^{k_0} } \alpha_{z,T} s_T  + \sum_{b\in B_z^{k_0}}\beta_{z,b} b \right) = p (z),
\end{align}
where we used that $\int_T p \dd x = |T| p(s_T)$ and that $\mathcal{SZ}_h p (z) = p(z)$ to arrive at the first equality.
For a boundary node $z \in \mathcal{N}_h^\partial$, we conclude that $\mathcal{C}_hp (z) = (\mathcal{SZ}_h p) (z) = p(z)$.  Further, it is clear from \eqref{eq:conv_hull_weights} that $\mathcal{C}_h c = c$ for any constant $c$. Therefore, $\mathcal{C}_h p =  p$ for all $p \in \mathbb{P}_1(\Delta_T)$ with $T \in \mathcal{T}_h$. 
This implies that 
\begin{align}\label{eq:Clement_err_0}
\|w - \mathcal{C}_h w \|_{L^2(T)} \leq  \|w - p \|_{L^2(T)} + \|\mathcal{C}_h (p - w)\|_{L^2(T)}
\end{align}
for any $T \in \mathcal{T}_h$ and $p \in \mathbb P_1(\Delta_T)$.
For the second term, we bound 
$$\mathcal C_h(p-w)(z) =
\sum_{ T \in \mathcal T_z^{k_0} } \frac{\alpha_{z,T}}{|T|}\int_{T} p -w  \dd x + 
\sum_{b\in B_z^{k_0}}\beta_{z,b}\mathcal{SZ}_h (p -w)(b) = K_1 +K_2
$$ for each $z \in \mathcal{V}_T$. For an interior node $z \in \mathcal{N}_h$, the Cauchy--Schwarz inequality and $ 0 \leq \alpha_{z,T} \leq 1$ give
\begin{align} 
|K_1|
\leq \sum_{T \in \mathcal T_z^{k_0} } |\alpha_{z,T}| |T|^{-1} \|1\|_{L^2(T)} \|p-w\|_{L^2(T)} \lesssim |T|^{-1/2} \| p -w \|_{L^2(\Delta_z^{k_0})} .
\end{align}
For $K_2$, we apply local inverse and trace inequalities since each $b \in B_z^{k_0}$ belongs to some $T_b \in \mathcal{T}_z^{k_0}$. We also use the local stability property of $\mathcal{SZ}_h $ \cite[Theorem 3.1]{scott1990finite}. 
We obtain 
\begin{align}
|K_2| & \lesssim  \sum_{b \in B_z^{k_0}} h_T^{(1-n)/2} \|\mathcal{SZ}_h(w-p)\|_{L^2(\partial T_b)}
\\ & \lesssim  \sum_{b \in B_z^{k_0}}  h_T^{-n/2} \|\mathcal{SZ}_h(w-p)\|_{L^2(T)}  \nonumber 
\\ \nonumber &   \lesssim h_T^{-n/2}  ( \| p-w\|_{L^2(\Delta_z^{k_0})}  +  h_T \|\nabla (p-w)\|_{L^2(\Delta_z^{k_0})})
\end{align}
For both bounds on $K_1$ and $K_2$, we used \eqref{eq:local_uniformity} (note that this bound automatically holds for $k_0 =1$ by shape regularity) and that $\mathrm{card}(\mathcal T_z^{k_0})$ is bounded independently of $h$ since the mesh is shape regular and $k_0$ is independent of $h$.  
If $z \in \mathcal{N}_h^\partial$, then $z$ is a vertex of some $T \subset \mathcal{T}_z^1$. Using the definition of $\mathcal{C}_h$ on boundary nodes, we proceed similarly to $K_2$ and bound 
\begin{align}
|\mathcal{C}_h (p-w)(z)| = |\mathcal{SZ}_h(p-w)(z)|
& \lesssim h_T^{(1-n)/2} \|\mathcal{SZ}_h(w-p)\|_{L^2(\partial T)}
\\ & \lesssim h_T^{-n/2} \|\mathcal{SZ}_h(w-p)\|_{L^2(T)}  \nonumber 
\\ & \lesssim h_T^{-n/2}  ( \| p-w\|_{L^2(\Delta_z^{k_0})}  +  h_T \|\nabla (p-w)\|_{L^2(\Delta_z^{k_0})}) .  \nonumber
\end{align}
Using that $\mathcal C_h (p - w) \vert_{T} = \sum_{z \in \mathcal{V}_T}  \mathcal C_h(p-w)(z) \varphi_z$ and that $\|\varphi_z\|_{L^2(T)} \lesssim h_T^{n/2}$, we arrive at 
\begin{align}\label{eq:Clement_err_1}
\|\mathcal{C}_h (p-w)\|_{L^2(T)}  \lesssim \|p -w\|_{L^2(\Delta_T)} +  h_T \|\nabla (p-w)\|_{L^2(\Delta_T)}  . 
\end{align}
Combining \eqref{eq:Clement_err_0} with \eqref{eq:Clement_err_1} and using the Bramble--Hilbert Lemma on the macro-element $\Delta_T$ which is connected with $\operatorname{diam}\Delta_T\lesssim h_T$ by \eqref{eq:local_uniformity} yields the required bound on the first term of \eqref{eq:clement_error_estimate}. To obtain the required bound on the second term, we apply triangle and inverse estimates: 
\begin{align} \nonumber
\|\nabla (w-\mathcal{C}_h w) \|_{L^2(T)} \lesssim \|\nabla (w- \mathcal{SZ}_h w)\|_{L^2(T)} + h_T^{-1}( \| \mathcal{SZ}_h w - w \|_{L^2(T)} + \| w - \mathcal{C}_h w\|_{L^2(T)}) . 
\end{align}
The proof is completed by using the approximation properties of $\mathcal{SZ}_h $ \cite[Theorem 4.1]{scott1990finite} and the proven bound on $\| w - \mathcal{C}_h w\|_{L^2(T)}$. 
\end{proof}

\begin{remark}
An earlier version of this work employed a simpler interpolant $\mathcal{C}_h$ that can be constructed by fixing $k_0 = 1$ in definition~\eqref{eq:clement_case_1}.
This alternative interpolant was introduced in \cite{Fuhrer+2024+363+378}, where it is claimed that $z \in \operatorname{conv}({S}_z^1)$ for shape regular meshes. While it is easy to check that this claim holds when no mesh elements have obtuse angles, and our own numerical tests on structured, unstructured, and graded meshes indicate that it holds far more generally, the statement is false in general.
Expanding the element patch, as we have done in~\eqref{eq:clement_case_1}, avoids the assumption that $z \in \operatorname{conv}(S_z^1)$, which our earlier version of \Cref{lemma:Clement} relied on.
Instead, \Cref{lemma:Clement} requires \Cref{assum:conv_hull}, which we stress in a far weaker assumption that at least all locally quasi-uniform meshes satisfy; see \Cref{sec:locally_quasi_uniform}. 
\end{remark} 
} 
\subsection{Preliminaries for the $(\mathbb{P}_1,\mathbb{P}_{1})$-element pair} \label{subsec:obstacle_prelim_2}
For this case, the Fortin operator is simply the $L^2(\Omega)$-projection; cf.\ \cite[Remark 5.2]{keith2023proximal}. Thus, the main task is to construct the reconstruction operator associated with \Cref{assumption:maptointerior}, which again depends on a modified Cl\'ement quasi-interpolant. Using the definition of $K_h$ \cref{eq:general_set_Kh} and testing with $\varphi_z$, we observe that
\begin{align}
K_h = \left\{v_h \in V_h \mid \int_{\ramy{\Delta_z^1}} (v_h - \phi ) \varphi_z \dd x \geq 0 ~\fa z \in \mathcal{N}_h \right\}, 
\end{align}
since the support of $\varphi_z$ is $\ramy{\Delta_z^1}$, $\nabla \mathcal{R}^*(\psi_h^\ell) - \phi > 0$, and $\varphi_z \geq 0$. 
In what follows, we consider the quasi-interpolant proposed in \cite{yan2001posteriori,carstensen2002each}: 
\begin{equation} \label{eq:clemenet_case_2}
{\mathcal{C}}_h v = \sum_{z \in \mathcal{N}_h \ramy{\cup \mathcal{N}_h^\partial} } v_z  \varphi_z,  \quad  v_z =  \frac{1}{\int_{\ramy{\Delta_z^1}} \varphi_z} \int_{\ramy{\Delta_z^1}} v \varphi_z \dd x  \text{ if } z \in \ramy{\mathcal{N}_h}. 
\end{equation}
If $z \in \ramy{\mathcal{N}_h^\partial}$, we proceed as before and select $ v_z = (\mathcal{SZ}_h v) (z)$.
For simplicity, we make the following symmetry assumption on the mesh, which is sufficient for optimality of the quasi-interpolant given in \cref{eq:clemenet_case_2}.
Note that this assumption is the same condition necessary for optimality of the classical Cl\'ement quasi-interpolant \cite{Fuhrer+2024+363+378}.
\begin{assumption}[Local mesh symmetry] \label{assump:symmetry} For all $z \in \mathcal{N}_h$, assume that 
\begin{equation} \label{eq:symmetry_assumption}
\frac{1}{|\ramy{\Delta_z^1}|}\sum_{T \in \ramy{\mathcal{T}_z^1}} |T| s_T = z, \text{ where } s_T= (n+1)^{-1} \sum_{v \in \mathcal{V}_T}v. 
\end{equation}
\end{assumption}
We now prove that this condition implies optimality of~\cref{eq:clemenet_case_2}.
\begin{lemma}\label{lemma:clement_continuous}
Suppose that the mesh $\mathcal{T}_h$ satisfies \Cref{assump:symmetry}.
For any $T \in \mathcal{T}_h$ and for  $ 0 \leq s \leq 1$, 
\begin{alignat}{2}\label{eq:clement_error_estimate_0}
\|w- \mathcal{C}_h w \|_{L^2(T)} + h_T \|\nabla(w - \mathcal{C}_h w)\|_{L^2(T)}  &\lesssim h_T^{1+s} | w|_{H^{1+s}(\Delta_T)}  && ~\fa w \in H^{1+s}(\Delta_T), 
\end{alignat}
where $\Delta_T$ is the macro-element given by $\Delta_T = \cup_{z \in \mathcal{V}_T} \ramy{\Delta_z^1}$. 
\end{lemma}
\begin{proof}
The proof follows similar arguments to the proof of \Cref{lemma:Clement}. We only highlight the key points. 
We first show that under \cref{eq:symmetry_assumption}, $\mathcal{C}_h p (z) = p(z)$ for all $p \in \mathbb{P}_1(\ramy{\Delta_z^1})$. To see this first note that for linear $p$, we have
\begin{align*}
\int_{T} p(x) \varphi_z(x) \dd x = p \left(  \sum_{z' \in \mathcal{V}_T} \int_T  z' \varphi_{z'}(x) \varphi_z(x)  \dd x \right)  = \frac{|T|}{(n+1)(n+2)}  p (  z +(n+1)s_T ), 
\end{align*}
where we used that $\int_{T} \varphi_{z'} \varphi_z \dd x = (n+1)^{-1}(n+2)^{-1}|T| (1+\delta_{z',z})$ \cite[Exercise 4.1.1]{ciarlet2002finite}. Hence, 
\begin{align}
\frac{1}{\int_{\ramy{\Delta_z^1}} \varphi_z} \int_{\ramy{\Delta_z^1}} p \varphi_z \dd x = \frac{1}{(n+2)|\ramy{\Delta_z^1}|} p \left( \sum_{T \in \ramy{\mathcal{T}_z^1} } |T| (z + (n+1)s_T)  \right) = p(z).
\end{align}
In the last step, we used that $\int_{\ramy{\Delta_z^1}} \varphi_z = |\ramy{\Delta_z^1}|/(n+1)$ and \cref{eq:symmetry_assumption}. Along with the observation that $\overline{b}_z = b$ for any constant $b$, we conclude  that  $\mathcal{C}_h p = p $ for any $p \in \mathbb{P}_1(\ramy{\Delta_z^1})$.
In turn, it suffices to bound $\|\mathcal{C}_h(p-w)\|_{L^2(T)}$ as done in \cref{eq:Clement_err_0}. If $z \in \ramy{\mathcal{N}_h}$, we use the Cauchy--Schwarz inequality and the fact that $\|\varphi_z \|_{L^2(T)} \lesssim h_T^{n/2}$ to deduce
\begin{align}
|\mathcal{C}_h (p - w) (z)| \leq \sum_{T \in \ramy{\mathcal{T}_z^1} } (n+1)  |\ramy{\Delta_z^1}|^{-1} \|\varphi_z\|_{L^2(T)} \|p-w\|_{L^2(T)} \lesssim h_T^{-n/2} \|p-w\|_{L^2(\ramy{\Delta_z^1})} . 
\end{align}
The case $z \in \ramy{\mathcal{N}_h^\partial}$ and the remaining steps follow identically to the proof of \Cref{lemma:Clement}. The details are omitted for brevity. 
\end{proof} 
\begin{remark}[Removing \Cref{assump:symmetry}] As suggested by the proof of \Cref{lemma:clement_continuous}, one can not guarantee that $\mathcal C_h p (z) = p(z)$ for all $p \in \mathbb P_1(\ramy{\Delta_z^1})$ if \cref{eq:symmetry_assumption} does not hold.
Thus, a lack of local symmetry results in the proposed quasi-interpolant lacking optimality.
Of course, one may be inspired by the construction of $\mathcal{C}_h$ in \Cref{subsec:obstacle_prelim1} to generalize the proposed operator~\cref{eq:clemenet_case_2}.
In particular, one can define a convex combination 
\ramy{
$\{\alpha_{z'}\}_{z'\in\mathcal N_h\cap\overline{\Delta_z^1}}$ such that the reweighted operator
\[
  \mathcal{C}_h v(z)=\sum_{z'\in\mathcal N_h\cap\overline{\Delta_z^1}}
  \frac{\alpha_{z'}}{\int_{\Delta_{z'}^1}\varphi_{z'}\dd x}
  \int_{\Delta_{z'}^1} v\,\varphi_{z'}\dd x,
  \qquad z\in\mathcal N_h,
\]
}
preserves linear functions on the super-macro element \ramy{
$\Delta_z^2$}
This construction requires a more delicate analysis  \ramy{and intricate handling of the nodes that lie close to the boundary, similar to \eqref{eq:clement_case_1}.}
Although we save the technical details for future work, we perform this type of construction on a one-dimensional subdomain for the Signorini problem considered in \Cref{sec:error_rates_sig}; in particular, see~\cref{eq:clement_sig}.
\end{remark}
\subsection{Error estimates for the obstacle problem}\label{sec:obstacle_problem}
The preliminary results in \Cref{subsec:obstacle_prelim1,subsec:obstacle_prelim_2} leave us ready to state and prove optimal a priori error estimates for the PG method applied to the obstacle problem. {We emphasize that the $(\mathbb{P}_1\text{\normalfont-bubble}, \mathbb{P}_{0}\text{\normalfont-broken})$ pair does \textit{not} require quasi-uniform meshes. }
 \begin{corollary}[A priori error estimates for the obstacle problem]
\label{cor:Obstacle}
Assume that $u^*, \phi \in H^{1+s}(\Omega)$, $\lambda^* = E'(u^*) = -\Delta u^* \ramy{-} f \in H^{\ramy{r-1}}(\Omega)$ for fixed $s, r\in (0,1]$ and that $\|\psi_h^0\|_{L^\infty(\Omega)} \lesssim 1$.
\ramy{For the $(\mathbb{P}_1\text{\normalfont-bubble}, \mathbb{P}_{0}\text{\normalfont-broken})$-element pair, let the assumptions of \Cref{lemma:Clement} hold.} Moreover, for the $(\mathbb{P}_1,\mathbb{P}_{1})$-element pair, assume that the mesh is quasi-uniform and satisfies \Cref{assump:symmetry} \ramy{and that $\phi \vert_{\partial \Omega
}= -\delta < 0.$}

For both the $(\mathbb{P}_1\text{\normalfont-bubble}, \mathbb{P}_{0}\text{\normalfont-broken})$ and $(\mathbb{P}_1,\mathbb{P}_{1})$ element pairs, the following error estimate holds:
\begin{equation} \label{eq:error_bound_obstacle}
\| u^* -  u_h^\ell\|_{H^1(\Omega)}^2 + \|\lambda^* - \lambda_h^\ell\|^2_{H^{-1}(\Omega)}  \lesssim \frac{C_{\mathrm{stab}}}{\sum_{k=1}^\ell \alpha_k}  + C_{\mathrm{reg}}\, h^{2\cdot \min\{r,s\}},  
\end{equation} 
where the constants $C_{\mathrm{stab}}$ and $C_{\ramy{\mathrm{reg}}}$ are independent of $h,\; \ell$, and the proximity parameters $\alpha_k$. 

In addition, for the $(\mathbb{P}_1\text{\normalfont-bubble}, \mathbb{P}_{0}\text{\normalfont-broken})$ elements, we have the following estimate on the bound-preserving approximation $\tilde u_h^\ell$:
\begin{equation}
\|u^* - \tilde u_h^\ell \|^2_{L^2(\Omega)} \lesssim  \frac{C_{\mathrm{stab}}}{\sum_{k=1}^\ell \alpha_k}  + C_{\mathrm{reg}}\, h^{2\cdot \min\{r,s\}}. \label{eq:error_tilde_u}
\end{equation}
\end{corollary}
\begin{proof}
We verify the assumptions of \Cref{cor:err_rate_general}.
For the $(\mathbb{P}_1\text{\normalfont-bubble}, \mathbb{P}_{0}\text{\normalfont-broken})$ elements, \Cref{lemma:fortin_bubble} shows that \cref{eq:fortin_stability_assump_L2}-\cref{eq:fortin_stability_assump_H1}  and \cref{eq:fortin_rates} in \Cref{assumption:err_rates} hold.
For the $(\mathbb{P}_1,\mathbb{P}_{1})$ elements, the Fortin operator is simply the $L^2(\Omega)$ projection onto $V_h$.
It is standard to show that \cref{eq:fortin_stability_assump_L2}-\cref{eq:fortin_stability_assump_H1} and \cref{eq:fortin_rates} hold if the mesh is quasi-uniform, as assumed in this case; see, e.g., \cite[Proposition 22.21]{ern2021finite1} and \cite[Remark 5.2]{keith2023proximal}.

To verify \Cref{assumption:maptointerior}, define the reconstruction operator as follows:
{
\begin{equation} \label{eq:reconstruction_map_obstacle}
 \mathcal{E}_h  w  :=  \mathcal{C}_h (w - \phi) + \phi \end{equation}
}
 where $\mathcal{C}_h$ is the appropriate quasi-interpolant analyzed in \Cref{lemma:Clement,lemma:clement_continuous}.
Observe that $\mathcal{E}_h w \vert_{\partial \Omega} = 0$ for any $w \in H^1_0(\Omega)$, which follows from $\mathcal{C}_h w \vert_{\partial \Omega} = 0$ and $(\phi - \mathcal{C}_h \phi)\vert_{\partial \Omega} = 0$ since $\phi$ is a constant on $\partial \Omega$.  Hence, we obtain that $\mathcal{E}_h w \in H^1_0(\Omega)$ because $\mathcal{C}_h(w-\phi) \in \mathbb{P}_1(\mathcal{T}_h) \cap H^1(\Omega)$ and $\phi \in H^1(\Omega)$.  
\ramy{Further, for any $w \in K \cup K_h$, $\int_T (w- \phi)\dd x \geq 0$ and for any $ z \in \mathcal{N}_h^\partial$, $\mathcal{SZ}_h(w-\phi)(z) = - \mathcal{SZ}_h \phi(z) = -\phi \geq 0$.}
By construction of the interpolants, cf.\ \cref{eq:clement_case_1} and \cref{eq:clemenet_case_2},  we have that $\mathcal{C}_h(w-\phi)(z) \geq 0$ for any $z \in \mathcal{N}_h \ramy{\cup \mathcal{N}_h^\partial}$. 
{Hence, we conclude that $\mathcal{E}_h w \geq \phi$ and $\mathcal{E}_h w \in K$.} 
Observing that $\mathcal{E}_h w = \mathcal{C}_h w + \varepsilon $ where {$\varepsilon = \phi - \mathcal{C}_h \phi$}, the inequality \cref{eq:map_Ch_rates} is verified in \Cref{lemma:Clement} and \Cref{lemma:clement_continuous}. It remains to verify \cref{eq:map_Eh_rates_r}. {This follows directly from \cref{eq:interpolation_fortin}} 
{
\begin{align} \label{eq:obstacle_err_0}
\|\varepsilon \|_{H^t(\Omega)}  & =  
 \|\phi - \mathcal{C}_h \phi\|_{H^t(\Omega)}  
\lesssim h^{1+s-t} |\phi|_{H^{1+s-t}(\Omega)}. 
 \end{align}
 }
Having verified both \Cref{assumption:maptointerior,assumption:err_rates}, we invoke \Cref{cor:err_rate_general} to conclude that \cref{eq:error_bound_obstacle} holds.

To show \cref{eq:error_tilde_u}, we use \Cref{lemma:nonpoly_approx} and note that $\ramy{(I - P_{h}) \tilde u_h^\ell = (I-P_h)\phi}$ as $\ramy{\exp{\psi_h^\ell} \in \mathbb{P}_0(\mathcal{T}_h)}$. Therefore, since $B$ is the identity map  and $\Omega_d = \Omega$ in this case, \cref{eq:estimate_tilde_general} reads 
\begin{align}
\label{eq:estimate_tilde_obstacle}
\|u^* - \tilde u_h^\ell \|_{L^2(\Omega)}  & \leq \|u_h^\ell - u^*\|_{L^2(\Omega)} + \|(I- P_{h}) u_h^\ell\|_{L^2(\Omega)}  \ramy{ + \|(I- P_h) \phi\|_{L^2(\Omega)}} \\ 
\nonumber & \leq   \|u_h^\ell - u^*\|_{L^2(\Omega)} + \|(I- P_{h}) ( u_h^\ell - u^*)\|_{L^2(\Omega)}  + \|(I- P_{h}) u^*\|_{L^2(\Omega)}  \\ 
& \nonumber 
\ramy{ \quad + \|(I- P_h) \phi\|_{L^2(\Omega)}}   \\ 
& \lesssim \|u_h^\ell - u^*\|_{L^2(\Omega)} + h (|u^*|_{H^1(\Omega)} +\ramy{ |\phi|_{H^1(\Omega)} }),  \nonumber
\end{align}
where we used the stability of $P_{h}$ and that $\|v - P_{h} v\|_{L^2(T)} \lesssim h_K |v|_{H^1(T)}$ for any $T \in \mathcal{T}_h$ and \ramy{$v \in H^1(T)$}.
The final estimate~\cref{eq:error_tilde_u} is obtained by combining \cref{eq:error_bound_obstacle} and the final inequality in~\cref{eq:estimate_tilde_obstacle}. 
\end{proof}

\section{Application II: The Signorini problem} \label{sec:error_rates_sig}
We consider the following version of the Signorini problem on a two-dimensional domain for simplicity; cf.\ \Cref{example:signorini}.
Find $u^* \in V = H^1_\mathrm{D}(\Omega)^2$ where $H^1_{\mathrm{D}}(\Omega):= \{ v \in \ramy{H^1(\Omega)} \mid  v \vert_{\Gamma_{\rm D}} = 0 \}$  minimizing the strain energy function 
\begin{equation}
\min_{u \in K} E(u), \;\; E(u):= \frac{1}{2} \int_{\Omega} \mathsf C \, \epsilon(u) : \epsilon(u) \, \dd x -  \int_{\Omega} f \cdot u  \, \dd x,  \label{eq:signorini_pb}
\end{equation}
where $\partial \Omega = \overline{\Gamma_\mathrm{D} \cup \Gamma_\mathrm{T}}$ and
\begin{equation}
    K = \{ u \in V \mid u \cdot n \leq g \text{ on } \Gamma_\mathrm{T}\}. 
\end{equation}
Here, the boundary of $\Omega$ consists of two (relatively) open, disjoint subsets $\partial \Omega = \overline{\Gamma_{\rm T} \cup \Gamma_{\rm D}}$ with $|\Gamma_{\rm T}| > 0 $ and $|\Gamma_{\rm D}| > 0$.
For further simplicity, we assume that the contact boundary $\Gamma_{\rm T}$ is an open straight line segment and fix a gap function $g \in H^1(\Omega)$ with $ g \vert_{\Gamma_{\rm D}} =  \delta \ramy{ > } 0$. Recall that in this case, $\lambda^* = (\mathsf C \epsilon(u^*) n) \cdot n$.

We work with a quasi-uniform mesh $\mathcal{T}_h$, and consider the continuous Lagrange spaces on $\Omega$ and on $\Omega_d = \Gamma_{\mathrm{T}}$: 
\begin{subequations}
\begin{align}
    V_h & = (\mathbb{P}_1(\mathcal{T}_h) \cap H^1_{\mathrm{D}}(\Omega))^2, \\ 
    W_h & = \{ w_h \in C(\overline {\Gamma_{\mathrm{T}}}) \mid w_h \vert_E \in \mathbb{P}_1(E) ~\fa \text{edges }E  \subset \Gamma_{\mathrm{T}}, \;\; w_h = 0 \text{ on } \partial \Gamma_{\mathrm{T}} \}. 
\end{align}
\end{subequations}
In the above, $\partial \Gamma_{\mathrm{T}}$ denotes the boundary of $\Gamma_{\mathrm{T}}$; i.e.,\ the nodes shared between $\overline{\Gamma_{\mathrm{D}}}$ and $\overline{\Gamma_{\mathrm{T}}}$. To ensure compatibility, 
note that the degrees of freedom of $\psi_h^k$ in \cref{eq:discrete_lvpp} are set to $\nabla \mathcal{R}(0)$ on the nodes belonging to $\overline{\Gamma_{\mathrm{D}}} \cap \overline{\Gamma_{\mathrm{T}}}$; cf.\ \Cref{rem:Dirichlet}. In what follows, we will denote the normal and tangential components of a vector field $v \in H^1(\Omega)^2$ by $v_n$ and $v_\tau$, respectively.  \ramy{We also denote the set of interior vertices on $\Gamma_{\mathrm{T}}$ by $\mathcal{N}_h^{\Gamma_{\mathrm T}}$}.

For the considered problem, one can utilize the following choice of Legendre function $\mathcal{R}$: 
\begin{equation}
\ramy{\mathcal{R}(u) = (u+g) \ln (u+g) - (u+g)}
        \label{eq:Legendre_sing}
\end{equation}
which admits the convex conjugate 
\begin{equation} 
\ramy{ 
\mathcal{R}^*(\psi) = \exp(\psi) - g \psi.  }
\end{equation}
We provide \Cref{alg:sign}, the application of the PG method (\Cref{alg:main_alg_discrete}) to this problem for completeness; see also \cite[Example 2]{dokken2025latent}.  \ramy{Note that here we made the change of variable $\psi \to - \psi$}. 
\begin{algorithm}[htb]
\caption{The Proximal Galerkin Method for the Signorini Problem}
\begin{algorithmic}[1]\label{alg:sign}
    \State \textbf{input:} Initial latent solution guess $\psi_h^0  \in W_h$, a sequence of positive proximity parameters $\{\alpha_k\}$.
    \State Initialize \(k = 1\). 
    \State \textbf{repeat}
    \State \quad Find $u_h^{k} \in V_{h}$
and $\psi_h^{k} \in \ramy{-} \nabla \mathcal{R}(0) +   W_h$ such that 
   \begin{subequations} \label{eq:lvpp_sig}
\begin{alignat}{2}  
\alpha_k \, (\mathsf C\, \epsilon(u^{k}_h), \epsilon(v_h) )  + (v_h \cdot n,\psi^{k}_h - \psi_h^{k-1})_{\Gamma_{\mathrm{T}}} & = \alpha_k \, (f, v_h)  && ~\fa v_h \in V_h, \label{eq:lvpp_sing_0}\\ 
(u^{k}_h \cdot n,  w_h)_{\Gamma_{\mathrm{T}}} +  (\exp (-\psi^{k}_h), w_h)_{\Gamma_{\mathrm{T}}}& = (g , w_h)_{\Gamma_{\mathrm{T}}} && ~\fa  w_h \in W_h.  \label{eq:lvpp_sing_1}
\end{alignat}
\end{subequations} 
\State \quad Assign  \(k \gets k + 1\).
    \State \textbf{until} a convergence test is satisfied.
\end{algorithmic}
\end{algorithm}

\noindent\textbf{Main goal}:  We derive  error estimates for \Cref{alg:sign} in \Cref{cor:sig_error}. We also apply the framework presented in \Cref{sec:main_results}, utilizing \Cref{cor:err_rate_general}. To this end, we proceed by constructing Fortin and reconstruction operators satisfying \Cref{assumption:maptointerior,assumption:err_rates}. 
 \begin{lemma}[Fortin operator]  \label{lemma:fortin_sing}
Let $\mathcal{T}_h$ be quasi-uniform.  There exists a map $ \tilde \Pi_h : H^1_{\mathrm{D}}(\Omega) \rightarrow \ramy{\mathbb{P}_1(\mathcal{T}_h) \cap H^1_{\mathrm{D}}(\Omega)}$ such that 
 \[ 
 \int_{\Gamma_{\mathrm{T}}} \tilde \Pi_h v  w_h \dd s =  \int_{\Gamma_{\mathrm{T}}}  v  w_h  \dd s ~\fa w_h \in W_h. 
 \]
In addition, the Fortin operator $\Pi_h: H^1_{\mathrm{D}}(\Omega)^2  \rightarrow V_h$ given by $\Pi_h w = (\tilde \Pi_{h} w_n ) n + (\mathcal{SZ}_h w_\tau) \tau$ satisfies \cref{eq:fortin_map} and the stability and approximation bounds \cref{eq:fortin_stability_assump_L2,eq:fortin_stability_assump_H1,eq:fortin_rates} stated in \Cref{assumption:err_rates}.  
\end{lemma}
\begin{proof}
We define $\tilde \Pi_h$ and show its stability and approximation properties. The stated bounds for the operator $\Pi_h$ can then be deduced from the properties of $\tilde \Pi_h$ and of the Scott-Zhang interpolant. Let $\pi_h$ denote the $L^2(\Gamma_{\mathrm{T}})$ projection onto $W_h$: for $v\in L^2(\Gamma_{\mathrm{T}})$,  $\pi_h v \in W_h$ solves
\begin{equation}
 \int_{\Gamma_{\mathrm{T}}} (\pi_h v - v)  w_h = 0 ~\fa w_h \in W_h.
\end{equation}
For the nodes $z$ interior to $\Gamma_{\mathrm{T}}$ (i.e., $z \in \ramy{\mathcal{N}_h^{\Gamma_{\mathrm T}}}$), set $\tilde \Pi_h v(z) = \pi_h v(z)$. On the remaining nodes in $\mathcal{N}_h \ramy{\cup \mathcal{N}_h^\partial}$, set $\tilde \Pi_h v(z) = \mathcal{SZ}_h v (z)$. 
We split $\mathcal{T}_h$ into two subsets $\mathcal{T}_{\mathrm{T},h}$ and $\mathcal{T}_h \backslash \mathcal{T}_{\mathrm{T},h}$, where $\mathcal{T}_{\mathrm{T},h}$ consists of elements that share a node $z \in \ramy{\mathcal{N}_h^{\Gamma_{\mathrm T}}}$.
Denoting the associated macro-element by $\Delta_T$, we observe that
\[ 
\|\tilde \Pi_h v \|_{L^2(T)} = \|\mathcal{SZ}_h v\|_{L^2(T)} \lesssim \|v\|_{L^2(\Delta_T)} + h_T \|\nabla v\|_{L^2(\Delta_T)}
,\quad T \in \mathcal{T}_h \backslash \mathcal{T}_{\mathrm{T},h},
\]
where we used that $\tilde \Pi_h v(z) = \mathcal{SZ}_h v(z)$ for all $z \in \ramy{(\mathcal{N}_h \cup \mathcal{N}_h^\partial) \backslash} \ramy{\mathcal{N}_h^{\Gamma_{\mathrm T}}}$ and the local stability properties of $\mathcal{SZ}_h$ \cite{scott1990finite}.
For $z \in \ramy{\mathcal{N}_h^{\Gamma_{\mathrm T}}}$, we can apply a local inverse inequality to deduce that
\begin{align} \label{eq:stability_fortin_sing}
|\tilde \Pi_h v(z)|  = |\pi_h v(z)| \lesssim h_E^{-1/2} \|\pi_h v\|_{L^2(E) }, \quad E \subset \Gamma_{\mathrm{T}}. 
\end{align}
Therefore, since $\|\varphi_z\|_{L^2(T)} \lesssim h_T$, we obtain
\begin{align*}
\|\tilde \Pi_h v\|_{L^2(T)}^2 \lesssim \sum_{z \in \mathcal{V}_T} \| \tilde \Pi_h v(z) \varphi_z\|_{L^2(T)}^2 \lesssim  h_E \|\pi_h v\|^2_{L^2(E)} +   \|v\|_{L^2(\Delta_T)}^2 + h_T^2\|\nabla v\|_{L^2(\Delta_T)}^2, 
 \end{align*}
for all $T \in \mathcal{T}_{\mathrm{T},h}$ and some facet $E \subset \Gamma_{\mathrm{T}}$.
In particular, if $T$ shares a facet $E$ with $\Gamma_{\rm T}$, then we can select this facet. Otherwise, the facet of a neighboring element is selected. Summing over all the elements and noting that each facet $E \subset \Gamma_{\mathrm{T}}$ is counted at most $\max_{z \in \ramy{\mathcal{N}_h^{\Gamma_{\mathrm T}}}} \operatorname{card}(\ramy{\Delta_z^1})$ times, which is bounded uniformly w.r.t.\ $h$ thanks to shape-regularity of the mesh $\mathcal{T}_h$, we obtain that 
\begin{equation}\label{eq:stability_fortin_sig_0}
\|\tilde \Pi_h v\|_{L^2(\Omega)} \lesssim h^{1/2} \|\pi_h v\|_{L^2(\Gamma_{\mathrm{T}})} + \|v\|_{L^2(\Omega)} + h\|\nabla v\|_{L^2(\Omega)}.
\end{equation}
Using the stability of the projection $\pi_h$ and a global trace inequality coming from the quasi-uniformity of $\mathcal{T}_h$, we obtain that 
\begin{equation}
    h^{1/2} \|\pi_h v \|_{L^2(\Gamma_\mathrm{T})} \leq h^{1/2} \|v\|_{L^2(\Gamma_\mathrm{T})} \lesssim \|v\|_{L^2(\Omega)} + h \|\nabla v\|_{L^2(\Omega)}. \label{eq:global_trace}
\end{equation} 
Combining \cref{eq:global_trace,eq:stability_fortin_sig_0} shows that \cref{eq:fortin_stability_assump_L2} holds. 

We proceed to verify the error estimate \cref{eq:fortin_rates}. Note that $\tilde \Pi_h v_h = v_h $ for any $v_h \in \mathbb{P}_1(\mathcal{T}_h) \cap H^1_{\rm D}(\Omega)$. To see this, note that $\pi_h v_h = v_h \vert_{\Gamma_{\rm{T}}}$ since $v_h \vert_{\Gamma_{\mathrm{T}}} \in W_h$ and $\mathcal{SZ}_h v_h (z) = v_h(z)$ for all $z \in \mathcal{N}_h \ramy{\cup \mathcal{N}_h^\partial}$ \cite{scott1990finite}.  Hence, 
\begin{align}
\| v - \tilde \Pi_h v\|_{L^2(\Omega)} &\leq \|v - \mathcal{SZ}_h v\|_{L^2(\Omega)} + \|\tilde \Pi_h (v - \mathcal{SZ}_h v) \|_{L^2(\Omega)} \\
\nonumber 
& \lesssim
\|v - \mathcal{SZ}_h v\|_{L^2(\Omega)}  + h \|\nabla (v - \mathcal{SZ}_h v)\|_{L^2(\Omega)},  
\end{align}
where we used the just proven stability property~\cref{eq:fortin_stability_assump_L2} of $\tilde \Pi_h$. The estimate on the $L^2$ error stated in \cref{eq:fortin_rates} can now be concluded from the approximation properties of $\mathcal{SZ}_h$. To show \cref{eq:fortin_stability_assump_H1}, one uses an inverse estimate followed by approximation properties.  Namely, we  have 
\begin{multline}
\|\nabla \tilde \Pi_h v\|_{L^2(\Omega)} \leq \|\nabla (\tilde \Pi_h v - \mathcal{SZ}_h v) \|_{L^2(\Omega)} + \| \nabla \mathcal{SZ}_h v\|_{L^2(\Omega)} \\  \lesssim  h^{-1}  \| \tilde  \Pi_h v - \mathcal{SZ}_h v \|_{L^2(\Omega)} + \| \nabla \mathcal{SZ}_h v\|_{L^2(\Omega)} \lesssim \|\nabla v\|_{L^2(\Omega)}, 
\end{multline}
where we used the stability of $\mathcal{SZ}_h$ in the $H^1(\Omega)$-seminorm, see \cite[Corollary 4.8.15]{brenner2008mathematical}.
The final error estimate for $\tilde \Pi_h$ in the $H^1(\Omega)$-seminorm is proven similarly. 
\end{proof} 
In order to check \Cref{assumption:maptointerior}, we observe that the set $K_h$ \cref{eq:general_set_Kh} satisfies 
\begin{equation} \label{eq:Kh_signorini}
K_h = \{ u_h \in V_h \mid ( u_h \cdot n  - g, \varphi_z)_{\Gamma_{\mathrm{T}}}  \leq 0 ~\fa z \in \ramy{\mathcal{N}_h^{\Gamma_{\mathrm T}}}\}. 
\end{equation}
To define the reconstruction operator required by \Cref{assumption:maptointerior}, we begin by defining a modified Cl\'ement interpolant ${\mathcal C}_h: H^1(\Omega) \rightarrow H^1(\Omega) \cap \mathbb{P}_1(\mathcal{T}_h)$ using the strategy outlined in \Cref{subsec:obstacle_prelim_2}.  
To this end, we require a technical assumption on the mesh. Note that this assumption can be easily satisfied by a local refinement near the boundary $\partial \Gamma_{\mathrm{T}} = \overline{\Gamma_{\rm T}} \cap \overline{\Gamma_{\rm D}}$.
\begin{assumption} \label{assump:mesh_Sing}
Assume that $\Gamma_{\mathrm{T}}$ contains at least two mesh facets.
For each boundary node $z \in \partial \Gamma_{\mathrm{T}}$, let $E_z^1$ denote the unique boundary facet in $\Gamma_{\rm T}$ containing $z$.
Likewise, let $E_z^2 \subset \Gamma_T$ denote the unique boundary facet neighboring $E_z^1$ and not containing $z$.
Assume that $|E_z^1| \geq |E_z^2|$.
\end{assumption} 

For any boundary node $z \in \mathcal{N}_h^\partial$, we let $\tilde{\omega}_z$ denote the union of the boundary facets sharing $z$.
Then, denoting the three distinct vertices of the facets in $\tilde{\omega}_z$ by $\{z_0, z_1, z_2\}$ with $z_0 = z$, we define 
\begin{align}
{\mathcal C}_h v  = \sum_{z \in \mathcal{N}_h \ramy{\cup \mathcal{N}_h^\partial}}  v_z \varphi_z, \;\;\;
 v_z =  \sum_{i=0}^2 \frac{ \alpha_{z,i}}{\int_{\tilde{\omega}_{z_i}} \varphi_{z_i} } \int_{\tilde{\omega}_{z_i}}  v \varphi_{z_i} \dd s ~\text{ if } z \in \ramy{\mathcal{N}_h^{\Gamma_{\mathrm T}}},  \label{eq:clement_sig}
\end{align}
where the weights $\alpha_{z,i} \geq 0$, $i=0,1,2$, are defined below. 
For the remaining nodes in $\mathcal{N}_h \ramy{\cup \mathcal{N}_h
^\partial}$, we set  ${v}_z = \mathcal{SZ}_h v (z)$.

We now define ${\alpha}_{z,i}$.  Let $(\hat w_j, \hat x_j)$, $j=0,1$,  be the non-negative weights and points given by the Gauss--Radau quadrature rule \cite[Table 6.1]{ern2021finite1}, which is an exact quadrature rule for polynomials of degree $2$ on the reference facet $\hat E$. Let $F_E$ denote the affine linear transformation from the reference facet $\hat E$ to a facet  $E \subset \Gamma_{\mathrm T}$.  Observe that these quadrature points define weights and points $(w_j, x_j)$ with $x_j \in \tilde{\omega}_z$ for $j = 0,1,2,3$, such that for any continuous piecewise second-order polynomial $v$ on  $\tilde{\omega}_z = E_1 \cup E_2 $, we can write
\begin{equation}
\int_{ \tilde{\omega}_z} v \dd s = \sum_{j=0}^1 \frac{|E_1|}{|\hat E|} \hat w_j v (F_{E_1}(\hat x_j)) +  \sum_{j=0}^1 \frac{|E_2|}{|\hat E|} \hat w_j v (F_{E_2}(\hat x_j))  
= \sum_{j=0}^3 w_j  v(x_j). 
\end{equation} 
For any $z \in  \ramy{\mathcal{N}_h^{\Gamma_{\mathrm T}}}$, there exists a point $s_z \in \tilde{\omega}_z$ such that 
\begin{equation}
s_z = \frac{1}{\int_{\tilde{\omega}_z} \varphi_z} \int_{\tilde{\omega}_z} x \varphi_z(x) \dd s =  \frac{1}{\int_{\tilde{\omega}_z} \varphi_z}  \sum_{j=0}^3 w_j x_j \varphi_z(x_j).
\end{equation}
Therefore, $s_z$ is a convex combination of the points $\{x_j\}_{j=0,\ldots, 3}$ belonging to $\tilde{\omega}_z$.
Note that $z$ must lie either on the line connecting $s_{z_0}$ to $ s_{z_1}$ or on the line connecting $s_{z_0}$ to $s_{z_2}$. This means that there exist convex weights $ \alpha_{z_i} \geq 0$ (with one $ \alpha_{z_i} = 0$), such that 
\begin{equation}
 z = \sum_{i=0}^2  \alpha_{z_i} s_{z_i}, \quad \sum_{i=0}^2   \alpha_{z_i}  = 1,  \quad    \alpha_{z_i} \geq 0.  \label{eq:def_convex_sing}
\end{equation}
Further, \Cref{assump:mesh_Sing} guarantees that the weights corresponding to the boundary nodes in $\partial \Gamma_{\mathrm{T}}$ are zero in \cref{eq:def_convex_sing}; i.e.,  $\ramy{\alpha_{z_i}} =0$ if $\ramy{z_i} \in  \partial \Gamma_{\mathrm{T}}$. 
In turn, for each $z \in \ramy{\mathcal{N}_h^{\Gamma_{\mathrm T}}}$, 
\ramy{if $p \in \mathbb{P}_1(\Delta_z^2)$, it holds that}
\begin{equation}
\label{eq:SignoriniInterpolantReproducingLinearFunctions}
{\mathcal{C}}_h  p(z) = p(z). 
\end{equation}

\begin{lemma}[Approximation properties of $\mathcal{C}_h$] \label{lemma:clement_sing}
Let $\mathcal{T}_h$ be quasi-uniform and let \Cref{assump:mesh_Sing} hold. For $ 0 \leq s \leq 1$, it follows that
\begin{alignat}{2}\label{eq:clement_error_estimate_1}
\|w- \mathcal{C}_h w \|_{L^2(\Omega)} + h \|\nabla(w - \mathcal{C}_h w)\|_{L^2(\Omega)}  &\lesssim h^{1+s}| w|_{H^{1+s}(\Omega)}  && ~\fa w \in H^{1+s}(\Omega). 
\end{alignat}
\end{lemma} 
\begin{proof}  
We first consider the error $(\mathcal{C}_h v_h - v_h)$ for any $v_h \in V_h$. For a node 
$z \in \Gamma_{\mathrm{T}}$, we consider the \ramy{macro element $\Delta_z^2$}  such that $\mathcal{C}_h p (z) = p(z)$ for $p \in \mathbb{P}_1(\ramy{\Delta_z^2})$; cf.\ \cref{eq:SignoriniInterpolantReproducingLinearFunctions}. Then, 
\begin{align*}
(\mathcal{C}_h v_h - v_h) (z) = \mathcal{C}_h(v_h-p) (z) + (p -v_h)(z), \; \; p \in \mathbb{P}_1(\ramy{\Delta_z^2}).  
\end{align*}
With Cauchy--Schwarz inequality, the  observation that 
$\|\varphi_z\|_{L^2(\tilde{\omega}_{z_i})} \lesssim h^{1/2}$, and a trace inequality, we estimate 
 \begin{align}
     |\mathcal{C}_h (v_h-p)(z)|   & \leq \sum_{i=0}^2 |\tilde \alpha_{z_i}| (2 |\tilde \omega_{z_i}|^{-1}) 
     \|v_h-p\|_{L^2(\tilde\omega_{z,i})} \|\varphi_z\|_{L^2(\tilde \omega_{z_i})}  \\  \nonumber 
     & \lesssim  \sum_{i=0}^2 |\tilde \alpha_{z_i} |h^{-1/2} \|v_h -p\|_{L^2(\tilde \omega_{z_i})}
     \\ \nonumber
    & \lesssim h^{-1} \|v_h -p\|_{L^2(\ramy{\Delta_z^2})}.
 \end{align}
 In addition, we apply an inverse estimate to bound 
 \begin{align}
    |(p-v_h)(z)| \leq \|p-v_h\|_{L^{\infty}(\ramy{\Delta_z^2})} \lesssim h^{-1} \|p-v_h\|_{L^2(\ramy{\Delta_z^2})}
 \end{align} 
 since $(p-v_h) \in V_h$.
 Combining the above estimates, followed by applying the triangle inequality and the Bramble--Hilbert Lemma, we arrive at
\begin{align}
|(\mathcal{C}_h v_h - v_h )(z)| \label{eq:bounding_nodal_clement}
& \lesssim h^{-1} ( \|v_h -v \|_{L^2(\ramy{\Delta_z^2})} + \|v - p\|_{L^2(\ramy{\Delta_z^2})})
 \\ \nonumber 
& \lesssim  h^{-1} \|v_h -v \|_{L^2(\ramy{\Delta_z^2})} + h^{s} |v|_{H^{1+s}(\ramy{\Delta_z^2})}.  
\end{align}
Select $v_h = \mathcal{SZ}_h v$ and recall that by definition of $\mathcal{C}_h$, we have that $(\mathcal{C}_h - \mathcal{SZ}_h ) \ramy{v_h}(z) = 0 $ for all $ z\in \ramy{(\mathcal{N}_h \cup \mathcal{N}_h^\partial) \backslash} \ramy{\mathcal{N}_h^{\Gamma_{\mathrm T}}}$.
Therefore, for any $T \in \mathcal{T}_h$, we conclude that 
\begin{equation}
(\mathcal{C}_h \ramy{v_h} - \mathcal{SZ}_hv ) \vert_T =  \sum_{z \in \Gamma_{\mathrm{T}} \cap \mathcal{V}_T} (\mathcal{C}_h \ramy{v_h} - \mathcal{SZ}_hv )(z) \varphi_z. 
\end{equation}
Using that $\|\varphi_z\|_{L^2(T)} \lesssim h_T$, \cref{eq:bounding_nodal_clement}, and applying Cauchy--Schwarz inequality, we obtain that 
\begin{equation}
\|\mathcal{C}_h \ramy{v_h} - \mathcal{SZ}_hv \|_{L^2(T)} \lesssim \sum_{z \in \Gamma_{\mathrm{T}} \cap \mathcal{V}_T}  \left( h^{1+s} |v|_{H^{1+s}(\ramy{\Delta_z^2})} + \|\mathcal{SZ}_h v -v\|_{L^2(\ramy{\Delta_z^2})} \right) . 
\end{equation}
Summing over all the elements that contain a node in $\Gamma_{\mathrm{T}}$ and applying an inverse estimate, we find 
\begin{multline} \nonumber
\|\mathcal{C}_h \ramy{v_h} - \mathcal{SZ}_hv\|_{L^2(\Omega)} + h \| \nabla (\mathcal{C}_h \ramy{v_h} - \mathcal{SZ}_hv) \|_{L^2(\Omega)} 
\lesssim h^{1+s} |v|_{H^{1+s}(\Omega)} + \|\mathcal{SZ}_h v -v\|_{L^2(\Omega)}. 
\end{multline}
The result can be concluded by applying the triangle inequality, \ramy{the stability of $\mathcal{C}_h$, $\|\mathcal{C}_h v \|_{L^2(\Omega)} \lesssim \|v\|_{L^2(\Omega)} + h \|\nabla v \|_{L^2(\Omega)}$}, and the approximation properties of $\mathcal{SZ}_h$ \cite{scott1990finite}. 
\end{proof}
\begin{corollary}[A priori error estimate for the Signorini problem] \label{cor:sig_error}
Let $u^*$ solve \cref{eq:signorini_pb} and let $(u_h^\ell, \psi_h^\ell)$ come from~\Cref{alg:sign}. Assume that $u^*\in H^{1+s}(\Omega)^2$, $g \in H^{1+s}(\Omega)$ and $\ramy{E'(u^*) \in (H^{1-r}(\Omega)^2)'}$ for some $\ramy{r,s \in (0,1/2)}$, $\|\psi_h^0\|_{L^\infty(\Gamma_{\rm T})} \lesssim 1$, and \Cref{assump:mesh_Sing} holds. The following error estimate holds:
\begin{equation}
    \| u^* -  u_h^\ell\|_{H^1(\Omega)}^2 + \|\lambda^* - \lambda_h^\ell\|^2_{H^{-1/2}(\Gamma_\mathrm{T})}  \leq \frac{C_{\mathrm{stab}}}{\sum_{k=1}^\ell \alpha_k}  + C_{\mathrm{reg}}\, h^{2 \cdot \min\{r,s\}}. 
\end{equation}
\end{corollary}
\begin{proof}
Given the Fortin operator that we constructed in \Cref{lemma:fortin_sing}, we are in the setting of \Cref{sec:main_results}. We only need to verify the assumptions of \Cref{cor:err_rate_general}.
To this end, we define the normal and tangential components of $\mathcal{E}_h$ as follows:
{
\begin{equation}
    (\mathcal{E}_h w)_n  = \mathcal{C}_h (w_n - g) + g ,  \quad (\mathcal{E}_h w)_\tau = \mathcal{SZ}_h w_\tau ,
\end{equation}
}
To check \Cref{assumption:maptointerior}, we first note that $(\mathcal{E}_h w)_n =  0 $ on $\Gamma_{\rm D}$, $(\mathcal{E}_h w)_n \in H_{\mathrm{D}}^1(\Omega)$ and  $(\mathcal{E}_h w)_\tau \in H_{\mathrm{D}}^1(\Omega)$ from the properties of $\mathcal{SZ}_h$. This implies that $\mathcal{E}_h w \in H^1_{\mathrm D}(\Omega)^2$. 
{We now check that $g - \mathcal{E}_h w \cdot n \geq 0$ on  $\Gamma_\mathrm{T}$ to  ensure that $\mathcal{E}_h w \in K$}.
For $w \in K \cup K_h$, we have that $(\ramy{w} \cdot n - g, \varphi_z)_{\tilde{\omega}_z} \leq 0$ for any $z \in \Gamma_{\mathrm{T}}$ since the support of $\varphi_z \ramy{\vert_{\partial \Omega}}$ on $\ramy{\partial \Omega}$ is $\tilde\omega_z$.  The additional \Cref{assump:mesh_Sing} on the mesh guarantees that $\tilde \omega_{z}$ for $z\in \partial\Gamma_{\rm T}$ is not included in the definition of $\mathcal{C}_h$ \cref{eq:clement_sig}. 
From \cref{eq:clement_sig}, we now obtain that $\mathcal{C}_h (w_n - g)(z)  \leq 0$ for all $z \in \Gamma_{\mathrm{T}}$.  
In addition, for the nodes on $\partial \Gamma_{\rm T}$, we find $\mathcal{C}_h(w_n - g) = -g \vert_{\Gamma_{\rm D}} = -\delta {\leq} 0$.
Since {$g - \mathcal{E}_h w \cdot n = -\mathcal{C}_h(w_n - g)$}, we conclude that $(g - \mathcal{E}_h w \cdot n ) \geq  0 $ {on $\Gamma_{\mathrm{T}}$}.  This verifies \Cref{assumption:maptointerior}.
 
Noting that {$\mathcal{E}_h w = (\mathcal{C}_h w_n, \mathcal{SZ}_h w_\tau) + (g- \mathcal{C}_h g, 0)$}, 
\Cref{assumption:err_rates} is verified by applying \Cref{lemma:fortin_sing},  \Cref{lemma:clement_sing}, standard estimates on 
${\varepsilon =(g- \mathcal{C}_h g , 0)}$ similar to \cref{eq:obstacle_err_0}, and the approximation properties of the Scott--Zhang interpolant \cite{scott1990finite}.  Details are skipped for brevity. 
\end{proof}
{
\color{black}
\begin{remark}
In this remark, we elaborate on why the restriction $r \in (0,1/2)$ in \Cref{cor:sig_error} is required. 
 Recall that 
$E'(u^*) = B'\lambda^*$, see \eqref{eq:def_lambda*}. Since, 
there exists a bounded trace map
with a continuous right inverse from $H^{1-r}(\Omega)^2$ onto
$H^{1/2-r}(\Gamma_{\mathrm{T}})$ when $1-r > 1/2$ \cite[Theorem 1.2.24]{bernardi2024navier},  we have for
$0 < r < 1/2$,
\begin{equation*}
  E'(u^*) \in \big(H^{1-r}(\Omega)^2\big)'
  \quad \Longleftrightarrow \quad
  \lambda^* \in \big(H^{1/2-r}(\Gamma_{\mathrm{T}})\big)' .
\end{equation*}
However, for $r \geq 1/2$,  $E'(u^*) \in \big(H^{1-r}(\Omega)^2\big)'$ implies that 
$\lambda^* = 0$. Indeed, one has that $-\operatorname{div}(\mathsf{C} \epsilon(u^*)) - f = 0$ a.e. in $\Omega$ \cite[Lemma 3.3]{Chouly2023}.  
Therefore, $\langle E'(u^*), v \rangle = 0$ for all $v \in C_c^{\infty}(\Omega)^2$. Since the space $C_c^{\infty}(\Omega)^2$ is dense in $H^{1-r}(\Omega)^2$ 
when $1-r \leq 1/2$ \cite[Remark 1.2.30]{bernardi2024navier}, we conclude that  $E'(u^*) = 0$ in $(H^{1-r}(\Omega)^2)'$ and in $(H^{1}(\Omega)^2)'$ as $H^{1}(\Omega)^2 \subset H^{1-r}(\Omega)^2$ . Hence,  $\lambda^* = 0$ since $B'$ is injective
by \cref{eq:inf_sup_continuous}.

Of course, the assumption that
$\lambda^* \in H^{r- 1/2}(\Gamma_{\mathrm{T}})$ remains
meaningful for $ 1/2 \leq r \leq 1$. Using our framework to derive error estimates necessitates replacing the pairing $\langle E'(u^*), v \rangle$ in \eqref{eq:main_estimate} with $\langle \lambda^* ,v \cdot n\rangle$.  However, obtaining optimal error rates for $r \geq 1/2$ requires a different approach than the proof of \Cref{cor:err_rate_general} and is deferred to future work.
\end{remark}
}
\section{Conclusion}
\label{sec:Conclusion}
The PG method offers a versatile and efficient approach for solving variational problems with pointwise inequality constraints.  
We provided an abstract framework for its \textit{a priori} error analysis in the context of quadratic optimization problems with such constraints. We utilized this framework to derive optimal error estimates \ramy{given the assumed regularity} for the obstacle and Signorini problems, demonstrating the effectiveness of the PG method.  Numerically, the PG method achieves high-order error rates; however, the error analysis we presented is currently limited to the lowest-order conforming approximations. 
Future work could extend our results to high-order approximation spaces and a broader class of energy functionals. 

\section*{Acknowledgments}
We thank Ioannis P.\ A.\ Papadopoulos for his comments on the manuscript and Thomas M.\ Surowiec and Dohyun Kim for numerous interesting discussions. \ramy{We thank Zichen Lu for pointing out a mistake in an earlier version of the Cl\'ement interpolant~\eqref{eq:clement_case_1}. This issue was fixed in the published manuscript with help from Dohyun Kim.}
BK and RM were supported in part by the U.S.\ Department of Energy, Office of Science Early Career Research Program under Award Number DE-SC0024335 and by the Center for Information Geometric Mechanics and Optimization (CIGMO), a PSAAP-IV Focused Investigatory Center funded by the U.S.\ Department of Energy, National Nuclear Security Administration under Award Number DE-NA0004261.
Keith's work was also supported in part by the Alfred P.\ Sloan Foundation via a Sloan Research Fellowship in Mathematics. MZ acknowledges support from an ETH
Postdoctoral Fellowship for the project “Reliable, Efficient, and Scalable Methods for Scientific
Machine Learning”.

\appendix 
\section{A stability result }
\label{appendix:stability}
In this section, we show a stability result (\Cref{thm:stability} below) for the discrete iterates $(u_h^k,\psi_h^k)$. For simplicity, we assume that 
 $\nabla \mathcal{R}^*$ takes the following form:
\begin{subequations}
\label{eq:specialform_R}
\begin{equation}
\label{eq:specialform_R_decomposition}
\nabla \mathcal{R}^*(\psi) (x) = \phi_0(x) \nabla \mathcal{R}^*_0(\psi(x)) + \phi_1(x), 
\end{equation}
where $\phi_0 \in L^{\infty}(\Omega_d)$ with $\phi_0(x) \geq \underline{\phi_0} > 0$ a.e.\ in $\Omega$,  $\phi_1 \in Q$ and 
\begin{equation} \label{eq:coercivity_R0}
\mathcal{R}_0^*(\psi) - \nu_1 |\psi | \geq c_1 ~\fa \psi \in L^\infty(\Omega_d; \R^m),  
\end{equation}
\end{subequations}
for some $\nu_1 \geq 0$ and $c_1 \in \R$. 

The generalized Shannon entropy~\cref{eq:shannon} does not fit into the form \cref{eq:specialform_R} for {all obstacles $\phi \in L^\infty(\Omega_d; \R^m)$}.
However, the decomposition \cref{eq:specialform_R} is written with sufficient generality to accommodate the settings analyzed in this paper.
In particular, consider the obstacle problem from \Cref{sec:error_rates}, where $\phi \in H^1(\Omega)$ and $\phi \vert_{\partial \Omega} = - \delta$ for a positive constant $\delta$.
In this case, \cref{eq:specialform_R} holds with $\mathcal R_0^*(\psi) = \exp(\psi) - \delta \psi$, $\phi_0 = 1$, and $\phi_1 = \phi + \delta$, implying $\nu_1 = \delta$.
For the Fermi--Dirac entropy given in \Cref{example:fermi_dirac}, which is a suitable choice for bilateral obstacle problems ($B= \mathrm{id}$ and $Q=V = H^1_0(\Omega)$), we can write 
\begin{equation}
\label{eq:FermiDiracMapDecomposition}
\nabla \mathcal{R}^*(\psi) = \frac{1}{2}(\overline{u} - \underline{u}) \frac{\exp(\psi) - 1}{\exp(\psi)+1} + \frac12 (\overline{u} + \underline{u}). 
\end{equation}
Thus, defining $\mathcal R_0^*(\psi) = 2 \ln (\exp(\psi)+1) - \psi$ and assuming that  $\overline{u} - \underline{u} \geq \underline{\phi_0}$ and  $\overline{u} + \underline{u} \in H^1_0(\Omega)$, we deduce that \cref{eq:specialform_R} holds with $\nu_1  = 1$.
Notably, if $\underline{u} , \overline{u} \in \R$ with $\underline{u} < 0 < \overline{u}$, the condition $\underline{u} + \overline{u} \in H^1_0(\Omega)$ is not generally satisfied.
However, in this case, we can still verify \cref{eq:specialform_R} using an alternative decomposition given by $\mathcal R_0^*(\psi) = (\overline{u} - \underline{u}) \ln (\exp (\psi) + 1) + \underline{u}\psi$,  $\phi_0 =1$, and $ \phi_1 = 0$.
The Hellinger entropy, introduced in \Cref{example:Hellinger}, readily satisfies \cref{eq:specialform_R} with $\mathcal R_0^* (\psi) = \gamma \sqrt{1 + |\psi|^2}$, $\phi_0 =1$, and $ \phi_1 = 0$, implying $\nu_1 = \gamma$ and $c_1 = 0$.

\begin{theorem}[Stability] \label{thm:stability}
Assume that \cref{eq:specialform_R} holds. Further, assume that $\psi_h^0$ is selected so that $\|B '\psi_h^{0}\|_{V'}  \lesssim 1$. 
Then there exists a constant $C_{\mathrm{stab}}$, independent of $h$, $\ell$, and $\{\alpha_k\}_{k=2, \ldots } $, such that 
\begin{equation}
\nu \|u_h^\ell\|_V + \|\lambda_h^\ell\|_{Q^\prime}
+
\frac{\| \psi_h^\ell \|_{Q'} +  \nu_1\underline{\phi_0} \| \psi_h^\ell \|_{L^1(\ramy{\Omega_d})}}{\sum_{k=1}^\ell \alpha_k}
\leq C_{\mathrm{stab}}
~\fa \ell \geq 1 \text{ and } h > 0.
\label{eq:stability_uh}
\end{equation}
\end{theorem} 
\begin{proof}
We first observe that testing \cref{eq:lvpp_g_0} with $v_h = u_h^k$ and \cref{eq:lvpp_g_1} with $w_h = \psi_h^k$, subtracting the resulting equations, and using the expression~\cref{eq:specialform_R_decomposition} yields
\begin{equation}
\alpha_k a(u_h^k, u_h^k) + (\phi_0 \nabla \mathcal{R}_0^*(\psi_h^k), \psi_h^k)_{\Omega_d} = \alpha_k F(u_h^k) + b(u_h^k, \psi_h^{k-1}) - (\phi_1, \psi_h^k)_{\Omega_d}
\end{equation}
for every $k \geq 1$.
To handle the second term above, we split $\Omega_d$ into $\Omega_d^{-} = \{ x \in \Omega_d \mid \nabla \mathcal{R}_0^*(\psi_h^k)(x) \psi_h^k(x) \leq 0 \}$ and $\Omega_d^+ = \{ x \in \Omega_d \mid \nabla \mathcal{R}_0^*(\psi_h^k)(x) \psi_h^k(x) \geq 0 \}$.
Proceeding, we estimate  
\begin{align*}
 (\phi_0 \nabla \mathcal{R}_0^*(\psi_h^k), \psi_h^k)_{\Omega^+_d} & \geq  \underline{\phi_0} (\nabla \mathcal{R}_0^*(\psi_h^k), \psi_h^k)_{\Omega^+_d}, \\ 
  (\phi_0 \nabla \mathcal{R}_0^*(\psi_h^k), \psi_h^k)_{\Omega^-_d} &  \geq \|\phi_0 \|_{L^\infty(\Omega_d)} (\nabla \mathcal{R}_0^*(\psi_h^k), \psi_h^k)_{\Omega^-_d}.  
\end{align*}
Using the subgradient inequality and \eqref{eq:coercivity_R0}, we also bound 
\begin{equation}
(\nabla \mathcal R_0^*(\psi_h^k), \psi_h^k)_{\Omega_d^\pm} \geq (\mathcal{R}_0^*(\psi_h^k) - \mathcal{R}_0^* (0), 1)_{\Omega_d^{\pm}}
\geq \nu_1 \|\psi_h^k\|_{L^1(\Omega_d^{\pm})} + (c_1 -\mathcal{R}_0^*(0),1)_{\Omega_d^\pm}
.
\end{equation}
Noting that $c_1 \leq \mathcal{R}_0^*(0)$ by~\cref{eq:coercivity_R0}, the above estimates yield 
\begin{align}
 (\phi_0 \nabla \mathcal{R}_0^*(\psi_h^k), \psi_h^k)_{\Omega_d}
 \geq
 \nu_1 \underline{\phi_0} \|\psi_h^k\|_{L^1(\Omega_d)} + \|\phi_0\|_{L^\infty(\Omega_d)}(c_1 -\mathcal{R}_0^*(0),1)_{\Omega_d}.
\end{align}
Employing the coercivity of the bilinear form $a$ \cref{eq:coercivity_a}, we obtain that 
\begin{align} \label{eq:stab_bound_0}
\alpha_k \nu \|u_h^k\|_{V}^2 + \nu_1 \underline{\phi_0}\|\psi_h^k\|_{L^1(\Omega_d)}  \leq (\alpha_k \|F\|_{V'} + \|B' \psi_h^{k-1}\|_{V'}) \|u_h^k\|_V +  \|\psi_h^k\|_{Q'} \|\phi_1\|_Q + c_2,
\end{align} 
where $c_2 := \|\phi_0\|_{L^\infty(\Omega_d)}(\mathcal{R}_0^*(0) - c_1 ,1)_{\Omega_d} \geq 0$.

Proceeding, we obtain a bound on $\|\psi_h^k\|_{Q'}$. Using the Fortin operator \cref{eq:fortin_map} together with \cref{eq:lvpp_g_0}, we write
\begin{align}
\langle B' \psi_h^k, v \rangle = b(v, \psi_h^k) = b(\Pi_h v, \psi_h^k) = b(\Pi_h v, \psi_h^{k-1})- \alpha_k a(u_h^k, \Pi_h v) + \alpha_k F(\Pi_h v)
\end{align}
for any $v \in V$.
The surjectivity of $B: V \rightarrow Q$ and the continuity of the bilinear form $a$ and operator $\Pi_h$ give
\begin{align} \label{eq:stab_appendix_0}
\beta \|\psi_h^k\|_{Q'} \leq \|B' \psi_h^k\|_{V'} \leq \|\Pi_h\| (\|B' \psi_h^{k-1}\|_{V'} + \alpha_k M \|u_h^k\|_V + \alpha_k \|F\|_{V'}), 
\end{align}
where $\beta > 0$ is the same constant appearing in the LBB condition~\cref{eq:inf_sup_continuous}. 
Using the above bound in \cref{eq:stab_bound_0} and appropriate applications of Young's inequality shows that there exists a constant $M_{k}$ (depending only on $\alpha_k$, $M$, $\|F\|_{V'}$, $\|\Pi_h\|$, $\|\phi_1\|_{Q}$, $\beta$, and $c_2$) such that 
\begin{align}
\nu \|u_h^k\|_{V}^2 + \nu_1 \underline{\phi_0}\|\psi_h^k\|_{L^1(\Omega_d)} \leq M_k +  \|B' \psi_h^{k-1}\|^2_{V'}.   \label{eq:stab_0_both}
\end{align}
Selecting $k =1$ in the above estimate, recalling that $\|\Pi_h v\|_V \lesssim \|v\|_V$, and using the assumption that $\|B '\psi_h^{0}\|_{V'} \lesssim 1$ for all $v \in V$ implies that there is a constant $C_1$ (independent of $h$ and $\alpha_k$ for $k >1$) such that
\begin{align}
\nu \|u_h^1\|^2_{V} \leq M_1 + \|B' \psi_h^{0}\|_{V'}^2 := C_1^2. 
\end{align}
To obtain a bound on $\|u_h^\ell\|_V$ for $\ell >1$,  we use the energy dissipation property (i.e., \Cref{lemma:energy_dissipation}): 
\begin{equation}
\frac12 a(u_h^\ell, u_h^\ell) - F(u_h^\ell)= E(u_h^\ell)  \leq E(u_h^1). 
\end{equation}
 Using \cref{eq:coercivity_a}, we obtain that 
 \begin{align}
\frac\nu2 \|u_h^\ell\|_{V}^2 
& \leq M\|u_h^1\|_{V}^2 + \|F\|_{V'}(\|u_h^1\|_{V} + \|u_h^\ell\|_V) \\ \nonumber
& \leq M C_1^2 + \|F\|_{V'} C_1 + \frac1\nu \|F\|_{V'}^2 + \frac\nu4 \|u_h^\ell\|_{V}^2.  
 \end{align}
We can then conclude that 
\begin{equation} \label{eq:bound_stab_u}
\frac{\nu}{4}\|u_h^\ell\|_{V}^2 \leq  M C_1^2 + \|F\|_{V'} C_1 + \frac1\nu \|F\|_{V'}^2  : = C_2^2.
\end{equation}
Employing again that $B: V \rightarrow Q$ is surjective, we obtain that 
\begin{equation} \label{eq:bound_stab_lambda}
\beta \|\lambda_h^\ell\|_{Q'} \leq \|B' \lambda_h^\ell\|_{V'} \leq M \|u_h^\ell\|_V + \|F\|_{V'},
\end{equation}
where we used \eqref{eq:lvpp_g_0} and similar arguments to 
\eqref{eq:stab_appendix_0}. The bounds \cref{eq:bound_stab_u} and \cref{eq:bound_stab_lambda} yield the first two terms in~\cref{eq:stability_uh}. 

To show the bound on $\|\psi_h^\ell\|_{Q'}$, we define the weighted averages  $\overline{u}_h^\ell = \sum_{k=1}^{\ell} \alpha_{k} u_h^{k} / \sum_{k=1}^{\ell} \alpha_{k}$ and sum \cref{eq:lvpp_g_0} from $k = 1$ to $k =\ell$. This gives 
 \begin{align}
     a(\overline{u}_h^\ell, v_h) + \frac{1}{\sum_{k=1}^{\ell} \alpha_{k}} b(v_h, \psi_h^\ell) = F(v_h) +  \frac{1}{\sum_{k=1}^{\ell} \alpha_{k}} b(v_h, \psi_h^0) ~\fa v_h \in V_h. 
 \end{align}
 We now observe that for any $v \in V$, 
 \begin{align}
   b(v,\psi_h^\ell) =  b(\Pi_h v, \psi_h^\ell)  = \left(\sum_{k=1}^\ell \alpha_k \right) (F(\Pi_h v) - a(\overline{u}_h^\ell, \Pi_h v) ) + b(\Pi_hv , \psi_h^0). 
 \end{align}
With similar arguments as before and the observation that $\|\overline u_h^\ell\|_{V} \leq 2\nu^{-1/2}C_2$, this implies that
 \begin{equation}
    \|B' \psi_h^\ell\|_{V'}  \leq C_3 \left( \sum_{k=1}^\ell \alpha_k +1 \right) 
    ~\fa \ell \geq 1, \label{eq:stabV'_psi}
 \end{equation}
 where $C_3 =(\|F\|_{V'} + 2 \nu^{-1/2}M C_2 + \|B' \psi_h^0\|_{V'}) \|\Pi_h\|$. To derive the bound on $\|\psi_h^k\|_{L^1(\Omega_d)}$, we substitute \cref{eq:stabV'_psi} into \cref{eq:stab_bound_0}.
\end{proof}

{
\color{black}
\section{Locally quasi-uniform meshes satisfy \Cref{assum:conv_hull} }
\label{sec:locally_quasi_uniform}
In this section, we show that if the mesh is locally quasi-uniform then \Cref{assum:conv_hull} is satisfied. 
 
\begin{definition}[Locally quasi-uniform meshes] \cite[ Definition 4.5.5]{bernardi2024navier}\label{ass:lqu}. 
Let $h:\overline\Omega\to(0,\infty)$ be the mesh-size function $h|_T= |T|^{1/d}$ and set
$h_{\min}=\min_{\overline\Omega}h$, $h_{\max}=\max_{\overline\Omega}h$. The family
$(\Th)_h$ is locally quasi-uniform if there are constants $\lambda>1$ and
$A>0$ such that for all $x \in \overline \Omega$, 

\begin{equation}\label{eq:graded}
  \lambda^{-1}h(x)\le h(y)\le \lambda\,h(x)
\end{equation}
for all $y \in \overline \Omega$ satisfying 
\begin{equation}
  |x-y|\le Ah(x)\ln\Big(1+\tfrac{h_{\max}}{h_{\min}}\Big).
\end{equation}
\end{definition}
We define the mesh size of the first vertex patch 
\begin{equation}
    h_z = \max\{ h_T \mid  T \in \mathcal T_z^1 \} . 
\end{equation}
\begin{lemma}\label{lemma:lqu_implication}
Assume that the mesh is locally quasi-uniform. Then  every
element $T$ with $\overline T\cap\overline{B(z,2\lambda h_z)}\ne\emptyset$ satisfies
\begin{equation}\label{eq:ball-qu-tau}
  \lambda^{-1}h_z \le h_T \le \lambda h_z .
\end{equation}
\end{lemma}
\begin{proof}
We consider two cases. 

Case 1. The mesh is quasi-uniform, i.e., $h_{\max}/h_{\min}\le \lambda$ for some $\lambda>1$. Then \eqref{eq:ball-qu-tau} holds directly.  

Case 2. The ratio $h_{\max} / h_{\min}$ is unbounded. That is, the mesh is graded and satisfies $h_{\max}/h_{\min}\ge M$ where $M = \exp(3 \lambda/A)$, which holds for all sufficiently refined graded meshes. WLOG, we assume $M > 1$. 
Indeed, 
$\ln(1+h_{\max}/h_{\min})\ge\ln M=3\lambda/A$ and 
\begin{equation}\label{eq:radius-i}
  A\,h_z\ln\Big(1+\tfrac{h_{\max}}{h_{\min}}\Big) \ge\ 3\lambda h_z.
\end{equation}
Let $T$ satisfy $\overline T\cap\overline{B(z,2\lambda h_z)}\ne\emptyset$. Then, there is a $p \in \overline T$ such that $|p-z|\le 2\lambda h_z$.  Let $T_0 \in \mathcal{T}_z^1$ with $h_{T_0} = \max_{T \in \mathcal{T}_z^1} h_T = h_z$. 
Choose $x_0 \in T_0$ and $x \in T$ such that $|x_0-z|\le\lambda h_z/2$ and $| x - p| \leq \lambda h_z/2$. Then 
\[
  |x-x_0|\le |x-p|+|p-z|+|z-x_0| \le 2\lambda h_z+\lambda h_z \leq 3 \lambda h_z. 
\]
The proof is completed by applying \eqref{eq:radius-i} and the definition of the locally quasi-uniform meshes given in \eqref{eq:graded}.  
\end{proof}

\begin{lemma}\label{lem:far-tau}
Assume that the mesh is locally quasi-uniform. There exists a constant $k_0$ independent of $h$ and $z$ such that
\begin{equation}
z \in \operatorname{conv} (S_z^{k_0} \cup B_z^{k_0}), 
\end{equation}
 for any $z \in \mathcal{N}_h$. 
That is, locally quasi-uniform meshes satisfy \Cref{assum:conv_hull}. 
\end{lemma}

\begin{proof} 
By the separating-hyperplane characterization, it suffices to exhibit, 
for every unit vector $g$, 
a member $q \in S_z^{k_0} \cup B_z^{k_0}$ such that\footnote{We show that if \eqref{eq:shc} is true for any unit vector $g$ then  $z \in \conv (S_z^{k_0} \cup B_z^{k_0})$. Indeed, if  $z \notin \conv( S_z^{k_0} \cup B_z^{k_0})$, then the separating-hyperplane theorem \cite[Section 2.5.1]{boyd2004convex} yields the existence of $a \in \R^n , a \neq 0$ and $b \in \R$ such that 
\[ 
a \cdot q \leq b  \,\, \forall q \in \conv (S_z^{k_0} \cup B_z^{k_0}) 
\; \; \text{and} \;\; a \cdot z \geq b. 
\]
Therefore, $a \cdot (q - z) \leq 0 \,\, \forall q \in \conv (S_z^{k_0} \cup B_z^{k_0})$. Defining $g = - a / \|a\|$, we see that $g \cdot (q - z) \geq 0 \;\; \forall q \in \conv (S_z^{k_0} \cup B_z^{k_0})$ which contradicts the characterization \eqref{eq:shc}.
} 
\begin{equation}\label{eq:shc}
    g\cdot(q-z)<0. 
\end{equation}
To this end, set $y=z-2\lambda h_zg$ and consider two cases. 

Case 1. Suppose 
that the whole segment $ [z,y] $ is included in $\overline \Omega$; i.e., 
\[
  \gamma=[z,y] \subseteq\overline\Omega .
\]
Every element $T \in \mathcal S:=\{T\in\Th:\overline T\cap\gamma\ne\emptyset\}$ intersects 
$\overline{B(z,2\lambda h_z)}$. Then,   \Cref{lemma:lqu_implication} provides the bound
\begin{equation} \label{eq:proof_0}
\lambda^{-1}h_z \leq  h_T \leq \lambda h_z \quad \text{for all } T \in \mathcal S.
\end{equation}
 Therefore, for an element $T_0 \in \mathcal{S}$ with $y\in\overline T_0$, 
\begin{align}\label{eq:witness-tau}
  g\cdot(s_{T_0}-z) & =g\cdot(s_{T_0}-y) + g\cdot(y-z)
  \\ \nonumber & \le |s_{T_0}-y|-2\lambda h_z
  \\ \nonumber & \le \lambda h_z-2\lambda h_z=-\lambda h_z<0 .
\end{align}
Hence, it suffices to show that $T_0 \in \mathcal{T}_z^{k_0}$  for some mesh and node independent constant $k_0$. For $T\in\mathcal S$, take any 
$p\in\overline T\cap\gamma$.  Then $|p-z|\le 2\lambda h_z$ and by \eqref{eq:proof_0}, 
we have that $$|x - z| \leq |x -p| + |p-z| \leq 
3\lambda h_z, \quad x \in \overline T.$$
Thus,  
$\overline T\subseteq\overline{B(z,3\lambda h_z)}$. By shape regularity
and \eqref{eq:ball-qu-tau}, for any $T \in \mathcal{S}$,
\begin{equation}  |T|\ge\kappa_n\rho_T^{n}
      \ge\kappa_n\bigl(h_T/\sigma\bigr)^{n}
      \ge\kappa_n\bigl(\lambda^{-1}h_z/\sigma\bigr)^{n}, 
      \label{eq:shape_regularity_control_ball}
\end{equation}
where $\kappa_n=|B(0,1)|$.
Summing the volumes of all the elements in $\mathcal{S}$ and using  that each $T \in \mathcal{S}$ satisfies $T \subset \overline{B(z,3\lambda h_z)}$, we have that 
\begin{equation}
\sum_{T \in \mathcal{S}} |T| \leq |B(z,3\lambda h_z)|
\end{equation} 
Substituting \eqref{eq:shape_regularity_control_ball} in the above bound yields 
\begin{equation}\label{eq:card-tau}
  \mathrm{card}\,\mathcal S\ \le\
  \frac{|B(z,3\lambda h_z)|}{\kappa_n(\lambda^{-1}h_z/\sigma)^n}
  =(3\lambda^2\sigma)^n. 
\end{equation}
Therefore, the  $\mathrm{card}\,\mathcal S\le k_0$ for a mesh independent constant $k_0$ depending only on $\lambda, \sigma, d$.
The elements of $\mathcal{S}$ are connected and their number is at most $k_0$.
Hence, the element $T_0$ containing $y$ is joined to an element containing $z$ by a
chain of at most $k_0$ elements of $\mathcal S$, and therefore
$T
_0\in\mathcal{T}_z^{k_0}$. We conclude that $z\in\operatorname{int}\conv S_z^{k_0}$
since $g$ was arbitrary. 

Case 2. The segment $\gamma$ is not entirely included in $\overline \Omega$. In this case, the point $y$ may lie outside the domain. We then take the point $q = z - t^* g$ where $t^* < 2\lambda h_z$ and satisfies $$t^* = \max \{ t \geq 0 \mid [z , z-tg] \subset \overline \Omega \}.$$ 
Since $q \in \partial \Omega$, there exists $T_q \in \mathcal{T}_h$ with $q \in F_q$ where $F_q \subset \partial T_q \cap \partial \Omega$. Note that $|z - q| = t^* < 2 \lambda h_z$ and thus $T_q \in \mathcal{T}_z^{k_0}$ (by using the same argument as before). 
Observe that 
\begin{align}
-t^* = g \cdot (q- z) = \sum_{i=1}^{n} \theta_i\,g\cdot(b_i-z) 
\end{align}
where $\{b_1,\ldots, b_{n}\}$ are the vertices of $F_q$ and $\{\theta_i\}_{i=1}^{n}$ are the non--negative weights satisfying $\sum_{i=1}^{n} \theta_i b_i = q$ and $\sum_{i=1}^n \theta_i = 1$. Therefore, there must exist a vertex $b\in B_z^{k_0}$ such that $g \cdot (b -z) < 0$. 
Since $g$ was arbitrary, we conclude that $z \in \operatorname{conv} (S_z^{k_0} \cup B_z^{k_0}). $

\end{proof}
}

\bibliographystyle{plain}
\bibliography{references}

\begin{thebibliography}{10}

\bibitem{allgower1986mesh}
Eugene~L Allgower, Klaus B{\"o}hmer, FA~Potra, and WC~Rheinboldt.
\newblock A mesh-independence principle for operator equations and their
  discretizations.
\newblock {\em SIAM Journal on Numerical Analysis}, 23(1):160--169, 1986.

\bibitem{amari2016information}
Shun-ichi Amari.
\newblock {\em Information Geometry and its Applications}, volume 194.
\newblock Springer, 2016.

\bibitem{ambrosetti1995primer}
Antonio Ambrosetti and Giovanni Prodi.
\newblock {\em A Primer of Nonlinear Analysis}.
\newblock Cambridge University Press, 1995.

\bibitem{antil2023alesqp}
Harbir Antil, Drew~P Kouri, and Denis Ridzal.
\newblock {ALESQP}: {A}n augmented {L}agrangian equality-constrained {SQP}
  method for optimization with general constraints.
\newblock {\em SIAM Journal on Optimization}, 33(1):237--266, 2023.

\bibitem{attouch2014variational}
Hedy Attouch, Giuseppe Buttazzo, and G{\'e}rard Michaille.
\newblock {\em Variational Analysis in {S}obolev and $BV$ Spaces:
  {A}pplications to {PDE}s and {O}ptimization}.
\newblock SIAM, 2014.

\bibitem{babuvska1973finite}
Ivo Babu{\v{s}}ka.
\newblock The finite element method with penalty.
\newblock {\em Mathematics of {C}omputation}, 27(122):221--228, 1973.

\bibitem{bartels2015numerical}
S{\"o}ren Bartels.
\newblock {\em Numerical methods for nonlinear partial differential equations},
  volume~47.
\newblock Springer, 2015.

\bibitem{bauschke1997legendre}
Heinz~H Bauschke and Jonathan~M Borwein.
\newblock Legendre functions and the method of random {B}regman projections.
\newblock {\em Journal of Convex Analysis}, 4(1):27--67, 1997.

\bibitem{bernardi2024navier}
Christine Bernardi, Vivette Girault, Fr{\'e}d{\'e}ric Hecht, Pierre-Arnaud
  Raviart, and Beatrice Rivi{\`e}re.
\newblock {\em Mathematics and Finite Element Discretizations of Incompressible
  {Navier--Stokes} Flows}.
\newblock Classics in Applied Mathematics. SIAM, 2025.

\bibitem{boffi2013mixed}
Daniele Boffi, Franco Brezzi, and Michel Fortin.
\newblock {\em Mixed Finite Element Methods and Applications}, volume~44.
\newblock Springer, 2013.

\bibitem{boyd2004convex}
Stephen Boyd and Lieven Vandenberghe.
\newblock {\em Convex Optimization}.
\newblock Cambridge University Press, 2004.

\bibitem{brenner2008mathematical}
Susanne~C Brenner and L~Ridgway Scott.
\newblock {\em The Mathematical Theory of Finite Element Methods}, volume~3.
\newblock Springer, 2008.

\bibitem{Brezis1971}
H.~Brezis and M.~Sibony.
\newblock Equivalence de deux inéquations variationnelles et applications.
\newblock {\em Archive for Rational Mechanics and Analysis}, 41(4):254–265,
  January 1971.

\bibitem{brezis2010functional}
Haim Brezis.
\newblock {\em Functional Analysis, Sobolev Spaces and Partial Differential
  Equations}, volume~2.
\newblock Springer, 2010.

\bibitem{bueler2024full}
Ed~Bueler and Patrick~E Farrell.
\newblock A full approximation scheme multilevel method for nonlinear
  variational inequalities.
\newblock {\em SIAM Journal on Scientific Computing}, 46(4):A2421--A2444, 2024.

\bibitem{burachik1998generalized}
Regina~S Burachik and Alfredo~N Iusem.
\newblock A generalized proximal point algorithm for the variational inequality
  problem in a {H}ilbert space.
\newblock {\em SIAM journal on Optimization}, 8(1):197--216, 1998.

\bibitem{carstensen2002each}
Carsten Carstensen and S{\"o}ren Bartels.
\newblock Each averaging technique yields reliable a posteriori error control
  in {FEM} on unstructured grids. {P}art {I}: Low order conforming,
  nonconforming, and mixed {FEM}.
\newblock {\em Mathematics of Computation}, 71(239):945--969, 2002.

\bibitem{chen1993convergence}
Gong Chen and Marc Teboulle.
\newblock Convergence analysis of a proximal-like minimization algorithm using
  {B}regman functions.
\newblock {\em SIAM Journal on Optimization}, 3(3):538--543, 1993.

\bibitem{Chouly2023}
Franz Chouly, Patrick Hild, and Yves Renard.
\newblock {\em Finite Element Approximation of Contact and Friction in
  Elasticity}, volume~48 of {\em Advances in Mechanics and Mathematics}.
\newblock Birkhäuser Cham, Cham, Switzerland, 2023.

\bibitem{ciarlet2002finite}
Philippe~G Ciarlet.
\newblock {\em The Finite Element Method for Elliptic Problems}.
\newblock SIAM, 2002.

\bibitem{ciarlet2013linear}
Philippe~G Ciarlet.
\newblock {\em Linear and Nonlinear Functional Analysis with Applications},
  volume 130.
\newblock SIAM, 2013.

\bibitem{clement1975approximation}
Ph~Cl{\'e}ment.
\newblock Approximation by finite element functions using local regularization.
\newblock {\em Revue fran{\c{c}}aise D'automatique, Informatique, Recherche
  Op{\'e}rationnelle. Analyse Num{\'e}rique}, 9(R2):77--84, 1975.

\bibitem{demkowicz2023mathematical}
Leszek~F Demkowicz.
\newblock {\em Mathematical Theory of Finite Elements}.
\newblock SIAM, 2023.

\bibitem{dokken2025latent}
J{\o}rgen~S Dokken, Patrick~E Farrell, Brendan Keith, Ioannis Papadopoulos, and
  Thomas~M Surowiec.
\newblock The latent variable proximal point algorithm for variational problems
  with inequality constraints.
\newblock {\em Computer Methods in Applied Mechanics and Engineering}, 2025.

\bibitem{zenodo:proximalgalerkin}
Jørgen~S. Dokken, Patrick~E. Farrell, Brendan Keith, Ioannis~P.A.
  Papadopoulos, and Thomas~M. Surowiec.
\newblock The latent variable proximal point algorithm for variational problems
  with inequality constraints, June 2025.
\newblock Version 0.4.1. \url{https://doi.org/10.5281/zenodo.15723819}.

\bibitem{eckstein1998approximate}
Jonathan Eckstein.
\newblock Approximate iterations in {B}regman-function-based proximal
  algorithms.
\newblock {\em Mathematical Programming}, 83(1):113--123, 1998.

\bibitem{ern2004theory}
Alexandre Ern and Jean-Luc Guermond.
\newblock {\em Theory and Practice of Finite Elements}, volume 159.
\newblock Springer, 2004.

\bibitem{ern2017finite}
Alexandre Ern and Jean-Luc Guermond.
\newblock Finite element quasi-interpolation and best approximation.
\newblock {\em ESAIM: Mathematical Modelling and Numerical Analysis},
  51(4):1367--1385, 2017.

\bibitem{ern2021finite1}
Alexandre Ern and Jean-Luc Guermond.
\newblock {\em Finite Elements I: Approximation and Interpolation}, volume~72
  of {\em Texts in Applied Mathematics}.
\newblock Springer, Cham, 2021.

\bibitem{ern2021finite2}
Alexandre Ern and Jean-Luc Guermond.
\newblock {\em Finite elements II: Galerkin Approximation, Elliptic and Mixed
  PDEs}, volume~73 of {\em Texts in Applied Mathematics}.
\newblock Springer, Cham, 2021.

\bibitem{evans2018measure}
Lawrence~C. Evans and Ronald~F. Gariepy.
\newblock {\em Measure Theory and Fine Properties of Functions}.
\newblock CRC Press, 2nd edition, 2024.

\bibitem{fu2025proximal}
Guosheng Fu, Brendan Keith, Dohyun Kim, Rami Masri, and Will Pazner.
\newblock The proximal galerkin method for non-symmetric variational
  inequalities.
\newblock arXiv preprint arXiv:2602.14967, 2026.

\bibitem{fu2024locally}
Guosheng Fu, Brendan Keith, and Rami Masri.
\newblock A locally-conservative proximal galerkin method for pointwise bound
  constraints.
\newblock {\em Mathematics of Computation}, 2026.

\bibitem{Fuhrer+2024+363+378}
Thomas F{\"u}hrer.
\newblock On a mixed {FEM} and a {FOSLS} with {$H^{-1}$} loads.
\newblock {\em Computational Methods in Applied Mathematics}, 24(2):363--378,
  2024.

\bibitem{Glowinski1984}
Roland Glowinski.
\newblock {\em Numerical Methods for Nonlinear Variational Problems}.
\newblock Springer Berlin Heidelberg, 1984.

\bibitem{glowinski1989augmented}
Roland Glowinski and Patrick Le~Tallec.
\newblock {\em Augmented Lagrangian and Operator-Splitting Methods in Nonlinear
  Mechanics}, volume~9 of {\em Studies in Applied and Numerical Mathematics}.
\newblock SIAM, Philadelphia, PA, 1989.

\bibitem{graser2009multigrid}
Carsten Gr{\"a}ser and Ralf Kornhuber.
\newblock Multigrid methods for obstacle problems.
\newblock {\em Journal of Computational Mathematics}, 27(1):1--44, 2009.

\bibitem{gustafsson2017finite}
Tom Gustafsson, Rolf Stenberg, and Juha Videman.
\newblock On finite element formulations for the obstacle problem--mixed and
  stabilised methods.
\newblock {\em Computational Methods in Applied Mathematics}, 17(3):413--429,
  2017.

\bibitem{hintermuller2002primal}
Michael Hinterm{\"u}ller, Kazufumi Ito, and Karl Kunisch.
\newblock The primal-dual active set strategy as a semismooth {N}ewton method.
\newblock {\em SIAM Journal on Optimization}, 13(3):865--888, 2002.

\bibitem{hintermuller2006feasible}
Michael Hinterm{\"u}ller and Karl Kunisch.
\newblock Feasible and noninterior path-following in constrained minimization
  with low multiplier regularity.
\newblock {\em SIAM Journal on Control and Optimization}, 45(4):1198--1221,
  2006.

\bibitem{hintermuller2004mesh}
Michael Hinterm{\"u}ller and Michael Ulbrich.
\newblock A mesh-independence result for semismooth {N}ewton methods.
\newblock {\em Mathematical Programming}, 101:151--184, 2004.

\bibitem{ito1990augmented}
Kazufumi Ito and Karl Kunisch.
\newblock An augmented {L}agrangian technique for variational inequalities.
\newblock {\em Applied Mathematics \& Optimization}, 21(1):223--241, 1990.

\bibitem{keith2026analysis}
Brendan Keith, Haojun Qin, and Noe Reyes~Rivas.
\newblock Analysis of {B}regman proximal point for the obstacle problem.
\newblock Forthcoming, 2026.

\bibitem{keith2023proximal}
Brendan Keith and Thomas~M Surowiec.
\newblock Proximal {G}alerkin: A structure-preserving finite element method for
  pointwise bound constraints.
\newblock {\em Foundations of Computational Mathematics}, pages 1--97, 2024.

\bibitem{kunisch2004total}
Karl Kunisch and Michael Hinterm{\"u}ller.
\newblock Total bounded variation regularization as a bilaterally constrained
  optimization problem.
\newblock {\em SIAM Journal on Applied Mathematics}, 64(4):1311--1333, 2004.

\bibitem{nocedal2006numerical}
Jorge Nocedal and Stephen~J Wright.
\newblock {\em Numerical Optimization}.
\newblock Springer, second edition, 2006.

\bibitem{nochetto2002positivity}
Ricardo Nochetto and Lars Wahlbin.
\newblock Positivity preserving finite element approximation.
\newblock {\em Mathematics of Computation}, 71(240):1405--1419, 2002.

\bibitem{papadopoulos2024hierarchical}
Ioannis Papadopoulos.
\newblock Hierarchical proximal {G}alerkin: a fast $hp$-{FEM} solver for
  variational problems with pointwise inequality constraints.
\newblock {\em arXiv preprint arXiv:2412.13733}, 2024.

\bibitem{papadopoulos2026mesh}
Ioannis Papadopoulos and Michael Hinterm{\"u}ller.
\newblock Mesh-dependent iteration count growth in primal-dual active set
  strategies.
\newblock {\em arXiv preprint arXiv:2607.26622}, 2026.

\bibitem{rockafellar1967conjugates}
Ralph~Tyrrell Rockafellar.
\newblock Conjugates and {L}egendre transforms of convex functions.
\newblock {\em Canadian Journal of Mathematics}, 19:200--205, 1967.

\bibitem{RTRockafellar_1970}
Ralph~Tyrrell Rockafellar.
\newblock {\em Convex Analysis}.
\newblock Princeton University Press, Princeton, 1970.

\bibitem{rudin1992nonlinear}
Leonid~I Rudin, Stanley Osher, and Emad Fatemi.
\newblock Nonlinear total variation based noise removal algorithms.
\newblock {\em Physica D: Nonlinear Phenomena}, 60(1-4):259--268, 1992.

\bibitem{scholz1984numerical}
Reinhard Scholz.
\newblock Numerical solution of the obstacle problem by the penalty method.
\newblock {\em Computing (Wien. Print)}, 32(4):297--306, 1984.

\bibitem{schwedes2017mesh}
Tobias Schwedes, David~A Ham, Simon~W Funke, and Matthew~D Piggott.
\newblock {\em Mesh dependence in {PDE}-Constrained Optimisation}.
\newblock Springer, 2017.

\bibitem{scott1990finite}
L~Ridgway Scott and Shangyou Zhang.
\newblock Finite element interpolation of nonsmooth functions satisfying
  boundary conditions.
\newblock {\em Mathematics of Computation}, 54(190):483--493, 1990.

\bibitem{tartar2007introduction}
Luc Tartar.
\newblock {\em An Introduction to {S}obolev spaces and Interpolation spaces},
  volume~3.
\newblock Springer Science \& Business Media, 2007.

\bibitem{TWTing_1969}
Tsuan~Wu Ting.
\newblock Elastic-plastic torsion of convex cylindrical bars.
\newblock {\em Journal of Mathematics and Mechanics}, 19(6):531--551, 1969.

\bibitem{verfurth1994posteriori}
R{\"u}diger Verf{\"u}rth.
\newblock A posteriori error estimation and adaptive mesh-refinement
  techniques.
\newblock {\em Journal of Computational and Applied Mathematics},
  50(1-3):67--83, 1994.

\bibitem{weiser2005asymptotic}
Martin Weiser, Anton Schiela, and Peter Deuflhard.
\newblock Asymptotic mesh independence of {N}ewton's method revisited.
\newblock {\em SIAM Journal on Numerical Analysis}, 42(5):1830--1845, 2005.

\bibitem{yan2001posteriori}
Ningning Yan.
\newblock A posteriori error estimators of gradient recovery type for elliptic
  obstacle problems.
\newblock {\em Advances in Computational Mathematics}, 15:333--361, 2001.

\end{thebibliography}

\end{document}